\pdfoutput=1
\RequirePackage{ifpdf}
\ifpdf 
\documentclass[pdftex]{sigma}
\else
\documentclass{sigma}
\fi

\usepackage{mathrsfs,pb-diagram,pb-xy}
\usepackage[all]{xy}
\usepackage[inline]{enumitem}
\usepackage{mathrsfs}
\usepackage{multicol}
\usepackage{bm}
\usepackage{tikz}

\numberwithin{equation}{section}

\newtheorem{Theorem}{Theorem}[section]
\newtheorem*{Theorem*}{Theorem}
\newtheorem{Corollary}[Theorem]{Corollary}
\newtheorem{Lemma}[Theorem]{Lemma}
\newtheorem{Proposition}[Theorem]{Proposition}
 { \theoremstyle{definition}
\newtheorem{Definition}[Theorem]{Definition}

\newtheorem{Remark}[Theorem]{Remark} }
\newtheorem{hypothesis}[Theorem]{Hypothesis}

\newtheorem{step}{Step}
\newtheorem{stepi}{Step}
\newtheorem{stepii}{Step}
\newtheorem{stepiii}{Step}
\newtheorem{stepiv}{Step}
\newtheorem{stepv}{Step}

\newcommand{\adc}{\operatorname{ad}_c}
\newcommand{\co}{\operatorname{co}}
\newcommand{\Der}{\operatorname{Der}}
\newcommand{\GK}{\operatorname{GK-dim}}
\newcommand{\gr}{\operatorname{gr}}
\newcommand{\Hom}{\operatorname{Hom}}
\newcommand{\id}{\operatorname{id}}
\newcommand{\supp}{\operatorname{supp}}

\newcommand{\toba}{\mathscr B}
\newcommand{\wtoba}{\widetilde{\mathscr B}}
\newcommand{\ghost}{\mathscr{G}}

\newcommand{\K}{{\mathcal{K}}}
\newcommand{\Ku}{\underline{\mathcal K}}
\newcommand{\Pc}{{\mathcal{P}}}
\newcommand{\Ss}{\mathcal{S}}
\newcommand{\Vc}{{\mathcal{V}}}
\newcommand{\Zc}{{\mathcal{Z}}}

\newcommand{\ab}{\mathbf a}

\newcommand{\bb}{\mathbf b}

\newcommand{\pale}{\mathfrak{P}}
\newcommand{\sele}{\mathfrak{S}}

\newcommand{\G}{{\mathbb G}}
\newcommand{\I}{{\mathbb I}}
\newcommand{\N}{{\mathbb N}}
\newcommand{\Z}{{\mathbb Z}}

\newcommand{\ttt}{\mathtt{t}}
\newcommand{\ut}{\mathtt{u}}
\newcommand{\vt}{\mathtt{v}}
\newcommand{\wt}{\mathtt{w}}
\newcommand{\xt}{\mathtt{x}}
\newcommand{\yt}{\mathtt{y}}
\newcommand{\zt}{\mathtt{z}}
\newcommand{\zc}[2]{\mathtt{z}_{#1,#2}}

\newcommand{\eny}{\mathfrak{E}}
\newcommand{\enyIIb}[1]{\eny_{\star,#1}}
\newcommand{\yd}[1]{{}_{#1}^{#1}\mathcal{YD}}

\newcommand{\bq}{\mathfrak{q}}
\newcommand{\bqs}{\bq^{\dagger}}
\newcommand{\VGamma}{\widehat{\Gamma}}
\newcommand{\tresdos}{\frac{3}{2}}
\newcommand{\cincodos}{\frac{5}{2}}
\newcommand{\idos}{\frac{2i+1}{2}}

\begin{document}
\allowdisplaybreaks

\newcommand{\arXivNumber}{2108.02608}

\renewcommand{\PaperNumber}{021}

\FirstPageHeading

\ShortArticleName{Rank 4 Nichols Algebras of Pale Braidings}

\ArticleName{Rank 4 Nichols Algebras of Pale Braidings}

\Author{Nicol\'as ANDRUSKIEWITSCH~$^{\rm a}$, Iv\'an ANGIONO~$^{\rm a}$ and Mat\'{\i}as MOYA GIUSTI~$^{\rm b}$}

\AuthorNameForHeading{N.~Andruskiewitsch, I.~Angiono and M.~Moya Giusti}

\Address{$^{\rm a)}$~Facultad de Matem\'atica, Astronom\'{\i}a, F\'{\i}sica y Computaci\'on,\\
\hphantom{$^{\rm a)}$}~Universidad Nacional de C\'ordoba, CIEM -- CONICET,\\
\hphantom{$^{\rm a)}$}~Medina Allende s/n (5000) Ciudad Universitaria, C\'ordoba, Argentina}
\EmailD{\href{mailto:nicolas.andruskiewitsch@unc.edu.ar}{nicolas.andruskiewitsch@unc.edu.ar}, \href{mailto:ivan.angiono@unc.edu.ar}{ivan.angiono@unc.edu.ar}}
\URLaddressD{\url{https://www.famaf.unc.edu.ar/~andrus/},\newline
\hspace*{14.5mm}\url{https://www.famaf.unc.edu.ar/~angiono/}}

\Address{$^{\rm b)}$~6 rue Rampal, 75019, Paris, France}
\EmailD{\href{mailto:matiasmoyagiusti@gmail.com}{matiasmoyagiusti@gmail.com}}

\ArticleDates{Received November 25, 2022, in final form March 21, 2023; Published online April 13, 2023}

\Abstract{We classify finite GK-dimensional Nichols algebras $\toba(V)$ of rank 4 such that $V$ arises as a Yetter--Drinfeld module over an abelian group but it is not a direct sum of points and blocks.}


\Classification{16T05}

\section{Introduction}
\subsection{The context}
Let $\Bbbk$ be an algebraically closed field of characteristic zero.
The problem of classifying Hopf algebras with finite Gelfand--Kirillov dimension, abbreviated $\GK$ henceforth,
is an active area of research. See \cite{AAH-triang,ASa-preNichols,angiono-campagnolo-sanmarco,brown-zhang}
and references therein.
Crucial for this problem and attractive in itself is the question of classifying Nichols algebras over abelian groups with finite~$\GK$; see~\cite{A-nilpotent} for its role in the study of pointed Hopf algebras over nilpotent groups.
Let $\Gamma$ be an abelian group and let $\Bbbk \Gamma$ be its group algebra.
The braided tensor category $\yd{\Bbbk\Gamma}$ of Yetter--Drinfeld modules over $\Bbbk \Gamma$
consists of $\Gamma$-graded $\Gamma$-modules, i.e., vector spaces $V = \oplus_{g \in \Gamma} V_g$
with a linear action of $\Gamma$ such that $h \cdot V_g = V_g$ for all $g,h\in \Gamma$, with usual tensor product of modules and gradings.
The braiding $c_{V,W}\colon V \otimes W \to W \otimes V$, for $V, W\in \yd{\Bbbk\Gamma}$, is given by
\begin{gather}\label{eq:braiding-ydG}
c_{V,W} (v \otimes w) = g\cdot w \otimes v, \qquad
v\in V_g, \quad g \in \Gamma, \quad w\in W.
\end{gather}
Given $V = \oplus_{g \in \Gamma} V_g \in\yd{\Bbbk\Gamma}$, its support is
$\supp V = \{g\in \Gamma\colon V_g \neq 0 \}$.
Since the Nichols algebra $\toba(V)$ depends only on the braiding, the question of classifying those $V$ with $\GK \toba(V) < \infty$
was approached via Nichols algebras of suitable classes of braided vector spaces. Concretely, we mention:
\begin{enumerate}[label=$(\alph*)$]\itemsep=0pt
\item\label{item:diagonal}
\emph{Braided vector spaces of diagonal type} (see Section~\ref{subsec:diagonal-type} for details).
\end{enumerate}
Nichols algebras arising from this class satisfy the following:

\begin{Theorem}[\cite{angiono-garcia2}]\label{thm:conj-AAH} The root system
of a Nichols algebra of diagonal type with finite GK-dimension is finite.
\end{Theorem}

This result was conjectured in \cite[Conjecture 1.3.3]{AAH-triang}, with supporting evidence from \cite{AA,AAH-diag,angiono-garcia, R-quantum-groups}.
By Theorem~\ref{thm:conj-AAH}, the classification of Nichols algebras of diagonal type with finite GK-dimension
follows from \cite{H-classif-RS}.
\begin{enumerate}[label=$(\alph*)$,resume]\itemsep=0pt
\item\label{item:blocks}
\emph{Blocks}.
\end{enumerate}

These are the braided vector spaces $\Vc(\epsilon,\ell)$, where $\epsilon\in \Bbbk^{\times}$ and $\ell \in \N_{\ge 2}$,
with a basis $(x_i)_{i\in\I_\ell}$ such that for $i, j \in \I_\ell$, $1 < j$:
\begin{gather*}
c(x_i \otimes x_1) = \epsilon x_1 \otimes x_i, \qquad
c(x_i \otimes x_j) =(\epsilon x_j+x_{j-1}) \otimes x_i.
\end{gather*}

\begin{Theorem}[{\cite[Theorem 1.2.2]{AAH-triang}}]\label{th:blocks}
$\GK \toba(\Vc(\epsilon,\ell)) < \infty$ if and only if $\ell = 2$ and $\epsilon \in \{\pm 1\}$, in which case $\GK \toba(\Vc(\epsilon,\ell)) = 2$.
\end{Theorem}

Here $\toba(\Vc(1,2))$ is the well-known Jordan plane while
$\toba(\Vc(-1,2))$ is called the super Jordan plane; the adjective super is justified in~\cite{ap1}.
\begin{enumerate}[label=$(\alph*)$,resume]\itemsep=0pt
\item\label{item:blocks-points}
\emph{Direct sums of blocks and points}.
\end{enumerate}

Here a point is a braided vector space of dimension 1 and the blocks are of the form $\Vc(\epsilon,2)$,
$\epsilon \in \{\pm 1\}$.
We require at least two blocks, or one block and at least one point (to avoid overlaps with the previous classes),
and specific types of braidings between blocks and points, or between blocks (from realizations in categories of Yetter--Drinfeld modules over groups).
The precise definition is in \cite[Section~1.3.1]{AAH-triang}.
The classification of the Nichols algebras with finite $\GK$ of such braided vector spaces is \cite[Theorem~1.3.8]{AAH-triang}.
\begin{enumerate}[label=$(\alph*)$,resume]\itemsep=0pt
\item \label{item:paleblock-point} \emph{Sums of one pale block and one point}.
\end{enumerate}

Any finite-dimensional Yetter--Drinfeld module is a direct sum of indecomposable subobjects in~$\yd{\Bbbk\Gamma}$.
If the underlying braided vector space of $U\in\yd{\Bbbk\Gamma}$ is a block, then $U$ is indecomposable in $\yd{\Bbbk\Gamma}$ but the converse is not true. An indecomposable $U\in\yd{\Bbbk\Gamma}$ which is not a block, i.e., is not an indecomposable braided vector space, is called a \emph{pale block}. These appear already in dimension 3.
Thus a braided vector space $V$, $\dim V = 3$, is a direct sum of of one pale block and one point if
$V = V_1 \oplus V_2$ where $V_1$ is a pale block and $V_2 $ is a point. This turns out to mean that there exist
\begin{itemize}
\itemsep=0pt
\item a basis $(x_i)_{1\le i \le 3}$ such that $V_1$ is generated by $x_1$ and $x_2$ and $V_2 = \Bbbk x_3$ and
\item a matrix $(q_{ij})_{1 \le i,j \le 2}$ of non-zero scalars
\end{itemize}
such that the braiding is given by
\begin{align}\label{eq:braiding-paleblock-point}
(c(x_i \otimes x_j))_{1 \le i,j \le 3} &=
\begin{pmatrix}
q_{11} x_1 \otimes x_1& q_{11} x_2 \otimes x_1& q_{12} x_3 \otimes x_1
\\
q_{11} x_1 \otimes x_2 & q_{11} x_2 \otimes x_2& q_{12} x_3 \otimes x_2
\\
q_{21} x_1 \otimes x_3 & q_{21}(x_2 + x_1) \otimes x_3& q_{22} x_3 \otimes x_3
\end{pmatrix}\!.
\end{align}
Indeed, it can be shown that such $V$ has a braiding like this \cite[Sections~4.1 and~8.1]{AAH-triang} and conversely
we realize a braided vector space $V$ with braiding~\eqref{eq:braiding-paleblock-point}
in $\yd{\Bbbk\Gamma}$, where $\Gamma = \Z^2$ with a basis $g_1$, $g_2$,
by $V_1 = V_{g_1}$, $V_2 = V_{g_2}$, $g_1\cdot x_1 = q_{11} x_1$, $g_2\cdot x_1 = q_{21} x_1$,
$g_1\cdot x_2 = q_{11} x_2$, $g_2\cdot x_2 = q_{21} (x_2 + x_1)$, $g_i\cdot x_3 = q_{i2} x_3$.

The underlying braided vector space of any Yetter--Drinfeld module of dimension 3 over an abelian group
belongs to one of the classes~\ref{item:diagonal}, \ref{item:blocks}, \ref{item:blocks-points} or~\ref{item:paleblock-point},
see \cite[Sections~4.1 and 8.1]{AAH-triang}. Below we shall use the notation $\widetilde{q}_{12} := q_{12}q_{21}$.

\begin{Theorem}[{\cite[Theorem 8.1.3]{AAH-triang}}]\label{th:paleblock-point-resumen}
Let $V$ be a braided vector space of dimension~$3$ with braiding~\eqref{eq:braiding-paleblock-point}.
Then $\GK \toba(V) < \infty$ if and only if
$q_{11} =-1$ and either of the following holds:
\begin{enumerate}[label=$(\roman*)$]\itemsep=0pt
\item
$\widetilde{q}_{12} = 1$ and $q_{22} = \pm 1$; in this case $\GK \toba(V) = 1$.

\item
$q_{22} = -1 = \widetilde{q}_{12}$; in this case $\GK \toba(V) = 2$.
\end{enumerate}
\end{Theorem}

The Nichols algebras in the theorem are described in Proposition~\ref{prop:paleblocks}.

\subsection{The main theorem}
Because of these antecedents, we consider the class $\pale$ of finite-dimensional braided vector spaces~$V$
with \emph{pale braiding} \cite{AAH-triang}, i.e., such that
\begin{itemize}\itemsep=0pt
\item $V$ can be realized as Yetter--Drinfeld module over an abelian group,
\item $V$ does not belong to classes~\ref{item:diagonal}, \ref{item:blocks}, nor~\ref{item:blocks-points}.
\end{itemize}

The problem is to determine when $\GK \toba(V) < \infty$ for $V \in \pale$.
Without loss of generality, we restrict ourselves to the following setting.
\begin{hypothesis}\label{hyp:pale}
$\Gamma$ is an abelian group and $V \in \yd{\Bbbk\Gamma}$ satisfies
\begin{enumerate}[label=\rm{(\Roman*)}]
\item\label{item:hyp-pale} $V \in \pale$,

\item\label{item:hyp-suppgen} $\sup V$ generates~$\Gamma$,

\item\label{item:hyp-connected} $V$ is connected, see Definition~$\ref{def:connected}$.
\end{enumerate}
\end{hypothesis}

Indeed, if~\ref{item:hyp-suppgen} does not hold, then
we replace $\Gamma$ by the subgroup generated by the support.
Also~\ref{item:hyp-connected} is controlled by Remark~\ref{obs:connected}.

Let $\Gamma$ and $V$ be as in
Hypothesis~\ref{hyp:pale}. To deal with our problem,
we consider the possible decompositions of $V$ in indecomposable Yetter--Drinfeld submodules.
Some cases are ruled out by our assumptions:
\begin{itemize}\itemsep=0pt
\item If $V$ is indecomposable, then by~\ref{item:hyp-suppgen} $V=V_g$ for some $g \in \Gamma$ and $g$ generates $\Gamma$.
Thus $g$ must act as a Jordan block of some eigenvalue $\epsilon$; i.e., $V$ is either a point or a block,
so it is not in $\pale$ since it belongs to class~\ref{item:diagonal} or~\ref{item:blocks}.

\item If $V$ is a direct sum of Yetter--Drinfeld submodules of dimension $1$, then it is of diagonal type, again out of $\pale$.
\end{itemize}

Suppose further that
$\dim V = 4$. There are three cases of decompositions
$V = V_1 \oplus V_2 \oplus \dots \oplus V_{\theta}$
where $\dim V_1 \geq \dim V_2 \geq \dots \geq \dim V_{\theta}$ and the $V_j\in \yd{\Bbbk\Gamma}$ are indecomposable to be considered,
namely
\begin{enumerate}[label=$(\arabic*)$]\itemsep=0pt
\item $\theta = 2$, $\dim V_1 = 3$ and $\dim V_2 = 1$,

\item $\theta = 3$, $\dim V_1 = 2$ and $\dim V_2 = \dim V_3 = 1$,

\item $\theta = 2$, $\dim V_1 = \dim V_2 = 2$.
\end{enumerate}

The classification of the possible $V$ with $\GK \toba(V) < \infty$ is carried out in each case in Sections
\ref{sec:block+pt}, \ref{sec:paleblock+2pts} and~\ref{sec:2blocks}, respectively, using Theorem~\ref{thm:conj-AAH}.
Putting together
the corresponding results, see Theorems~\ref{th:large-block-point}, \ref{th:2-block-points} and~\ref{th:2-blocks},
we get our main theorem:

\begin{Theorem}\label{th:paleblock-dim4}
Let $V$ be a braided vector space of dimension $4$ satisfying Hypothesis~$\ref{hyp:pale}$.
Then $\GK \toba(V) < \infty$ if and only if $V$ is in Table~$\ref{tab:pale-rk4}$.
\end{Theorem}

For the meaning of the graphical description in the last column in Table~\ref{tab:pale-rk4}, we refer to Section~\ref{subsubsec:terminology}.

\begin{table}[ht]
\caption{Pale braidings of rank 4 with finite GK.} \label{tab:pale-rk4}\vspace{1mm}
\centering
\begin{tabular}{c| c c c |c}
\hline Shape & Name & \multicolumn{1}{|c|}{\centering $\GK$} & Theorem & Diagram$\vphantom{\Big(}$
\\ \hline
{\small 1 pale block} & $\eny_{3, -}(q)$ & 2 & \ref{thm:q=1,q22=-1,n=3} & $\xymatrix{ \underset{1}{\begin{tikzpicture}
\draw[dashed] (0,0) rectangle node[pos=.5] {$3$} (0.36,0.36);
\end{tikzpicture}} \ar @{.}[r] & \overset{- 1}{\underset{2}{\bullet}}}$
\\ \cline{2-4}
$\&$ {\small 1 point} & $\eny_{3, +}(q)$ & 4 & \ref{thm:q=1,q22=1,n=3} & $\xymatrix{ \underset{1}{\begin{tikzpicture}
\draw[dashed] (0,0) rectangle node[pos=.5] {$3$} (0.36,0.36);
\end{tikzpicture}} \ar @{.}[r] & \overset{1}{\underset{2}{\bullet}}}$
\\ \hline

{\small 1 pale block}
& $\eny_{\mu, \nu}\big(\bqs, a\big)$ & 2 & \ref{th:endymion-rank4-1} & $\xymatrix{
\overset{\nu}{\underset{3}{\bullet}} \ar@{.}^{a}[r]
&\underset{1}{\begin{tikzpicture}
\draw[dashed] (0,0) rectangle (0.36,0.36);
\end{tikzpicture}} \ar @{.}[r]
& \overset{\mu}{\underset{2}{\bullet}}}$
\\ \cline{2-4}
$\&$ {\small 2 points} & $\enyIIb{\infty}\big(\bqs\big)$ & 4 & \ref{th:endymion-2pts-case-a} &
$\xymatrix{
\overset{-1}{\underset{3}{\bullet}} \ar@{-}^{0}[r]
&\underset{1}{\begin{tikzpicture}
\draw[dashed] (0,0) rectangle (0.36,0.36);
\end{tikzpicture}} \ar @{.}[r]
& \overset{1}{\underset{2}{\bullet}}}$
\\ \hline
{\small 2 pale blocks} & $\sele_{2,0}(q)$ & 2 & \ref{th:selene-2-0} &
$\xymatrix{
\underset{1}{\begin{tikzpicture}
\draw[dashed] (0,0) rectangle (0.36,0.36);
\end{tikzpicture}} \ar@{.}^{(1,0)}[r]
&\underset{2}{\begin{tikzpicture}
\draw[dashed] (0,0) rectangle (0.36,0.36);
\end{tikzpicture}} }$
\\ \hline
{\small 1 pale block}
& $\sele_{1,+}\big(q,-\frac12\big)$ & 2 &\ref{th:selene-1+}
& $
\xymatrix{
\underset{1}{\begin{tikzpicture}
\draw[dashed] (0,0) rectangle (0.36,0.36);
\end{tikzpicture}} \ar @{.}^{(-\frac12,1)}[rr]
& &\underset{2}{\boxplus}}$
\\
$\&$ {\small 1 block}
& $\sele_{1,+}(q,-1)$ & 4 &\ref{th:selene-1+}
& $
\xymatrix{
\underset{1}{\begin{tikzpicture}
\draw[dashed] (0,0) rectangle (0.36,0.36);
\end{tikzpicture}} \ar @{.}^{(-1,1)}[rr]
& &\underset{2}{\boxplus}}$
\\ \cline{2-4}
& $\sele_{1,-}(q)$ & 4 &\ref{th:selene-1-}
& $ \xymatrix{ \underset{1}{\begin{tikzpicture}
\draw[dashed] (0,0) rectangle (0.36,0.36);
\end{tikzpicture}} \ar @{.}^{(1,1)}[rr]
& &\underset{2}{\boxminus}}$
\\
\hline
\end{tabular}
\end{table}

Theorem~\ref{th:paleblock-dim4} is the crucial recursive step towards the classification of the Nichols algebras
satisfying Hypothesis~\ref{hyp:pale} and having finite Gelfand--Kirillov dimension, that is presently work in progress.
Indeed, we can show that the members of the list in Table~\ref{tab:pale-rk4}
either belong to natural families of braided vector spaces giving rise to Nichols algebras with finite Gelfand--Kirillov dimension
or else could not be extended to such a family.
Now the technical difficulties presented by the working Hypothesis~\ref{hyp:pale} prevent us from arguing inductively
in a naive way, and in fact there are new families beyond such a recursion, but the constraints given by Theorem~\ref{th:paleblock-dim4} would make this question tractable.

\section{Preliminaries}
\subsection{Conventions}
For us $\mathbb{N} = \{1,2,\dots\}$ and $\N_0 = \N \cup \{0\}$.
If $\ell \leq \theta \in\N_0$, then $\I_{\ell, \theta} := \{\ell, \ell +1,\dots,\theta\}$, $\I_{\theta} := \I_{1, \theta}$.
The cardinal of a set $I$ is denoted by $\vert I \vert$.
The antipode of a Hopf algebra is denoted by $\Ss$.
Given a vector space $V$, $\langle v_1, \dots, v_n\rangle$ denotes the subspace spanned by $v_1, \dots, v_n \in V$.
Given an algebra $A$, $\Bbbk\langle x_1, \dots, x_n\rangle$ denotes the subalgebra generated by $x_1, \dots, x_n \in A$.

\subsection{Nichols algebras}
Let $\Gamma$ be an abelian group.
The category $\yd{\Bbbk \Gamma}$ of Yetter--Drinfeld modules over $\Bbbk\Gamma$
was already defined; we refer to the literature for that of $\yd{H}$, $H$ a general Hopf algebra. See, e.g., Section~\ref{subsec:category-bvs}
for the concept of braided vector space and
\cite{A-leyva} for the notions of braided Hopf algebras and Hopf algebras in braided tensor categories.
Fix $R$ a Hopf algebra in $\yd{H}$. The braided commutator of $x,y\in R$ is $[x,y]_c = xy - \text{multiplication} \circ c (x\otimes y)$.
Let $\adc$ denote the braided adjoint action of $R$, see, e.g., \cite[p.~165]{A-leyva};
if $x\in R$ is primitive, then
$\adc x (y) =[x,y]_c$ for all $y\in R$.

\begin{Remark}
Let
$\toba$ be an algebra in $\yd{\Bbbk \Gamma}$ and $u,v,w\in\toba$ homogeneous of degrees $g,h,k\in\Gamma$. Then
\begin{gather}\label{eq:graded-bracket-derivation1}
[uv,w]_c = u[v,w]_c + [u,h\cdot w]_c v,
\\
\label{eq:graded-bracket-derivation2}
[u,vw]_c = [u,v]_c w + g\cdot v [u,w]_c,
\\
\label{eq:graded-bracket-jacobi}
[[u,v]_c,w]_c = [u,[v,w]_c]_c -(g\cdot v)[u,w]_c +[u,(h\cdot w)]_c v.
\end{gather}
These identities will be used frequently, sometimes implicitly, in what follows.
\end{Remark}

Given $V \in \yd{H}$, the tensor algebra $T(V)$ is naturally a Hopf algebra in $\yd{H}$.
The Nichols algebra $\toba(V)$ is a quotient of $T(V)$ by a suitable homogeneous Hopf ideal; see \cite{A-leyva} for details.

Let $V\in\yd{\Bbbk \Gamma}$ with a basis $(v_i)_{i\in \I_\theta}$ such that $v_i$ is homogeneous of degree $g_i$ for all $i$.
Then there are skew-derivations $\partial_i$, $i\in \I_{\theta}$, of $T(V)$ such that
\begin{align*}
\partial _i(v_j) =\delta _{ij},\qquad
\partial _i(xy) = \partial _i(x)(g_i\cdot y)+x\partial _i(y),\qquad
x,y\in T(V), \quad i, j \in \I_{\theta}.
\end{align*}
These skew-derivations extend to $\toba(V)$. Given $x\in \toba(V)$, if
$\partial _i(x) = 0$ for all $i\in \I_\theta$, then $x=0$.

Given a braided vector space $V$ with a basis $(x_i)_{i\in\I_{\theta}}$, we denote in any intermediate Hopf algebra between $T(V)$ and $\toba(V)$
\begin{align*}
x_{i_1\cdots i_k i_{k+1}}= (\adc x_{i_1}) \cdots (\adc x_{i_k})x_{i_{k+1}},\qquad i_1,\ldots,i_{k+1}\in\I_{\theta}.
\end{align*}

We refer to~\cite{KL} for the theory of Gelfand--Kirillov dimension.
By~\cite{Ufer}, the Nichols algebras considered here admit a PBW-basis; we derive the $\GK$, when finite, from
the explicit computation of one such PBW-basis. To decide that the $\GK$ is infinite, we use instead a~variety of arguments, mostly
reducing to a subalgebra or quotient algebra; in some cases we use Theorem~\ref{thm:conj-AAH}:
explicitly, in Lemmas~\ref{lemma:caseII-general} and~\ref{lemma:caseIII-general}
and in Proposition~\ref{prop:caseIV}.

\subsubsection{The splitting technique}\label{subsec:splitting-technique}

Let $V = U \oplus W$ be a direct sum of Yetter--Drinfeld modules over a Hopf algebra $H$.
Then $\toba(V)$ splits as
\begin{equation*} 
\toba(V) \cong \K \# \toba(W)
\end{equation*}
with $\K = \toba(V)^{\co \toba(W)}$.
Further, $\K$ is isomorphic to the Nichols algebra of $\K^1 =\adc (\toba(W))(U)$,
see \cite[Proposition~8.6]{HS-adv}, and also \cite[Lemma~3.2]{AHS}.
It is often easier to compute $\toba\big( \K^1\big)$ and then derive $\toba(V)$.

\section{Indecomposable Yetter--Drinfeld modules}

\subsection{The category of braided vector spaces}\label{subsec:category-bvs}
A braided vector space is a pair $(V,c)$ where $V$ is a vector space and $c \in GL(V \otimes V)$ is a solution of the braid equation
$(c\otimes \id) (\id \otimes c)(c\otimes \id) = (\id \otimes c)(c\otimes \id)(\id \otimes c)$.
As customary, the braiding of any braided vector space is denoted by $c$.
We assume that all braidings are rigid.
The class of braided vector spaces is a category, where a morphism $f\colon (W, c) \to (W', c)$
is a linear map $f\colon W \to W'$ such that $(f \otimes f)c = c(f \otimes f)$.
A collection of morphisms of braided vector spaces is an exact sequence if the underlying collection of
linear maps is so.

\begin{Definition}
A braided vector space $(W, c)$ is \emph{simple} if $W \neq 0$
and for any exact sequence
$0 \to (U,c) \to (W, c) \to (V,c) \to 0$ of braided vector spaces, either $U=0$ or else $V=0$.
\end{Definition}

There is a forgetful functor from $\yd{\Bbbk\Gamma}$ to the category of braided vector spaces sending
$V \in \yd{\Bbbk\Gamma}$ to $(V, c_{V,V})$, cf.~\eqref{eq:braiding-ydG}.

Following~\cite{T}, a braided subspace $(U,c)$ of $(W,c)$ is \emph{categorical} if
\begin{align*}
&c(U \otimes W) = W \otimes U\qquad \text{and}\qquad c(W \otimes U) = U \otimes W.
\end{align*}

Let $(U,c)$ be a categorical braided subspace of $(W,c)$. By~\cite[Proposition~6.6]{T}, there
exists a Hopf al\-ge\-bra~$K$ such that
\begin{itemize}\itemsep=0pt
\item $W \in \yd{K}$ and $U$ is a subobject of $W$ in $\yd{K}$,
\item the braidings of $W$ and $U$ coincide with those in $\yd{K}$.
\end{itemize}
Actually, $K$ can be chosen co-quasi-triangular so that $W$ and $U$ are just $K$-comodules with braiding arising from the universal $R$-matrix.

As in \cite[Definition 2.1]{gr-jalg}, a decomposition of a braided vector space $W$ is a family of non-zero subspaces
$(W_i)_{i \in I}$ such that
\begin{align*}
W = \oplus_{i \in I} W_i,\qquad
c(W_i \otimes W_j) = W_j \otimes W_i,\qquad i,j\in I.
\end{align*}
Given such a decomposition, every $W_i$ is a categorical subspace. By~\cite[Proposition~6.6]{T}, there
exists a~Hopf algebra $K$ such that $W = \oplus W_i$ is a direct sum in $\yd{K}$ with braidings coming from~$\yd{K}$.
We~say that a braided vector space $(W,c)$ is \emph{decomposable} if it admits a decomposition
with $\vert I \vert \geq 2$; otherwise, it is \emph{indecomposable}.
In this way, if $W \in \yd{K}$ is indecomposable as braided vector space, then it is indecomposable as Yetter--Drinfeld module,
but the converse is not true: there are simple Yetter--Drinfeld modules of dimension 2 over group algebras
that are of diagonal type as braided vector spaces.

\begin{Definition}\label{def:connected}
Let $W = \oplus_{i \in I} W_i$ be a decomposition of a braided vector space $W$.
Set $c_{ij} = c_{\vert W_i \otimes W_j}\colon W_i \otimes W_j \to W_j \otimes W_i$;
$i \sim j$ when $c_{ij}c_{ji} \neq \id_{W_j \otimes W_i}$, $i \neq j \in I$; and
let $\approx$ be the equivalence relation generated by $\sim$. We say that $W$ is \emph{connected} if
$i \approx j$ for all $i, j \in \I_\theta$.
\end{Definition}

\begin{Remark}\label{obs:connected} Let $W = \oplus_{i \in I} W_i$ be a decomposition of a
braided vector space $W$ such that $\dim W < \infty$ and
$c_{ij}c_{ji} = \id_{W_j \otimes W_i}$ for every pair $i,j \in I$. Then
$\toba(W) \simeq \underline{\otimes}_i\toba(W_i)$ \cite{gr-jalg} and $\GK\toba(W) = \sum_i \GK \toba(W_i)$.
\end{Remark}

We make precise a notion from \cite{AAH-triang}. Let $K$ be a Hopf algebra.

\begin{Definition}
We say that
$W \in \yd{K}$, $\dim W < \infty$, is a \emph{pale block} if it is decomposable as braided vector space but indecomposable in $\yd{K}$.
\end{Definition}

Thus there is a difference between the study of Nichols algebras of simple or indecomposable
braided vector spaces and ditto of simple or indecomposable Yetter--Drinfeld modules.

\subsubsection{Indecomposable modules of dimension 2}
Let $K$ be a Hopf algebra.
As illustration, we describe the indecomposable but not simple objects in $\yd{K}$ of dimension 2.
The one-dimensional objects in $\yd{K}$ are parametrized by
\emph{YD-pairs}, that is pairs
$(g, \chi) \in G(K) \times \Hom_{\text{alg}}(K, \Bbbk)$ such that
\begin{align}\label{eq:yd-pair}
\chi(k) g = \chi(k_{2}) k_{1} g \Ss(k_{3})\qquad \text{for all}\quad k\in K.
\end{align}
If $(g, \chi)$ is a YD-pair, then $g\in Z(G(K))$; also, the vector space $\Bbbk_g^{\chi}$ of dimension 1,
with action and coaction given by $\chi $ and $g$, is in $\yd{K}$.

Let $\chi_1, \chi_2 \in \Hom_{\text{alg}}(K, \Bbbk)$. The space of $(\chi_1,\chi_2)$-derivations is
\begin{align*}
\Der_{\chi_1,\chi_2}(K) = \{\eta\in K^*\colon \eta(kt) = \chi_1(k)\eta(t) + \eta(k)\chi_2(t), \, k,t \in K\}.
\end{align*}
For example, $\chi_1 - \chi_2 \in \Der_{\chi_1,\chi_2}(K)$.
Dually, let $g_1, g_2 \in G(K)$. The space of $(g_2,g_1)$-skew primitive elements is
\begin{align*}
\Pc_{g_2,g_1}(K) = \{k\in K\colon \Delta(k) = g_2\otimes k + k \otimes g_1\}.
\end{align*}
For example, $g_1 - g_2 \in \Pc_{g_2,g_1}(K)$.

\begin{Definition}
A \emph{rank} 2 \emph{YD-block} for $K$ is a collection $(g_1, g_2, \chi_1, \chi_2, \eta, \nu)$,
where{\samepage
\begin{enumerate}[label=$(\alph*)$]\itemsep=0pt
\item\label{eq:-ydblock-a} $(g_i, \chi_i)$, is a YD-pair for $K$, $i \in \I_2$;

\item\label{eq:-ydblock-b} $\eta \in \Der_{\chi_1,\chi_2}(K)$;

\item\label{eq:-ydblock-c} $\nu \in \Pc_{g_2,g_1}(K)$, and for all $k\in K$
\begin{align}\label{eq:YD-triple}
&\chi_2(k) \nu + \eta(k) g_1 = \chi_1(k_2) k_1 \nu \Ss(k_3)
+ \eta(k_2) k_1 g_2 \Ss(k_3).
\end{align}
\end{enumerate}}
\end{Definition}

\begin{Remark} The following sets are subalgebras of $K$:
\begin{itemize}\itemsep=0pt
\item given $(g, \chi) \in G(K) \times \Hom_{\text{alg}}(K, \Bbbk)$, $\{k\in K\colon \eqref{eq:yd-pair} \text{ holds} \}$;

\item provided that~\ref{eq:-ydblock-a}, \ref{eq:-ydblock-b} and \ref{eq:-ydblock-c} are valid, $\{k\in K\colon \eqref{eq:YD-triple} \text{ holds} \}$.
\end{itemize}
\end{Remark}

Let $(g_1, g_2, \chi_1, \chi_2, \eta, \nu)$ be a YD-block for $K$.
Let $\Vc_{g_1, g_2}^{\chi_1, \chi_2}(\eta, \nu)$ be the vector space with a basis $(x_i)_{i\in\I_2}$,
with action and coaction of $K$ given by
\begin{alignat*}{2}
&k\cdot x_1 = \chi_1(k) x_1,\qquad&& k\cdot x_2=\chi_2(k) x_2 + \eta(k)x_{1},\qquad k\in K,
\\
&\delta(x_1) = g_1\otimes x_1, && \delta(x_2) = \nu\otimes x_1 + g_2\otimes x_2.
\end{alignat*}

\begin{Proposition}\quad
\begin{enumerate}[label=$(\roman*)$]\itemsep=0pt
\item
$\Vc_{g_1, g_2}^{\chi_1, \chi_2}(\eta, \nu) \in \yd{K}$;
it is decomposable in $\yd{K}$ iff
\begin{gather*}
\eta = a(\chi_1 - \chi_2)\qquad \text{and}\qquad \nu = a (g_1 - g_2)\qquad
\text{for some} \quad a\in \Bbbk.
\end{gather*}

\item
Let $\Vc\in \yd{K}$ not simple with $\dim \Vc =2$.
Then $\Vc \simeq \Vc_{g_1, g_2}^{\chi_1, \chi_2}(\eta, \nu)$ for some YD-block $(g_1, g_2, \chi_1, \chi_2, \eta, \nu)$.
\end{enumerate}
\end{Proposition}

\begin{proof}
Left to the reader.
\end{proof}

\subsection{Pale blocks over abelian groups}
Let $\Gamma$ be an abelian group.

\subsubsection{Recollections}
Given $V = \oplus_{g\in \Gamma} V_g \in \yd{\Bbbk\Gamma}$, $\dim V < \infty$, we set
\begin{gather*}
V_g^{\lambda} := \ker (g - \lambda \id)_{\vert V_g} \subseteq V_g^{(\lambda)}
:= \bigcup_{n\in \N} \ker (g - \lambda \id)^n_{\vert V_g}, \qquad
\lambda \in \Bbbk^{\times}.
\end{gather*}
Then $V = \bigoplus_{\substack{g\in \Gamma, \\ \lambda \in \Bbbk^{\times}}} V_g^{(\lambda)}$ is a direct sum in $\yd{\Bbbk\Gamma}$, hence
\begin{align*}
c\big(V_g^{(\lambda)} \otimes V_h^{(\mu)}\big) = V_h^{(\mu)} \otimes V_g^{(\lambda)}, \qquad
g, h \in \Gamma,\quad \lambda,\mu \in \Bbbk^{\times}.
\end{align*}

\begin{Lemma}[{\cite[Lemma 8.1.1]{AAH-triang}}]
Assume that $\GK \toba(V_g)< \infty$. Then
\begin{enumerate}[label=$(\alph*)$]\itemsep=0pt
\item
If $\lambda \in \Bbbk^{\times}$, $\lambda \notin \G_2 \cup \G_3$, then
$V_g^{\lambda} = V_g^{(\lambda)}$ has dimension $\leq 1$.

\item
If $\lambda \in \G'_3$, then
$V_g^{\lambda} = V_g^{(\lambda)}$ has dimension $\leq 2$.

\item
If $V_g^{1} \neq 0$, then either
$V_g = V_g^{1}$ $($i.e., $g$ acts trivially on $V_g)$ or else $V_g$ has dimension $2$ and $g$ acts by a Jordan block.

\item
If $V_g^{-1} \neq 0$, then either
$V_g^{(-1)} = V_g^{-1}$ or else $V_g^{(-1)}$ has dimension $2$ and $g$ acts by a Jordan block.
\end{enumerate}
\end{Lemma}

\begin{Corollary}\label{coro:indecomposable-abelian}
Let $V \in \yd{\Bbbk\Gamma}$ be indecomposable, thus $V = V_{g}^{(\lambda)}$
for some $g\in \Gamma$, $\lambda\in \Bbbk^{\times}$. Then $\GK\toba(V) <\infty$
iff either of the following holds:
\begin{enumerate}[label=$(\alph*)$]\itemsep=0pt
\item\label{item:indecomposable-abelian1} $V$ is simple, i.e., $\dim V = 1$, or

\item\label{item:indecomposable-abelian2} $\dim V = 2$, $g$ acts by a Jordan block where $\lambda = \pm 1$, or

\item\label{item:indecomposable-abelian3} $\dim V = 2$, $g$ acts by $\lambda \id$ where $\lambda \in \G'_3$, or

\item\label{item:indecomposable-abelian4} $\dim V \geq 2$, $g$ acts by $\lambda \id$ where $\lambda = \pm 1$.
\end{enumerate}
\end{Corollary}
Clearly, $V$ is indecomposable as braided vector space only in cases~\ref{item:indecomposable-abelian1}
and~\ref{item:indecomposable-abelian2}, thus $V$ is a~pale block in cases~\ref{item:indecomposable-abelian3}
and~\ref{item:indecomposable-abelian4}.

\subsubsection{Diagonal type}\label{subsec:diagonal-type} Let $V \in \yd{\Bbbk\Gamma}$ be semisimple, $\dim V = \theta \in \N$; then
$V$ has a basis $(x_i)_{i\in \I_{\theta}}$ such that $x_i\in V_{g_i}$ and $g\cdot x_i = \chi_i(g) x_i$ for some $g_i\in \Gamma$
and $\chi_i\in \VGamma$, for all $i\in \I_{\theta}$. Hence the braiding is given by $c(x_i\otimes x_j) = q_{ij} x_j \otimes x_i$,
$i, j\in \I_{\theta}$. Such braided vector spaces are called of diagonal type and have been studied intensively, see
\cite{A-leyva,AA-diag-survey,AS,H-classif-RS} and their references.
The Dynkin diagram of the braided vector space defined by the matrix $(q_{ij})_{i, j\in \I_{\theta}}$ has $\theta$ vertices, the $i$-th
vertex labeled by $q_{ii}$; and one edge between $i$ and $j \neq i$ labeled by $\widetilde{q}_{ij} = q_{ij}q_{ji}$
(the edge is omitted when $\widetilde{q}_{ij} = 1$).\looseness=1

\subsubsection{Pale braidings of rank 3}
Let $q\in \Bbbk^{\times}$. As in \cite{AAH-triang},
we name the braided vector spaces with braiding \eqref{eq:braiding-paleblock-point}
with $q_{11}= -1$, cf.~Theorem~\ref{th:paleblock-point-resumen}, as follows:
\begin{itemize}\itemsep=0pt
\item $\eny_{\pm}(q)$, when $q_{12} = q = q_{21}^{-1}$, $q_{22} = \pm 1$;
\item $\eny_{\star}(q)$, when $q_{22}=-1$, $q_{12}=q$, $q_{21}=-q^{-1}$.
\end{itemize}

The Nichols algebras $\toba(\eny_{\pm}(q))$ and $\toba(\eny_{\star}(q))$ are called the \emph{Endymion algebras} of rank 3.
In the next proposition, $x_{\tresdos2} := x_{\tresdos}x_2 - q _{12} x_2x_{\tresdos}$.

\begin{Proposition}[{\cite[Propositions 8.1.6, 8.1.7 and 8.1.8]{AAH-triang}}]\label{prop:paleblocks}
The Endymion algebras are generated by $x_1$, $x_{\tresdos}$, $x_2$ with defining relations and PBW-basis as follows:
\begin{enumerate}[label=$(\alph*)$]\itemsep=0pt
\item The relations of $\toba(\eny_+(q))$ are
\begin{gather}\label{eq:endymion-1}
x_1^2=0, \qquad x_{\tresdos}^2=0,\qquad x_1x_{\tresdos} =- x_{\tresdos}x_1,
\\
\label{eq:endymion-1b}
x_1x_2 = q_{12} x_2x_1,
\\
\label{eq:endymion-2}
x_{\tresdos2}^2=0,\qquad
 x_2x_{\tresdos2} = q_{21}x_{\tresdos2} x_2.
\end{gather}
A PBW-basis is
$\big\{ x_1^{m_1} x_{\tresdos}^{m_{\tresdos}} x_2^{m_2}x_{\tresdos2}^{n_{1}}\colon m_1, m_{\tresdos}, n_1 \in \{0,1\}, \, m_2 \in\N_0\big\}$.

\item
The relations of $\toba(\eny_-(q))$ are \eqref{eq:endymion-1}, \eqref{eq:endymion-1b} and
\begin{gather}
\label{eq:endymion-2b}
x_2 ^2=0,\qquad
x_2 x_{\tresdos2} = -q_{21} x_{\tresdos2} x_2.
\end{gather}
A PBW-basis is
$\big\{ x_1^{m_1} x_{\tresdos}^{m_{\tresdos}} x_2^{m_2}x_{\tresdos2}^{n_{1}}\colon m_1, m_{\tresdos}, m_2 \in \{0,1\}, \, n_{1} \in\N_0\big\}$.

\item
The relations of $\toba(\eny_{\star}(q))$ are \eqref{eq:endymion-1},
\begin{gather*}
x_2^2=0, \qquad x_{12}^2=0, \qquad x_{\tresdos12}^2=0,
\\
x_{\tresdos} \big[x_{\tresdos2},x_{12}\big]_c- q_{12}^2 \big[x_{\tresdos2},x_{12}\big]_c x_{\tresdos}
=q_{12} x_{12}x_{\tresdos12}.
\end{gather*}
\end{enumerate}

A PBW-basis consists of monomials
$x_{\tresdos}^{m_1}x_{\tresdos2}^{m_{\tresdos}} x_{\tresdos12}^{m_2} \big[x_{\tresdos2},x_{12}\big]_c^{n_{1}} x_1^{m_5}x_{12}^{m_6} x_2^{m_7}$,
where $m_{\tresdos}, n_{1} \in\N_0$ and $m_1, m_2, m_5, m_6, m_7 \in \{0,1\}$.
\end{Proposition}

\subsubsection{Assumptions} We fix for the rest of the paper the following setting.
\begin{hypothesis}\label{hypo:recargada}
$V = \oplus_{i \in \I_\theta} V_i\in \yd{\Bbbk\Gamma}$
that satisfies
\begin{itemize}\itemsep=0pt
\item $\dim V < \infty$, $\dim V_1 \geq \dim V_2 \geq \dots \geq \dim V_{\theta}$,

\item $V_i \in \yd{\Bbbk \Gamma}$ is indecomposable for $i\in \I_{\theta}$ and

\item Hypothesis~$\ref{hyp:pale}$, i.e.,
\begin{enumerate}[label=\rm{(\Roman*)}]
\item $V \in \pale$,
\item $\sup V$ generates $\Gamma$,
\item $V$ is connected.
\end{enumerate}
\end{itemize}
\end{hypothesis}

As remarked, $\theta \geq 2$.
Observe that recursive arguments need care with condition~\ref{item:hyp-suppgen}.
Since~$V_i$ is indecomposable, it is homogeneous of degree $g_i \in \Gamma$, and $g_i$ acts on $V_j$ with generalized eigenvalue $q_{ij}$
for any $i,j\in \I_{\theta}$.

\subsubsection{Terminology and graphical description}\label{subsubsec:terminology}
We attach a diagram to (some of) those $V$ as in Hypothesis~\ref{hypo:recargada}
extending the graphical description of \cite{AAH-triang}.
\begin{itemize}\itemsep=0pt
\item By~\ref{item:hyp-pale}, at least one $V_i$ is a pale block; we assume that the pale $V_i$'s are
$V_1, \dots, V_{s}$, $s \in \I_{\theta}$.
A pale block $V_i \subseteq V_{g_i}^{-1}$ of dimension $2$, respectively $n\geq 3$, is depicted by
$\underset{i}{\begin{tikzpicture}
\draw[dashed] (0,0) rectangle (0.36,0.36);
\end{tikzpicture}}$\,, respectively~$\underset{i}{\begin{tikzpicture}
\draw[dashed] (0,0) rectangle node[pos=.5] {$n$} (0.36,0.36);
\end{tikzpicture}}$.
These are the only pale blocks we need to consider, cf. Theorem~\ref{th:paleblock-point-resumen}.
\item
By assumption there exists $t \in \I_{\theta}$ such that the $V_i$'s of dimension 1 correspond to $i\in\I_{t+1, \theta}$;
these are called points and depicted as $\overset{q_{ii}} {\underset{i}{\bullet}}$.
\item
A block $\Vc(\epsilon,2)$ is depicted as $\boxplus$ if $\epsilon= 1$, respectively $\boxminus$ if $\epsilon = -1$;
no other blocks are considered, cf. Theorem~\ref{th:blocks}. They belong to the interval $\I_{s+1,t}$.
\item
When $ i\neq j \in \I_{t+1, \theta}$ and $q_{ij}q_{ji} \neq 1$, we draw an edge between
them decorated by $\widetilde{q}_{ij} := q_{ij}q_{ji}$, as in Section~\ref{subsec:diagonal-type}.
\item
Let $V_i$ be a pale block of dimension $2$ and let $V_j$ be a point. Then there
is a suitable basis $\big\{x_i, x_{\idos}\big\}$ of $V_i$ and $a_j\in \Bbbk$ such that
for $k,\ell\in \big\{i,\idos, j \big\}$
\begin{align*}
c(x_k\otimes x_{\ell}) &=
\begin{pmatrix}
- x_i \otimes x_i& - x_{\idos} \otimes x_i& q_{ij} x_j \otimes x_i
\\
- x_i \otimes x_{\idos} & - x_{\idos} \otimes x_{\idos} & q_{ij} x_j \otimes x_{\idos}
\\
q_{ji} x_i \otimes x_j & q_{ji}\big(x_{\idos} + a_jx_i\big) \otimes x_j & q_{jj} x_j \otimes x_j
\end{pmatrix}\!.
\end{align*}
If $\widetilde{q}_{ij}=1$ and $q_{jj} = \pm 1$, then
a dotted line labeled by $a_j$ is drawn between $i$ and $j$, i.e.,
$\xymatrix@C=20pt{ \underset{i}{\begin{tikzpicture}
\draw[dashed] (0,0) rectangle (0.36,0.36);
\end{tikzpicture}} \ar @{.}^{a_j}[r] & \overset{\pm1}{\underset{j}{\bullet}}}$\!.
Here $V_i \oplus V_j \simeq \eny_{\pm}(q)$ if $a_j \neq 0$.

If $\widetilde{q}_{ij}=-1$ and $q_{jj} = - 1$,
then we draw and edge labeled by $a_{j}$ between $i$ and $j$, i.e.,
$\xymatrix{ \underset{i}{\begin{tikzpicture}
\draw[dashed] (0,0) rectangle (0.36,0.36);
\end{tikzpicture}} \ar @{-}^{a_{j}}[r] & \overset{-1}{\underset{j}{\bullet}}}$\!.
Note that $V_i \oplus V_j \simeq \eny_{\star}(q)$ if $a_j \neq 0$.

\item Let $V_i$ be a pale block, $\dim V_i = 3$, and let $V_j$ be a point.
When $\widetilde{q}_{ij}=1$ and $q_{jj} = \pm 1$, respectively $q_{jj} = - 1=\widetilde{q}_{ij}$
we join $i$ and $j$ by a dotted line, respectively a line; i.e.,
$\xymatrix{ \underset{i}{\begin{tikzpicture}
\draw[dashed] (0,0) rectangle node[pos=.5] {$3$} (0.36,0.36);
\end{tikzpicture}} \ar @{.}[r] & \overset{\pm1}{\underset{j}{\bullet}}}$\!,
$\xymatrix{ \underset{i}{\begin{tikzpicture}
\draw[dashed] (0,0) rectangle node[pos=.5] {$3$} (0.36,0.36);
\end{tikzpicture}} \ar @{-}[r] & \overset{-1}{\underset{j}{\bullet}}}$\!.
\end{itemize}

The Nichols algebras $\toba(V)$ when $V$ has just one pale block and points (that is, $s=t=1$)
are informally called \emph{Endymion algebras}; and when $V$ has only pale blocks and blocks (that is, $t = \theta$),
they are called \emph{Selene algebras}.

\section{A point and a pale block of dimension 3} \label{sec:block+pt}

In this section, we assume Hypothesis~\ref{hypo:recargada} with $\theta = 2$, $\dim V_1= 3$ and $\dim V_2 = 1$.
For simplicity set $U = V_{1}$, $W = V_{2}$, $g = g_{1}$, $h = g_{2}$, $ q_{11}= \lambda_1$, $ q_{22}= \lambda_2$.
By Corollary~\ref{coro:indecomposable-abelian}, $U = U_g^{q_{11}}$ and $q_{11} = \pm 1$.
As $U$ is indecomposable and $\Gamma = \langle g,h \rangle$, $h$ must act as a Jordan block on $U$ with eigenvalue $q_{21} \in \Bbbk^{\times}$;
thus $g \neq h$ and $U = V_g$.
Fix a basis $\big\{x_1, x_2, x_{3}\big\}$ of $U$ such that $h_{\vert U}$ is given in this basis by the block
$\left(\begin{smallmatrix}
q_{21} & q_{21} & 0 \\
0 & q_{21} & q_{21} \\
0 & 0 & q_{21}
\end{smallmatrix}\right)$.
Let $\{x_{4}\}$ be a basis of $W$, so that $g\cdot x_{4} = q_{12}x_{4}$, $h\cdot x_{4} = q_{22}x_{4}$
where $q_{12}, q_{22} \in \Bbbk^{\times}$.
As usual $\widetilde{q}_{12} := q_{12}q_{21}$.

Let $q \in \Bbbk^{\times}$.
Let $\eny_{3,\pm}(q)$ denote the braided vector space $V$ as above with
\begin{align*}
q_{11}= -1, \qquad q_{22} = \pm 1, \qquad q_{12} = q= q_{21}^{-1}.
\end{align*}

In this section, we prove:

\begin{Theorem} \label{th:large-block-point} The Nichols algebra
$\toba (V)$ has finite $\GK$ if and only if $V\simeq \eny_{3,+}(q)$ or $\eny_{3,-}(q)$ for some $q \in \Bbbk^{\times}$.
\end{Theorem}

The proof of the Theorem goes as follows.
First, the Nichols algebras $\toba(\eny_{3,\pm}(q))$ have finite $\GK$ by Theorems~\ref{thm:q=1,q22=-1,n=3} and~\ref{thm:q=1,q22=1,n=3}.
Second, let $V$ be as above.
By Theorem~\ref{th:paleblock-point-resumen} applied to the subspace
$\langle x_1, x_2, x_{4}\rangle$, we have

\begin{Lemma}
If $\GK \toba (V) < \infty$, then $q_{11} = -1$ and either
\begin{enumerate}[label=$(\roman*)$]\itemsep=0pt
\item\label{item:endymion-rg4-i}
$\widetilde{q}_{12} = 1$ and $q_{22} \in \big\{1, -1 \big\}$, or else

\item\label{item:endymion-rg4-ii}
$\widetilde{q}_{12} = -1$ and $q_{22} = -1$.
\end{enumerate}
\end{Lemma}

To conclude the proof, we discard the possibility~\ref{item:endymion-rg4-ii}:

\begin{Proposition} 
If $\widetilde{q}_{12}=-1$, then $\GK \toba(V) = \infty$.
\end{Proposition}

\begin{proof} Let $\K^1 =\adc(\toba(W))(U)$. We shall prove that $\GK \toba\big(\K^1\big) = \infty$.
Set $x_0 := 0$ and $z_{i} =\adc(x_{4})(x_i) \in \K^1$, that is
\begin{align}\label{eq:zi-largeblock-point}
z_{i} = x_{4}x_i - q_{21}(x_i + x_{i-1})x_{4},\qquad i \in \I_{3}.
\end{align}
Let $(\partial_i)_{i \in \I_{4}}$ be the skew-derivations associated to the basis
$(x_i)_{i \in \I_{4}}$.
Since $q_{22} = -1$, we have $x_4^2 = 0$. Then
\begin{align*}
&\partial_i(z_{j}) =
\begin{cases}
2x_{4} & \mbox{if}\quad j = i, \\
x_{4} & \mbox{if}\quad j = i + 1, \\
0 & \text{otherwise,}\\
\end{cases}\qquad i \in \I_3.
\end{align*}
Thus $\{z_{1}, z_{2}, z_{3}\}$ is linearly independent.
Let $H = \toba(V)\# \Bbbk \Gamma$. Then
\begin{align*}
\Delta_{H}(z_{i}) = z_{i} \otimes 1 + 2x_{4} g \otimes x_i + x_{4} g \otimes x_{i-1} + gh \otimes z_{i},
\qquad i \in \I_3.
\end{align*}
Using $\delta =(\pi _{\toba (W)\# \Bbbk \Gamma}\otimes \id)\Delta _{H}$, we see that
\begin{align*}
\delta(z_{i}) = x_{4} g \otimes (2x_i + x_{i-1}) + gh \otimes z_{i},\qquad
 i \in \I_3.
\end{align*}
Hence for every $\mathfrak{y} \in \K^1$ and $i \in \I_3$
\begin{align*}
c(z_{i}\otimes \mathfrak{y}) = \adc(x_{4})(g \cdot \mathfrak{y}) \otimes (2x_i + x_{i-1})
+ (gh\cdot \mathfrak{y}) \otimes z_{i}.
\end{align*}
Let $Z$ be the braided subspace of $\K^1$ generated by $\big\{z_{1}, z_{2}, z_{3} \big\}$.
Then
\begin{align*}
(c(z_i \otimes z_j))_{i,j\in \I_3} =
\begin{pmatrix}
-z_{1} \otimes z_{1} & -(z_{2} + z_{1}) \otimes z_{1} & -(z_{3} + z_{2}) \otimes z_{1} \\
-z_{1} \otimes z_{2} & -(z_{2} + z_{1}) \otimes z_{2} & -(z_{3} + z_{2}) \otimes z_{2} \\
-z_{1} \otimes z_{3} & -(z_{2} + z_{1}) \otimes z_{3} & -(z_{3} + z_{2}) \otimes z_{3}
\end{pmatrix}\!.
\end{align*}
Hence $Z$ is isomorphic to $\mathcal{V}(-1, 3)$
and Theorem~\ref{th:blocks} applies.
\end{proof}

\subsection[The algebra B(E\_{3,-}(q))]{The algebra $\boldsymbol{\toba(\eny_{3,-}(q))}$}
To state our result, we need the elements
\begin{align*}
z_{i} = x_4x_i - q _{21} (x_i+x_{i-1}) x_4,\qquad
w = z_2x_3 + q_{21} (x_3+x_2) z_2,
\end{align*}
recall the notation \eqref{eq:zi-largeblock-point}.
By a direct computation, one has
\begin{align*}
\partial_i(z_{j}) = -\delta_{j,i+1}x_4,\qquad
\partial_1(w) =z_3, \qquad \partial_2(w) =-z_2,\qquad \partial_3(w) =\partial_4(w) =0.
\end{align*}

\begin{Theorem}\label{thm:q=1,q22=-1,n=3} The algebra
$\toba(\eny_{3,-}(q))$ is presented by generators $x_1$,
$x_2$, $x_3$, $x_4$ with defining relations
\begin{gather}\label{eq:endymion3-1-1}
x_i^2=0,\qquad x_ix_j = -x_jx_i, \qquad i\ne j\in\I_3,
\\ \label{eq:endymion3-1-2}
x_4^2=0,\qquad x_1x_4 = q_{12} x_4x_1,
\\ \label{eq:endymion3-1-3}
z_3z_2-z_2z_3+\tfrac{1}{2}z_2^2=0,
\\ \label{eq:endymion3-1-4}
z_2w+q_{21}wz_2=0.
\end{gather}
The monomials
\begin{align}\label{eq:monomials-eny3-}
x_1^{m_1} x_2^{m_2} x_3^{m_3} w^{p} z_2^{n_2}z_{3}^{n_3} x_4^{m_4}
m_i, \qquad p \in \{0,1\}, \quad n_j \in\N_0,
\end{align}
form a PBW-basis of $\toba(\eny_{3,-}(q))$. Hence $\GK \toba(\eny_{3,-}(q)) = 2$.
\end{Theorem}

\begin{proof}
Let $\toba$ be the algebra with the desired presentation. We claim that
there is a surjective map $\pi\colon \toba \to \toba(\eny_{3,-}(q))$.
Indeed,
the relations \eqref{eq:endymion3-1-1} and \eqref{eq:endymion3-1-2} hold
in $\toba(\eny_{3,-}(q))$ because the braiding of
$\langle x_1, x_2, x_3 \rangle$ is minus the flip and $\langle x_1, x_2, x_4 \rangle \simeq \eny_-(q)$ as braided vector spaces. We~check that \eqref{eq:endymion3-1-3} holds using skew-derivations: indeed $\partial_3$ and $\partial_4$ annihilate the left side
since they kill $z_2$ and $z_3$, while for $\partial_1$ and $\partial_2$ we use \eqref{eq:endymion3-1-other-rels-1}.
Similarly, \eqref{eq:endymion3-1-4} holds since $\partial_3$ and $\partial_4$ annihilate $z_2$ and $w$, while for $\partial_1$ and $\partial_2$ we use \eqref{eq:endymion3-1-other-rels-2}.

To prove that $\pi$ is surjective, we observe that if
$\wtoba$ is an algebra and $x_1,x_2, x_3, x_4 \in \wtoba$ satisfy\eqref{eq:endymion3-1-1} and \eqref{eq:endymion3-1-2},
then $x_1$ $q$-commutes with $z_2$, $z_3$ and $w$, and the following relations also hold:
\begin{equation}\label{eq:endymion3-1-other-rels-1}
\begin{split}
&z_2x_2 =-q_{21}(x_2+x_1)z_2,
\\
&z_3x_2 =-w-q_{21}(x_2+x_1)z_3, \qquad&& x_4z_2 =-q_{21}z_2x_4,
\\
&z_3x_3 =-q_{21}(x_3+x_2)z_3, && x_4z_3 =-q_{21}(z_3+z_2)x_4,
\\
&wx_2 = q_{21}(x_2+x_1)w, && wx_3 =q_{21} (x_3+x_2)w.
\end{split}
\end{equation}

If in addition, \eqref{eq:endymion3-1-3} holds in $\wtoba$, then the following holds:
\begin{align}\label{eq:endymion3-1-other-rels-2}
x_4w &= -q_{21}^2wx_4+\frac{q_{21}}{2}z_2^2.
\end{align}

Finally, if \eqref{eq:endymion3-1-3} and \eqref{eq:endymion3-1-4} hold in $\wtoba$, then the following also holds:
\begin{align}\label{eq:endymion3-1-other-rels-3}
z_3w =-q_{21}wz_3, \qquad w^2 =0.
\end{align}

From the defining relations, the definitions of $z_2$, $z_3$ and $w$,
\eqref{eq:endymion3-1-other-rels-1}, \eqref{eq:endymion3-1-other-rels-2} and \eqref{eq:endymion3-1-other-rels-3}
we see that the monomials \eqref{eq:monomials-eny3-}
generate $\toba$ and \emph{a fortiori} $\toba(\eny_{3,-}(q))$.
Next we prove that they are linearly independent. Suppose on the contrary that there exists a non-trivial linear combination~$\mathtt{S}$ of these elements: we may assume that $\mathtt{S}$ is homogeneous of minimal degree. As
\begin{align*}
\partial_4 \big( x_1^{m_1} x_2^{m_2} x_3^{m_3} w^{p} z_{3}^{n_3}z_2^{n_2} x_4^{m_4} \big) = \delta_{m_4,1} x_1^{m_1} x_2^{m_2} x_3^{m_3} w^{p} z_{3}^{n_3}z_2^{n_2},
\end{align*}
all the elements in $\mathtt{S}$ with non-zero coefficient have $m_{4}=0$ by the minimality of the degree. Analogously, $n_2=n_3=p=0$ since
\begin{gather*}
\partial_4\partial_1 \big( x_1^{m_1} x_2^{m_2} x_3^{m_3} w^{p} z_{3}^{n_3}z_2^{n_2}\big) = -n_2 x_1^{m_1} x_2^{m_2} x_3^{m_3} w^{p} z_{3}^{n_3}z_2^{n_2-1},
\\[1mm]
(\partial_4\partial_1)^{n_3-1} \partial_4\partial_2 \big( x_1^{m_1} x_2^{m_2} x_3^{m_3} w^{p} z_{3}^{n_3}\big) = (-1)^{n_3}n_3! x_1^{m_1} x_2^{m_2} x_3^{m_3} w^{p},
\\[1mm]
\partial_4\partial_1\partial_2 \big( x_1^{m_1} x_2^{m_2} x_3^{m_3} w^{p}\big) = \delta_{p,1} x_1^{m_1} x_2^{m_2} x_3^{m_3}.
\end{gather*}
Hence $\mathtt{S}$ is a non-trivial linear combination of $x_1^{m_1} x_2^{m_2} x_3^{m_3}$, $m_i\in\{0,1\}$, and we get a contradiction. Thus the monomials \eqref{eq:monomials-eny3-} are linearly independent in $\toba(\eny_{3,-}(q))$
so they form a basis of $\toba(\eny_{3,-}(q))$; hence $\toba \simeq \toba(\eny_{3,-}(q))$.
\end{proof}

\subsection[The algebra B(E\_\{3,+\}(q))]{The algebra $\boldsymbol{\toba(\eny_{3,+}(q))}$}
We need the elements
\begin{gather*}
x_{4j}=(\adc x_4)x_j, \qquad j=2,3,
\\
x_{443} =(\adc x_4)^2x_3,
\\
\vt =[x_{42}, x_3]_c=x_{42}x_3+q_{21}(x_3+x_2)x_{42},
\\
\ut=[x_{43},x_{42}]_c=x_{43}x_{42}+x_{42}x_{43},
\\
\wt=[x_{43},\vt]_c=x_{43}\vt-q_{21}\vt x_{43}.
\end{gather*}
Observe that $\partial_3(\vt)=\partial_3(\ut)=\partial_3(\wt)=0$,
\begin{alignat}{3}
\label{eq:block-pt-1-derivations}
&\partial_1(x_{42})=-x_4,\qquad &&\partial_2(x_{43}) =-x_4, &&\partial_1(x_{443}) =x_4^2,
\\
&\partial_1(\vt)=x_{43}, &&\partial_1(\ut)=q_{12}x_{443}+x_{42}x_4,\qquad &&\partial_1(\wt)=2x_{43}^2,
\\
&\partial_2(\vt)=-x_{42}, &&\partial_2(\ut)=0, &&\partial_2(\wt)=-2u.
\end{alignat}
and all the other skew-derivations annihilate $x_{42}$, $x_{43}$, $x_{443}$.

\begin{Theorem}\label{thm:q=1,q22=1,n=3} The algebra
$\toba(\eny_{3,+}(q))$
is presented by generators $x_1$, $x_2$, $x_3$, $x_4$ with defining relations
\begin{gather}\label{eq:block-pt-1-rels-0}
x_ix_j=-x_jx_i,\qquad x_i^2 = 0, \qquad i\ne j \in\I_3,
\\
\label{eq:block-pt-1-rels-1}
x_4x_1=q_{21}x_1x_4, \qquad x_4x_{42} = q_{21}x_{42}x_4,\qquad x_4x_{443} = q_{21}x_{443}x_4,
\\
\label{eq:block-pt-1-rels-2}
x_{443}x_{42}+q_{21}x_{42}x_{443} =0,
\\
\label{eq:block-pt-1-rels-3}
x_{443}x_{43}+q_{21}(x_{43}+2x_{42})x_{443} =0,
\\
\label{eq:block-pt-1-rels-4}
x_{4}\wt-q_{21}^3\wt x_{4} +2q_{21}^2 x_{42}\ut= 0,
\\
\label{eq:block-pt-1-rels-4prim}
x_{43}\ut-\ut x_{43}+x_{42}\ut = 0,
\\
\label{eq:block-pt-1-rels-5}
x_{42}\wt+q_{21}\wt x_{42} = 0,
\\
\label{eq:block-pt-1-rels-6}
x_{43}\wt+q_{21}\wt x_{43} = 0.
\end{gather}
The monomials
\begin{align}\label{eq:monomials-eny3+}
x_1^{m_1} x_2^{m_2} x_3^{m_3} \vt^{p_1} x_{42}^{p_2} \wt^{p_3} \ut^{p_4} x_{43}^{p_5} x_{443}^{p_6} x_4^{p_7}, \qquad m_i, p_2, p_3, p_6 \in \{0,1\}, \quad p_1, p_4, p_5, p_7 \in\N_0,
\end{align}
form a PBW-basis of $\toba(\eny_{3,+}(q))$.
Hence $\GK \toba(\eny_{3,+}(q)) =4$.
\end{Theorem}

\begin{proof}
As before, let $\K^1 =\adc(\toba(W))(U)$.
Set
\begin{align*}
\zc{i}{j} := (\adc x_{4})^j x_i,\qquad i\in\I_3, \quad j\in\N_0;
\end{align*}
clearly, $\K^1$ is spanned by the $\zc{i}{j}$ with $i\in\I_3$, $ j\in\N_0$. Observe that
\begin{align*}
g\cdot \zc{i}{j} = -q_{12}^j\zc{i}{j}, \qquad
h\cdot \zc{i}{j} = q_{21}(\zc{i}{j}+\zc{i-1}{j}).
\end{align*}

\begin{step}
The set $\Zc := \{\zc{i}{j}\colon i\in\I_3, \, j\in\I_{0,i-1}\}$ is a basis of $\K^1$.
\end{step}

\begin{proof}[Proof of Step 1]
We prove by induction on $j$ that
\begin{align}\label{eq:derivadak-zij}
\partial_k(\zc{i}{j}) = \delta_{k,i-j} (-1)^j x_{4}^{j}, \qquad i,k\in\I_3, \quad j\in\N_0.
\end{align}
If $j=0$, then $\zc{i}{0}=x_i$ and the claim follows. Next if \eqref{eq:derivadak-zij} holds for $j$, then
\begin{align*}
\partial_k(\zc{i}{j+1})&=\partial_k(x_{4}\zc{i}{j}-q_{21}(\zc{i}{j}+\zc{i-1}{j})x_{4})
=x_{4}\partial_k(\zc{i}{j})-q_{21}\partial_k(\zc{i}{j}+\zc{i-1}{j})g\cdot x_{4}.
\end{align*}
If $k\ne i-j,i-j-1$, then $\partial_k(\zc{i}{j+1})=0$ by inductive hypothesis. Also,
\begin{align*}
&\partial_{i-j}(\zc{i}{j+1})=x_{4}\partial_{i-j}(\zc{i}{j})-\partial_{i-j}(\zc{i}{j})x_{4}=0,
\\
&\partial_{i-j-1}(\zc{i}{j+1})=-\partial_{i-j-1}(\zc{i-1}{j})x_{4}=-(-1)^j x_{4}^{j}x_{4}
=(-1)^{j+1} x_{4}^{j+1}.
\end{align*}

Also, $\partial_{4}(\zc{i}{j})=0$ for all $i\in\I_3$, $j\in\N_0$. Therefore, $\partial_k(\zc{i}{i})=0$ for all $k\in\I_{4}$, so $\zc{i}{i}=0$. Then $\zc{i}{j}=0$ for all $j\ge i$ and $\K^1$ is spanned by $\Zc$.
It remains to prove that $\Zc$ is linearly independent. As $\zc{i}{j}$ has degree $j+1$ in $\toba(V)$, it suffices to prove that $\{\zc{i}{j}\colon j<i\le 3\}$ is linearly independent for $j\in\I_{0,2}$. This follows from \eqref{eq:derivadak-zij} and the fact that $x_{4}^k\ne 0$ for all $k\in\N_0$.
\end{proof}

\begin{step}
The coaction on $\K^1$ satisfies
\begin{align*}
\delta(\zc{i}{j}) = \sum_{t=0}^{j}(-1)^t\binom{j}{t} x_{4}^t h^{j-t}g \otimes \zc{i-t}{j-t},\qquad
 i\in\I_3, \quad j\in\I_{0,i-1}.
\end{align*}
\end{step}

\begin{proof}[Proof of Step 2]
We proceed inductively.
If $j=0$, then $ \delta(\zc{i}{0})=\delta(x_i)=g\otimes x_i =g\otimes \zc{i}{0}$.
Assume that~\eqref{eq:derivadak-zij} holds for $j$. Then
\begin{align*}
\delta ( \zc{i}{j+1}) &= (\pi _{\toba (W)\# \Bbbk \Gamma}\otimes \id)\Delta _{H} \big(x_{4} \zc{i}{j}-q_{21}(\zc{i}{j}+\zc{i-1}{j})x_{4}\big)
\\[1mm]
&= (x_{4}\otimes 1+h\otimes x_{4})\delta(\zc{i}{j})-q_{21}\delta(\zc{i}{j}+\zc{i-1}{j})(x_{4}\otimes 1+h\otimes x_{4})
\\
&=\sum_{t=0}^{j}(-1)^t\binom{j}{t} x_{4}^{t+1} h^{j-t}g \otimes \zc{i-t}{j-t}
+ x_{4}^t h^{j+1-t}g \otimes x_{4}\zc{i-t}{j-t}
\\
&\phantom{=} -q_{21} \sum_{t=0}^{j}(-1)^t\binom{j}{t} x_{4}^{t} h^{j-t}gx_{4} \otimes (\zc{i-t}{j-t}+\zc{i-t-1}{j-t})
\\
&\phantom{=} -q_{21} \sum_{t=0}^{j}(-1)^t\binom{j}{t}x_{4}^t h^{j+1-t}g \otimes (\zc{i-t}{j-t}+\zc{i-t-1}{j-t})x_{4}
\\
&=-\sum_{t=0}^{j}(-1)^t\binom{j}{t} (x_{4}^{t+1} h^{j-t}g \otimes \zc{i-t-1}{j-t}
- x_{4}^t h^{j+1-t}g \otimes \zc{i-t}{j-t+1})
\\
&= (-1)^{j+1} x_{4}^{j+1} g \otimes \zc{i-j-1}{0}
+\sum_{t=1}^{j}(-1)^t\binom{j+1}{t} x_{4}^t h^{j+1-t}g \otimes \zc{i-t}{j+1-t}
\\
&\phantom{=} +h^{j+1}g \otimes \zc{i}{j+1},
\end{align*}
and the inductive step follows.
\end{proof}

\begin{step}\label{lem:block-pt-1-rels}
If $\wtoba$
is an algebra and $x_i\in\wtoba$, $i\in\I_4$, satisfy \eqref{eq:block-pt-1-rels-0} and \eqref{eq:block-pt-1-rels-1},
then
\begin{align}\label{eq:block-pt-1-rels-others-1}
& x_{4j}x_j=-q_{21}(x_j+x_{j-1})x_{4j},
\\\label{eq:block-pt-1-rels-others-2}
& x_{43}x_2=-\vt-q_{21}(x_2+x_1)x_{43},
\\ \label{eq:block-pt-1-rels-others-prv1}
& x_{42}^2=0,
\\\label{eq:block-pt-1-rels-others-3}
& x_4\vt=q_{21}^2\vt x_4+q_{21}\ut,
\\\label{eq:block-pt-1-rels-others-4}
& \vt x_j=q_{21}(x_j+x_{j-1})\vt,
\\\label{eq:block-pt-1-rels-others-5}
& x_{443}x_3 = -q_{21}^2(x_3+2x_2+x_1)x_{443}-2q_{21}x_{43}^2-2q_{21}x_{42}x_{43},
\\\label{eq:block-pt-1-rels-others-6}
& x_{443}x_2 = -q_{21}^2(x_2+2x_1)x_{443}-2q_{21}\ut,
\\\label{eq:block-pt-1-rels-others-7}
& x_{42}\vt = q_{21}\vt x_{42},
\\\label{eq:block-pt-1-rels-others-8}
& \ut x_2 =q_{21}^2(x_2+2x_1)\ut,
\\\label{eq:block-pt-1-rels-others-9}
&\ut x_3 = \wt+q_{21}\vt x_{42}+q_{21}^2(x_3+2x_2+x_1)\ut,
\\\label{eq:block-pt-1-rels-others-10}
& \ut x_{42} = x_{42}\ut.
\end{align}
\end{step}

\begin{proof}[Proof of Step 3]
Argue recursively on the degree of the relations.
\end{proof}

Let $\toba$ be the algebra with the desired presentation.

\begin{step}
There is a surjective map $\pi\colon \toba \to \toba(\eny_{3,+}(q))$.
\end{step}

\begin{proof}[Proof of Step 4]
Arguing as in the proof of Theorem~\ref{thm:q=1,q22=-1,n=3},
we see that the relations \eqref{eq:block-pt-1-rels-0} and \eqref{eq:block-pt-1-rels-1} hold in $\toba(\eny_{3,+}(q))$.
Using \eqref{eq:block-pt-1-derivations} and \eqref{eq:block-pt-1-rels-1}, we compute
\begin{gather*}
\partial_1 (x_{443}x_{42}+q_{21}x_{42}x_{443}) = -x_{443}x_{4} -q_{12} x_{4}^2 x_{42} + q_{21}x_{42}x_{4}^2 + q_{12}x_{4}x_{443}=0,
\\
\partial_1 (x_{443}x_{43}+q_{21}(x_{43}+2x_{42})x_{443}) =
-q_{12} x_{4}^2x_{43} + 2q_{12}x_{4}x_{443}
 + q_{21}(x_{43}+2x_{42})x_{4}^2
\\ \qquad
{}= -2x_{443}x_{4} -q_{21}(x_{43}+2x_{42})x_4^2
+ 2x_{443}x_{4}+ q_{21}(x_{43}+2x_{42})x_{4}^2 =0,
\\
\partial_2 (x_{443}x_{43}+q_{21}(x_{43}+2x_{42})x_{443}) = -x_{443}x_{4}-q_{12}x_{4}x_{443}=0.
\end{gather*}
Since $\partial_i(x_{42})=\partial_i(x_{443})=0$ for $i\in\I_{2,4}$, we conclude that \eqref{eq:block-pt-1-rels-2} holds in $\toba(\eny_{3,+}(q))$.
Similarly, $\partial_i(x_{43})=\partial_i(x_{42})=\partial_i(x_{443})=0$ for $i\in\I_{3,4}$, and \eqref{eq:block-pt-1-rels-3} holds
in $\toba(\eny_{3,+}(q))$.
For the remaining relations, we first check that
\begin{align}
\label{eq:block-pt-1-rels-others-11}
& x_4 \ut=q_{21}^2\ut x_4,
\\ \label{eq:block-pt-1-rels-others-prv2}
& x_{443}^2=0,
\\ \label{eq:block-pt-1-rels-others-15}
& x_{443}\ut = 2q_{21}^2\ut x_{443}.
\end{align}
Indeed for \eqref{eq:block-pt-1-rels-others-prv2} we use \eqref{eq:block-pt-1-rels-others-1}, \eqref{eq:block-pt-1-rels-1}, \eqref{eq:block-pt-1-rels-2} and \eqref{eq:block-pt-1-rels-3}:
\begin{align*}
x_{443}^2 = x_{443}(x_4x_{43}-q_{21}(x_{43}+x_{42})x_4)
=-x_{443}^2-2(x_{4}x_{42}-q_{21}x_{42}x_4)x_{443}=-x_{443}^2,
\end{align*}
so $x_{443}^2 = 0$. Now, by \eqref{eq:block-pt-1-rels-others-6},
$2\ut=-q_{12}x_{443}x_2-q_{21}(x_2+2x_1)x_{443}$, hence
\begin{align*}
2x_{443}\ut &= -q_{12}x_{443}^2x_2-q_{21}x_{443}(x_2+2x_1)x_{443}=-q_{21}x_{443}(x_2+2x_1)x_{443}
\\
& = -q_{21}(-2q_{21}\ut -q_{21}^2(x_2+4x_1)x_{443})x_{443} =2q_{21}^2\ut x_{443},
\end{align*}
and \eqref{eq:block-pt-1-rels-others-15} follows.
For \eqref{eq:block-pt-1-rels-others-11}, we use \eqref{eq:block-pt-1-rels-1}, \eqref{eq:block-pt-1-rels-others-prv1} and \eqref{eq:block-pt-1-rels-2}:
\begin{align*}
x_4\ut &= (x_{443}+q_{21}(x_{43}+x_{42})x_4)x_{42}+q_{21}x_{42}(x_{443}+q_{21}(x_{43}+x_{42})x_4)
\\
&= q_{21}^2(x_{43}+x_{42})x_{42}x_4+q_{21}^2x_{42}(x_{43}+x_{42})x_4 = q_{21}^2 \ut x_4.
\end{align*}

Next we evaluate appropriately the skew-derivations:
\begin{gather*}
\begin{split}
&\partial_1(x_{43}\ut-\ut x_{43}) = x_{43}(q_{12}x_{443}+x_{42}x_4)+q_{12}(q_{12}x_{443}+x_{42}x_4)x_{43}
\\ &\hphantom{\partial_1(x_{43}\ut-\ut x_{43})}
{}= q_{12}x_{43}x_{443}+(u-x_{42}x_{43})x_4-q_{12}(x_{43}+2x_{42})x_{443}
\\&\hphantom{\partial_1(x_{43}\ut-\ut x_{43})=}
{} +q_{12}x_{42}(x_{443}+q_{21}(x_{43}+x_{42})x_4)=ux_4-q_{12}x_{42}x_{443},
\end{split}
\\
\partial_1 (x_4\wt-q_{21}^3\wt x_4)= 2(x_{443}+q_{21}(x_{43}+x_{42})x_4)x_{43}-2q_{21}^3q_{12}x_{43}^2x_4
\\ \hphantom{\partial_1 (x_4\wt-q_{21}^3\wt x_4)}
{}= -2q_{21}(x_{43}+2x_{42})x_{443}+2q_{21}(x_{43}+x_{42})x_{443}
+2q_{21}^2(x_{43}+x_{42})^2x_4
\\ \hphantom{\partial_1 (x_4\wt-q_{21}^3\wt x_4)=}
{}-2q_{21}^3q_{12}x_{43}^2x_4 = 2q_{21}^2 \ut x_4-2q_{21}x_{42}x_{443},
\\
\partial_1 (x_{42}\ut) =-q_{12}^2 x_4\ut+x_{42}(q_{12}x_{443}+x_{42}x_4)=-\ut x_4+q_{12}x_{42}x_{443},
\\
\partial_2 \big(x_4\wt-q_{21}^3\wt x_4\big) = 2x_4\ut-q_{21}^3q_{12}\ut x_4=0,
\\
\partial_2 (x_{43}\ut-\ut x_{43})= -q_{12}^2x_4\ut+\ut x_4 = 0 = \partial_2(x_{42}\ut).
\end{gather*}
Then \eqref{eq:block-pt-1-rels-4} and \eqref{eq:block-pt-1-rels-4prim} hold since $\partial_3$ and $\partial_4$ annihilate both sides. Now
\begin{gather*}
\partial_1 (x_{42}\wt+q_{21}\wt x_{42}) = q_{12}^2x_{4}\wt - q_{21}\wt x_{4} + 2x_{42}x_{43}^2 - 2q_{21}q_{12}x_{43}^2 x_{42}
\\ \hphantom{\partial_1 (x_{42}\wt+q_{21}\wt x_{42})}
{} = q_{12}^2\big(q_{21}^3\wt x_{4} + 2q_{21}^2\yt\big) - q_{21}\wt x_{4} + 2x_{42}x_{43}^2 - 2x_{43}(\ut-x_{42}x_{43})
\\ \hphantom{\partial_1 (x_{42}\wt+q_{21}\wt x_{42})}
 {}= 2\yt + 2x_{42}x_{43}^2 - 2(\yt+\ut x_{43}) + 2(\ut-x_{42}x_{43})x_{43}=0,
\\
\partial_2 (x_{42}\wt+q_{21}\wt x_{42}) = -2x_{42}\ut + 2q_{12}q_{21}\ut x_{42}=0,
\\
\partial_1 (x_{43}\wt+q_{21}\wt x_{43}) =2x_{43}^2-2q_{12}q_{21}x_{43}^2=0,
\\
\partial_2 (x_{43}\wt+q_{21}\wt x_{43}) = q_{12}^2 x_4\wt-2x_{43}\ut+2q_{21}q_{12}\ut x_{43}-q_{21}\wt x_4=0,
\end{gather*}
so \eqref{eq:block-pt-1-rels-5} and \eqref{eq:block-pt-1-rels-6} also hold.
\end{proof}

To prove that $\pi$ is bijective and that \eqref{eq:monomials-eny3+} is a basis we need the following.

\begin{step}\label{lem:block3-pt-1-relations-2}
The following relations hold in $\toba$:
\begin{alignat*}{2}
&\wt x_2 = -q_{21}^2(x_2+2x_1)\wt, &&\wt x_3 = -q_{21}^2(x_3+2x_2+x_1)\wt-q_{21}\vt^2,
\\
&x_{443}\vt = q_{21}^3\vt x_{443}-2q_{21}^2x_{42}\ut,\qquad &&\ut\vt = q_{21}^2\vt\ut,
\\
&\ut \wt = q_{21}^2 \wt \ut, &&\wt \vt = q_{21} \vt \wt,
\\
&x_{443} \wt = q_{21}^4 \wt x_{443}, &&\wt^2=0.
\end{alignat*}
\end{step}

\begin{proof}[Proof of Step 5]
Use Step~\ref{lem:block-pt-1-rels} and proceed recursively on the degree.
\end{proof}

We now finish the proof of Theorem~\ref{thm:q=1,q22=1,n=3}.
By the defining relations and those in Steps~\ref{lem:block-pt-1-rels} and~\ref{lem:block3-pt-1-relations-2}, we see that the monomials \eqref{eq:monomials-eny3+} generate $\toba$ and \emph{a fortiori} $\toba(\eny_{3,+}(q))$.
Next we prove that the monomials \eqref{eq:monomials-eny3+} are linearly independent in $\toba(\eny_{3,+}(q))$.
By direct computations,
\begin{gather*}
\partial_4 \big(x_1^{m_1} x_2^{m_2} x_3^{m_3} \vt^{p_1} \wt^{p_2} x_{43}^{p_3} \ut^{p_4} x_{42}^{p_5} x_{443}^{p_6} x_4^{p_7}\big)
 = p_7 x_1^{m_1} x_2^{m_2} x_3^{m_3} \vt^{p_1} \wt^{p_2} x_{43}^{p_3} \ut^{p_4} x_{42}^{p_5} x_{443}^{p_6} x_4^{p_7-1},
\\[.5mm]
\partial_4^2 \partial_1 \big(x_1^{m_1} x_2^{m_2} x_3^{m_3} \vt^{p_1} \wt^{p_2} x_{43}^{p_3} \ut^{p_4} x_{42}^{p_5} x_{443}^{p_6}\big)
= 2p_6 x_1^{m_1} x_2^{m_2} x_3^{m_3} \vt^{p_1} \wt^{p_2} x_{43}^{p_3} \ut^{p_4} x_{42}^{p_5},
\\[.5mm]
\partial_4 \partial_1 \big(x_1^{m_1} x_2^{m_2} x_3^{m_3} \vt^{p_1} \wt^{p_2} x_{43}^{p_3} \ut^{p_4} x_{42} \big) =
-x_1^{m_1} x_2^{m_2} x_3^{m_3} \vt^{p_1} \wt^{p_2} x_{43}^{p_3} \ut^{p_4},
\\[.5mm]
\partial_4^2 \partial_1^2 \big( x_1^{m_1} x_2^{m_2} x_3^{m_3} \vt^{p_1} \wt^{p_2} x_{43}^{p_3} \ut^{p_4} \big) = 4p_4 x_1^{m_1} x_2^{m_2} x_3^{m_3} \vt^{p_1} \wt^{p_2} x_{43}^{p_3} \ut^{p_4-1},
\\[.5mm]
(\partial_4 \partial_1)^{2k} \partial_4 \partial_2 \big( x_1^{m_1} x_2^{m_2} x_3^{m_3} \vt^{p_1} \wt^{p_2} x_{43}^{2k+1}\big) = (-1)^k k!(k+1)! x_1^{m_1} x_2^{m_2} x_3^{m_3} \vt^{p_1} \wt^{p_2},
\\[.5mm]
(\partial_4 \partial_1)^{2k-1} \partial_4 \partial_2 \big( x_1^{m_1} x_2^{m_2} x_3^{m_3} \vt^{p_1} \wt^{p_2} x_{43}^{2k}\big) =
(-1)^k (k!)^2 x_1^{m_1} x_2^{m_2} x_3^{m_3} \vt^{p_1} \wt^{p_2},
\\[.5mm]
\partial_4^2 \partial_1^2 \partial_2 \big( x_1^{m_1} x_2^{m_2} x_3^{m_3} \vt^{p_1} \wt^{p_2}\big) =
-8q_{12}p_2 x_1^{m_1} x_2^{m_2} x_3^{m_3} \vt^{p_1},
\\[.5mm]
\partial_4 \partial_1 \partial_2 \big( x_1^{m_1} x_2^{m_2} x_3^{m_3} \vt^{p_1}\big) =
p_1 x_1^{m_1} x_2^{m_2} x_3^{m_3}.
\end{gather*}
The claim is established by a recursive argument as for $\toba(\eny_{3,-}(q))$. Thus \eqref{eq:monomials-eny3+} is a basis of $\toba(\eny_{3,+}(q))$ and $\toba \simeq \toba(\eny_{3,+}(q))$.
\end{proof}

\section{Two points and a pale block of dimension 2}\label{sec:paleblock+2pts}
\subsection{Notations and the main result}
In this section, we assume Hypothesis~\ref{hypo:recargada} with $\theta = 3$,
$\dim V_1= 2$ and $\dim V_2 = \dim V_3 = 1$.
Let $g_i\in \Gamma$ be such that $V_i$ is homogeneous of degree $g_i$, for $i\in \I_3$.
Let $\big\{x_1, x_{\tresdos}\big\}$ be a basis of $V_1$ and let $\{x_i\}$ be a basis of $V_i$, $i = 2,3$.
Then
\begin{itemize}\itemsep=0pt
\item If $i\in \I_3$ and $j= 2,3$, then there exists $q_{ij} \in \Bbbk^{\times}$ such that $g_i \cdot x_j = q_{ij} x_j$.
\item Since $V \in \pale$ and $V_1$ is indecomposable, $g_1$ acts on $V_1$ by $q_{11} \id$, $q_{11} \in \Bbbk^{\times}$.
\item Since $V_1$ is indecomposable, there exists $j \in \{2,3\}$ such that $g_j$ acts on $V_1$ by a Jordan block. We assume that $j = 2$
and that $g_2$ acts in the basis $\big\{x_1, x_{\tresdos}\big\}$ by
$\left(\begin{smallmatrix} q_{21} & q_{21} \\ 0 & q_{21}\end{smallmatrix}\right)$, $q_{21} \in \Bbbk^{\times}$.
Set $a_2 := 1$.
\item Since the action of $g_3$ on $V_1$ commutes with that of $g_2$,
it is given in the basis $\big\{x_1, x_{\tresdos}\big\}$ by
$\left(\begin{smallmatrix} q_{31} & q_{31}a \\ 0 & q_{31}\end{smallmatrix}\right)$, for some $q_{31} \in \Bbbk^{\times}$, $a \in \Bbbk$.
Set $a_3 := a$.
\end{itemize}
Thus the braiding of $V$ is determined by the matrix $\bq = (q_{ij})_{i,j \in \I_{3}}$ with entries in $\Bbbk^{\times}$
and the scalar $a$.
Explicitly, the braiding is
\begin{align*}
&c(x_k\otimes x_\ell) = q_{k\ell} x_{\ell} \otimes x_k,
\\
&(c(x_i \otimes x_j))_{i,j\in \{1,\tresdos, k \}} =
\begin{pmatrix}
q_{11} x_1 \otimes x_1& q_{11} x_{\tresdos} \otimes x_1& q_{1k} x_k \otimes x_1
\\
q_{11} x_1 \otimes x_{\tresdos} & q_{11} x_{\tresdos} \otimes x_{\tresdos} & q_{1k} x_k \otimes x_{\tresdos}
\\
q_{k1} x_1 \otimes x_k & q_{k1}(x_{\tresdos} + a_kx_1) \otimes x_k & q_{kk} x_k \otimes x_k
\end{pmatrix}\!,
\end{align*}
$k,\ell = 2,3$.
We give a notation in some special cases. Fix $\bqs = (q_{12}, q_{13}, q_{23})$ such that
$q_{ij} \in \Bbbk^{\times}$ for all $i<j $ and $a \in \Bbbk$.
We have the braided vector spaces
\begin{itemize}\itemsep=0pt
\item $\eny_{\mu, \nu}\big(\bqs, a\big)$, $\mu, \nu \in \{\pm\}$, where $a \neq 0$ and $\bq$ is determined by
\begin{align*}
q_{11} =-1,\qquad q_{ij}q_{ji} = 1, \quad i<j \in \I_3,\qquad q_{22} =\mu 1,\qquad q_{33} =\nu 1.
\end{align*}

\item$\enyIIb{\infty}\big(\bqs\big)$, where $ a= 0$ and $\bq$ is determined by
\begin{align}\label{eq:bq-IIa-b}
q_{11}= q_{33}=-1, \qquad
q_{3 1} = -q_{13}^{-1},\qquad
q_{22} = 1, \qquad q_{21} = q_{12}^{-1},\qquad
q_{3 2} = q_{23}^{-1}.
\end{align}

\end{itemize}
The diagrams of $\eny_{\mu, \nu}\big(\bqs, a\big)$ and $\enyIIb{\infty}\big(\bqs\big)$ are respectively
$$
\xymatrix{
\overset{\nu}{\underset{3}{\bullet}} \ar@{.}^{a}[r]
&\underset{1}{\begin{tikzpicture}
\draw[dashed] (0,0) rectangle (0.36,0.36);
\end{tikzpicture}} \ar @{.}[r]
& \overset{\mu}{\underset{2}{\bullet}}\,,}\qquad
\xymatrix{
\overset{-1}{\underset{3}{\bullet}} \ar@{-}^{0}[r]
&\underset{1}{\begin{tikzpicture}
\draw[dashed] (0,0) rectangle (0.36,0.36);
\end{tikzpicture}} \ar @{.}[r]
& \overset{1}{\underset{2}{\bullet}}\,.}
$$

In this section, we prove:

\begin{Theorem} \label{th:2-block-points}
The Nichols algebra
$\toba (V)$ has finite $\GK$ if and only if there exists $\bqs = (q_{12}, q_{13}, q_{23}) \in (\Bbbk^{\times})^3$ and $a\in \Bbbk^{\times}$ such that
$V\simeq \eny_{\mu, \nu}\big(\bqs, a\big)$ or $\enyIIb{\infty}\big(\bqs\big)$.
\end{Theorem}

Here is the scheme of the proof of Theorem~\ref{th:2-block-points}.\
We show in Theorems~\ref{th:endymion-rank4-1} and~\ref{th:endymion-2pts-case-a}
that $\toba \big(\eny_{\mu, \nu}\big(\bqs, a\big)\big)$ and
$\toba\big(\enyIIb{\infty}\big(\bqs\big)\big)$ have finite $\GK$.

Assume that $\GK \toba(V) < \infty$.
By Theorem~\ref{th:paleblock-point-resumen} applied to $V_1 \oplus V_2$, $q_{11} = -1$ and
either $\widetilde{q}_{12} = 1$ and $q_{22} = \pm 1$; or $q_{22} = -1 = \widetilde{q}_{12}$.
If $a\neq 0$, then by Theorem~\ref{th:paleblock-point-resumen} applied to $V_1 \oplus V_3$, either
$\widetilde{q}_{13} = 1$ and $q_{33} = \pm 1$; or $q_{33} = -1 = \widetilde{q}_{13}$;
but $\widetilde{q}_{13}$ could be $\neq \pm 1$ if $a = 0$.
We consider four cases:
\begin{enumerate}[label=(\rm{\Roman*)}]
\item\label{case:I} $\widetilde{q}_{12}= \widetilde{q}_{13}=1$;

\item\label{case:II} $\widetilde{q}_{12}=1$, $\widetilde{q}_{13} \neq 1$;

\item\label{case:III} $\widetilde{q}_{12}=-1$, $\widetilde{q}_{13}=1$;

\item\label{case:IV} $\widetilde{q}_{12}=-1$, $\widetilde{q}_{13} \neq 1$.
\end{enumerate}

In case~\ref{case:I}, we distinguish two subcases:{\samepage
\begin{enumerate}[label=$(\alph*)$]\itemsep=0pt
\item\label{case:Ia} $\widetilde{q}_{23}=1$, dealt with by Theorem~\ref{th:endymion-rank4-1},

\item\label{case:Ib} $\widetilde{q}_{23} \neq 1$; here $\GK \toba(V) = \infty$ by Proposition~\ref{prop:pale-2pts-Ib}.
\end{enumerate}}

In case~\ref{case:II}, by Lemma~\ref{lemma:caseII-general} we are reduced to
$q_{22} = \widetilde{q}_{23} = 1$, $q_{33} = \widetilde{q}_{13} = -1$ and
either $a = 0$ or
$a \neq 0$, dealt with by Theorem~\ref{th:endymion-2pts-case-a} and Proposition~\ref{prop:endymion-two-pts-case-b},
respectively.

Finally, in cases~\ref{case:III} and~\ref{case:IV}, $\GK \toba(V) = \infty$, or $V$ belongs to case~\ref{case:II} after reindexing,
by Lemma~\ref{lemma:caseIII-general} and
Propositions~\ref{prop:caseIII-1} and~\ref{prop:caseIV}.

\subsection[Case~(I)]{Case~\ref{case:I}}
In this subsection, we assume that $\widetilde{q}_{12}= \widetilde{q}_{13}=1$.

\subsubsection[Case (I)(a): widetilde\{q\}\_\{23\} = 1]{Case~\ref{case:I}\ref{case:Ia}: $\boldsymbol{\widetilde{q}_{23} = 1}$}

Here $a\neq 0$ because of Hypothesis~\ref{hyp:pale}\ref{item:hyp-connected}, or the vertex 3 would be disconnected.
Thus $q_{22} = \pm 1$, $q_{33} = \pm 1$. All four posibilities give rise to Nichols algebras with finite $\GK$.
For convenience we introduce
\begin{align*}
z = x_{\tresdos}x_2 - q _{12} x_2x_{\tresdos},\qquad
w = x_{\tresdos}x_3 - q _{13} x_3x_{\tresdos}.
\end{align*}

\begin{Theorem} \label{th:endymion-rank4-1}
The algebras $\toba\big(\eny_{\mu, \nu}\big(\bqs, a\big)\big)$ are generated by $x_1$, $x_{\tresdos}$, $x_2$, $x_3$ with defining relations and PBW-basis as follows:
\begin{enumerate}[label=$(\alph*)$]\itemsep=0pt
\item The relations of $\toba\big(\eny_{+, +}\big(\bqs, a\big)\big)$ are
\eqref{eq:endymion-1}, \eqref{eq:endymion-1b}, \eqref{eq:endymion-2},
\begin{alignat}{2}\label{eq:endymion-31}
&x_1x_3 = q_{13}x_3x_1,&&
\\ \label{eq:endymion-31b}
&w^2 =0,&&x_3w = q_{31} wx_3,
\\ \label{eq:endymion-41}
&x_2 x_3 \overset{\blacklozenge}{=} q_{23} x_3x_2,\qquad &&
 x_3z \overset{\diamond}{=} q_{32}q_{31} z x_3.
\end{alignat}
A PBW-basis is formed by the monomials \eqref{eq:monomials-eny++}.

\item The relations of $\toba\big(\eny_{+, -}\big(\bqs, a\big)\big)$ are
\eqref{eq:endymion-1}, \eqref{eq:endymion-1b}, \eqref{eq:endymion-2}, \eqref{eq:endymion-31}, \eqref{eq:endymion-41} and
\begin{align}
\label{eq:endymion-32}
x_3 ^2=0, \qquad x_3 w = -q_{31} w x_3.
\end{align}
A PBW-basis is formed by the monomials
\begin{align*}
x_1^{m_1} x_{\tresdos}^{m_{\tresdos}} x_2^{m_2}z^{n_{1}}x_{3}^{m_{3}}w^{p_{3}}\colon\
m_1, m_{\tresdos}, m_2, p_{3} \in \{0,1\}, \quad n_1, m_3 \in\N_0.
\end{align*}

\item The relations of $\toba(\eny_{-, +}(\bqs, a))$ are
\eqref{eq:endymion-1}, \eqref{eq:endymion-1b}, \eqref{eq:endymion-2b}, \eqref{eq:endymion-31}, \eqref{eq:endymion-31b} and \eqref{eq:endymion-41}.
A PBW-basis is formed by the monomials
\begin{align*}
x_1^{m_1} x_{\tresdos}^{m_{\tresdos}} x_2^{m_2}z^{n_{1}}x_{3}^{m_{3}}w^{p_{3}}\colon\
m_1, m_{\tresdos}, n_1, m_3\in \{0,1\}, \quad m_2, p_{3} \in\N_0.
\end{align*}

\item The relations of $\toba\big(\eny_{-, -}\big(\bqs, a\big)\big)$ are \eqref{eq:endymion-1}, \eqref{eq:endymion-1b}, \eqref{eq:endymion-2b}, \eqref{eq:endymion-31},
\eqref{eq:endymion-41} and \eqref{eq:endymion-32}. A PBW-basis is formed by the monomials
\begin{align*}
x_1^{m_1} x_{\tresdos}^{m_{\tresdos}} x_2^{m_2}z^{n_{1}}x_{3}^{m_{3}}w^{p_{3}}\colon\
m_1, m_{\tresdos}, m_2, m_3\in \{0,1\}, \quad n_1, p_{3} \in\N_0.
\end{align*}
\end{enumerate}

Hence
$\GK \toba\big(\eny_{\mu, \nu}\big(\bqs, a\big)\big) = 2$ for all $\mu, \nu \in \{\pm\}$.
\end{Theorem}

\begin{proof} We prove the claim for $\toba(\eny_{+, +}) := \toba\big(\eny_{+, +}\big(\bqs, a\big)\big)$; for the other algebras is similar.
The relations \eqref{eq:endymion-1}, \eqref{eq:endymion-1b}, \eqref{eq:endymion-2} hold
in $\toba\big(\eny_{+, +}\big(\bqs, a\big)\big)$ because
the braided subspace $\langle x_1, x_{\frac 32}, x_2 \rangle \simeq \eny_+(q_{12})$,
while \eqref{eq:endymion-31}, \eqref{eq:endymion-31b} hold because
$\langle x_1, x_{\frac 32}, x_3 \rangle \simeq \eny_+(q_{13})$ and in both cases Proposition~\ref{prop:paleblocks} applies.
The relation \eqref{eq:endymion-41}$ \blacklozenge$ holds because $\langle x_1, x_3 \rangle$ generates a quantum plane
and $\diamond$ is verified using derivations. Thus we have a surjective map $\toba \to \toba(\eny_{+, +})$,
where $\toba$ is the algebra with the claimed presentation.

From the defining relations, we deduce
\begin{equation}\label{eq:endymion-61}
\begin{split}
&x_1 z \overset{\ast}{=}-q_{12} z x_1, &&x_\tresdos z \overset{\ast}{=} -q_{12} z x_\tresdos,
&&x_1 w \overset{\ast}{=} -q_{13} w x_1,
\\
&x_\tresdos w \overset{\ast}{=} -q_{13} w x_\tresdos,\qquad
&&wz \overset{\circ}{=} -q_{32}q_{31}q_{12} zw,\qquad
&&x_2 w \overset{\bullet}{=} q_{23} q_{21} w x_2.
\end{split}
\end{equation}
Indeed the verification of $\ast$ is direct and $\circ$ follows from them and \eqref{eq:endymion-31b}.
In turn $\bullet$ follows from~\eqref{eq:endymion-41}$\diamond$. Using the defining relations, the definitions of $z$ and $w$
and the relations \eqref{eq:endymion-61}, we see that the monomials
\begin{align}\label{eq:monomials-eny++}
x_1^{m_1} x_{\tresdos}^{m_{\tresdos}} x_2^{m_2}z^{n_{1}}x_{3}^{m_{3}}w^{p_{3}}\colon\
m_1, m_{\tresdos}, n_1, p_{3} \in \{0,1\}, \quad m_2, m_3 &\in\N_0
\end{align}
generate $\toba$ and \emph{a fortiori} $\toba(\eny_{+, +})$. The monomials $x_1^{m_1} x_{\tresdos}^{m_{\tresdos}} x_2^{m_2}z^{n_{1}}$,
respectively $x_{3}^{m_{3}}w^{p_{3}}$,
are linearly independent in $\toba(\eny_{+, +})$ because
\begin{align*}
&\toba(\eny_+(q_{12})) \simeq \Bbbk \langle x_1, x_{\frac 32}, x_2 \rangle \hookrightarrow \toba(\eny_{+, +})
\hookleftarrow \Bbbk \langle x_1, x_{\frac 32}, x_3 \rangle
\simeq \toba(\eny_+(q_{13})).
\end{align*}
The decomposition $V = (V_1 \oplus V_2) \oplus V_3$ induces a linear isomorphism
$\toba(\eny_{+, +}) \simeq \toba(\eny_{+}(q_{12})) \otimes \K$
and $x_3, w \in \K = \toba\left(\adc(\toba(\eny_{+}(q_{12})))(V_3) \right)$, hence we conclude that the monomials \eqref{eq:monomials-eny++} form a basis
of $\toba\big(\eny_{+, +}\big(\bqs, a\big)\big) $.
Finally, the ordered monomials \eqref{eq:monomials-eny++} define an ascending algebra filtration whose associated graded algebra is a
(truncated) quantum polynomial algebra. Hence $\GK \toba\big(\eny_{+, +}\big(\bqs, a\big)\big) = 2 $.
\end{proof}

\subsubsection[Case (I)(b): widetilde\{q\}\_\{23\} neq 1]{Case~\ref{case:I}\ref{case:Ib}: $\boldsymbol{\widetilde{q}_{23} \neq 1}$}
Recall that $\widetilde{q}_{12}= \widetilde{q}_{13}=1$.

\begin{Proposition}\label{prop:pale-2pts-Ib} $\GK \toba(V) = \infty$.
\end{Proposition}

\begin{proof}
We check that $x_2, x_3, z\in \K^1 = \adc(\toba(V_1))(\langle x_2, x_3 \rangle)$ are linearly independent using skew-derivations.\
We show that they span a braided subspace $W$ of diagonal type. First, we have
\begin{gather*}
\Delta_{H}(x_j) = x_j \otimes 1 + g_j \otimes x_j ,\qquad j =2,3, \\
\Delta_{H}(z) = z \otimes 1 - x_1 g_2 \otimes x_2 + g_1g_2 \otimes z.
\end{gather*}
Therefore,
$\delta(x_j) = g_j \otimes x_j$, $j =2,3$,
$\delta(z) = -x_1 g_2 \otimes x_2 + g_1g_2 \otimes z$.
Thus
\begin{gather*}
c(x_j\otimes y) = g_j \cdot y \otimes x_j,\qquad j =2,3, \\
c(z\otimes y) = -\adc(x_1)(g_2 \cdot y) \otimes x_2 + g_1g_2 \cdot y \otimes z
\end{gather*}
for every $y \in \K^1$.
Hence $W$ is a braided vector subspace of $\K^1$ with braiding given in the basis $\{y_1 = x_2, y_2= x_3, y_3 = z \}$ by
\begin{align*}
(c(y_i \otimes y_j))_{i,j\in \I_{3}}=
\begin{pmatrix}
q_{22} x_2 \otimes x_2 & q_{23} x_3 \otimes x_2 & q_{21}q_{22} z \otimes x_2 \\
q_{32} x_2 \otimes x_3 & q_{33} x_3 \otimes x_3 & q_{31}q_{32} z \otimes x_2 \\
q_{12}q_{22} x_2 \otimes z & q_{13}q_{23} x_3 \otimes z & -q_{22} z \otimes z,
\end{pmatrix}
\end{align*}
which is of diagonal type with diagram
$\xymatrix{
\overset{q_{22}}{\circ} \ar @{-}[r]^{\widetilde{q}_{23}} & \overset{q_{33}}{\circ} \ar @{-}[r]^{\widetilde{q}_{23}} &
\overset{-q_{22}}{\circ}}$.
Since $\widetilde{q}_{23} \neq 1$ and $q_{22}\in \{\pm 1\}$, $\GK \toba(W) = \infty$ by \cite[Lemma~2.3.7]{AAH-triang}.
\end{proof}

\subsection[Case~(II)]{Case~\ref{case:II}} In this subsection, we assume that $\widetilde{q}_{12}=1$, $\widetilde{q}_{13} \neq 1$.
We set $v = x_1x_3 - q_{13}x_3x_1$ which is $\neq 0$ by hypothesis.

\begin{Lemma}\label{lemma:caseII-general}
If $\GK \toba(V)$ is finite, then $q_{22} = \widetilde{q}_{23} = 1$, $q_{33} = \widetilde{q}_{13}=-1$.
%
%
%
\end{Lemma}

\begin{proof}
Assume that $q_{22} = 1$.
Then $\widetilde{q}_{23} = 1$ by \cite[Lemma~2.3.7]{AAH-triang} applied to $\langle x_2, x_3\rangle$.
If $a\ne 0$, then $q_{33} = -1 = \widetilde{q}_{13}$ by Theorem~\ref{th:paleblock-point-resumen}.
Next we assume $a = 0$: here, $0 \subset \langle x_1, x_2, x_3\rangle \subset V$ is a flag of Yetter--Drinfeld submodules
such that $\gr V$ (the associated graded object in $\yd{\Bbbk G}$) is of diagonal type. By \cite[Lemma~3.4.2\,(b)]{AAH-triang},
$\gr \toba(V)$ (the graded algebra associated to the filtration induced by the one on $V$) is a pre-Nichols algebra of $\gr V$.
The class $\overline{z}$ of $z$ in $\gr \toba(V)$ is primitive in $\gr \toba(V)$ since
\begin{align*}
\Delta(z) =z\otimes 1 -x_1\otimes x_2 + 1\otimes z
\end{align*}
and $\overline{z}$ is non-zero by \cite[Propositions 8.1.6 and~8.1.7]{AAH-triang}. Let $\mathcal{H}=\gr \toba(V)\#\Bbbk\Gamma$: $\mathcal{H}$ is a pointed Hopf algebra and the diagram of $\mathcal{H}$ is of diagonal type. Let $W$ be the infinitesimal braiding of~$\mathcal{H}$. In~$\mathcal{H}$, $\overline{z}$ has degree 2, $x_1$, $x_{\tresdos}$ and $x_3$ are linearly independent of degree 1 and
\begin{align*}
\Delta(\overline{z}) =\overline{z}\otimes 1+g_1g_2 \otimes \overline{z},\qquad
\Delta(x_i)=x_i\otimes 1+g_{\lfloor i \rfloor}\otimes x_i,
\end{align*}
(where $\lfloor i\rfloor$ is the integral part of $i$) so $\overline{z}$ and the $x_i$'s are linearly independent elements in $W$.
Computing the actions of the corresponding group-like elements on $\overline{z}$ and $x_i$, we see that
$$
\xymatrix@C=40pt{\overset{-1}{\underset{1}{\circ}} \ar@{-}[r]^{\widetilde{q}_{13}} & \overset{q_{33}}{\underset{3}{\circ}} \ar@{-}[d]_{\widetilde{q}_{13}} \ar @{-}[r]^{\widetilde{q}_{13}} & \overset{-1}{\underset{\overline{z}}{\circ}}
\\ & \overset{-1}{\underset{\tresdos}{\circ}} &}
$$
is a subdiagram of the Dynkin diagram of $W$.
By Theorem~\ref{thm:conj-AAH}, we see that $q_{33}=\widetilde{q}_{13} =-1$ by~\cite{H-classif-RS}.

Assume that $q_{22} = -1$.
We check that $z, v \in \K^1 = \adc(\toba(v))(\langle x_2, x_3 \rangle)$ are linearly independent using skew-derivations.
We show that they span a braided subspace $W$ of diagonal type. First, we have
\begin{align*}
&\Delta_{H}(z) = z \otimes 1 - x_1 g_2 \otimes x_2 + g_1g_2 \otimes z , \\
&\Delta_{H}(v) = v \otimes 1 + 2 x_1 g_3 \otimes x_3 + g_1g_3 \otimes v;
\end{align*}
one therefore has
\begin{align*}
\delta(z) = - x_1 g_2 \otimes x_2 + g_1g_2 \otimes z, \qquad
\delta(v) = 2 x_1 g_3 \otimes x_3 + g_1g_3 \otimes v.
\end{align*}
Consequently we have for every $y \in \K^1$
\begin{align*}
&c(z \otimes y) = -\adc(x_1)(g_2 \cdot y) \otimes x_2 + g_1g_2 \cdot y \otimes z, \\
&c(v \otimes y) = 2\adc(x_1)(g_3 \cdot y) \otimes x_3 + g_1g_3 \cdot y \otimes v.
\end{align*}
Hence $W$ is a braided vector subspace of $\K^1$ with braiding in the basis $\{z, v \}$ given by
\begin{align*}
\begin{pmatrix}
z \otimes z & -q_{13}q_{21}q_{23} v \otimes z \\
-q_{12}q_{31}q_{32} z \otimes v & q_{33} v \otimes v
\end{pmatrix}\!,
\end{align*}
so is of diagonal type with diagram $\xymatrix{
\overset{1}{\circ} \ar @{-}[r]^{\widetilde{q}_{13}\widetilde{q}_{23}} & \overset{q_{33}}{\circ} }$.
Assume that $\GK \toba(W) < \infty$.
Then $\widetilde{q}_{13}\widetilde{q}_{23} = 1$ by \cite[Lemma~2.3.7]{AAH-triang}; thus $\widetilde{q}_{23} = \widetilde{q}_{13}^{-1} \neq 1$.
Again, $0 \subset \langle x_1, x_2, x_3\rangle \subset V$ is a flag of Yetter--Drinfeld submodules such that $\gr V$ is of diagonal type; its diagram is
$$
\xymatrix{
\overset{-1}{\underset{1} {\circ}}\ar @{-}[rr]^{\widetilde{q}_{13}} & & \overset{q_{33}} {\underset{3} {\circ}}
\ar @{-}[d]^{\widetilde{q}_{13}} \ar @{-}[dll]_{\widetilde{q}_{23}} \\
\overset{-1}{\underset{2} {\circ}} && \overset{-1}{\underset{\tresdos} {\circ}} }
$$
By \cite[Lemma 3.4.2\,(c)]{AAH-triang}, $\GK \toba(\gr V)\le \GK \toba(V)$.
By Theorem~\ref{thm:conj-AAH}, the unique open case is $q_{33}=\widetilde{q}_{23} = \widetilde{q}_{13} = -1$, see \cite{H-classif-RS}.

Now we fix $q_{33}=\widetilde{q}_{23} = \widetilde{q}_{13} = -1$ and suppose that $\GK \toba(V)<\infty$.
Then $\gr V$ is a braided vector space of Cartan type $D_4$,
and the corresponding graded Hopf algebra $\toba:=\gr \toba(V)$ is a pre-Nichols algebra of $\gr V$ such that $\GK \toba<\infty$,
see \cite[Lemma~3.4.2\,(b)]{AAH-triang}.
Let $y_i$ be the class of $x_i$ in $\toba$,
\begin{align*}
y_{3\tresdos2}=\adc y_3\big(\adc y_{\tresdos} (y_2)\big), \qquad u=(\adc x_3)\big(\adc x_{\tresdos} (x_2)\big).
\end{align*}
Notice that its class $\overline{u}$ in $\toba$ is $\overline{u}=y_{3\tresdos2}$. Then $\overline{u}=0$ by
\cite[Lemma 5.8\,(b)]{ASa-preNichols}.
We claim that there exist $a_i\in\Bbbk$ such that
\begin{align}
u = {}&a_1 x_{132} + a_2 x_2x_{13} +a_3 x_{32}x_1 + a_4 x_3x_{13} +a_5 x_{13}x_1
 + a_6 x_2x_{32}\nonumber
 \\
 & + a_7 x_{32}x_3 + a_8x_2x_3x_1.
\label{eq:linear-comb-u}
\end{align}
Indeed, $u\in\toba(V)^3_4$ and the subspace $\langle x_1,x_2,x_3\rangle$ is of Cartan type
$A_3$ with parameter $-1$, so $\{x_{132},x_2x_{13},x_{32}x_1,x_3x_{13},x_{13}x_1,x_2x_{32},x_{32}x_3, x_2x_3x_1\}$
is a basis of $\toba(V)^3_3$. As $\partial_1(u)\!={\partial_3(u)\!=0}$, we have that
\begin{gather*}
\begin{split}
&0 = a_3 x_{32}-2a_5 x_{3}x_1 +a_5 x_{13}+ a_8x_2x_3,
\\
&0 = 2a_1 x_{12} + 2a_2 x_2x_{1} + a_4 (2x_3x_1 -q_{31} x_{13}) + a_7 (x_{32}- 2 x_{2}x_3) + q_{31} a_8x_2x_1.
\end{split}
\end{gather*}
As $\{ x_{13}, x_{32}, x_2x_3, x_2x_1, x_3x_1 \}$ are linearly independent, we get
$a_3=a_5=a_8=0$ from the first equality, and $a_1=a_2=a_4=a_7=0$, so \eqref{eq:linear-comb-u} reduces to
$u = a_6 x_2x_{32}$. But applying $\partial_2$, we get
\begin{align*}
q_{31} x_{13} -2x_3x_1 = -a_6 q_{32}x_{32}+2a_6x_2x_3,
\end{align*}
a contradiction. Hence $\GK \toba(V)=\infty$.
\end{proof}

In the next subsections, we study two subcases of the situation left open in Lemma~\ref{lemma:caseII-general}, namely $a=0$ and $a \neq 0$.

\subsubsection[Case~(II), when the ghost is infinite]{Case~\ref{case:II}, when the ghost is infinite}
Here $q_{22} = \widetilde{q}_{23} = 1$, $q_{11} = q_{33} = \widetilde{q}_{13}=-1$, $a = 0$.
To spell out our next result, we introduce
\begin{align}\label{eq:defzt-abc}
\zt_{\ell mn}:= x_{\tresdos}^\ell x_{3\tresdos}^m x_{13\tresdos}^n\cdot x_2 ,\qquad
\yt :=[\zt_{110},\zt_{001}]_c= \zt_{110}\zt_{001} +\zt_{001}\zt_{110}.
\end{align}

\begin{Theorem} \label{th:endymion-2pts-case-a}
The algebra $\toba\big(\enyIIb{\infty}\big(\bqs\big) \big)$ is generated by $x_1$, $x_{\tresdos}$, $x_2$, $x_3$ with defining relations
\begin{align}
&\label{eq:endymion-2pts-def-rel-1}
\begin{aligned}
& x_{\tresdos}x_1+x_1x_{\tresdos}, & & x_{\tresdos}^2, & & x_{3\tresdos}^2, & & x_{13\tresdos}^2,
\\
&x_{13\tresdos}x_3+q_{13}^2x_3x_{13\tresdos},\quad & & x_3^2, & & x_{13}^2, & & x_1^2,
\end{aligned}
\\
\label{eq:endymion-2pts-def-rel-2}
& \begin{aligned}
x_2x_3-q_{23}x_3x_2, & & x_1x_2-q_{12}x_2x_1, && x_{\tresdos2}x_2-q_{12}x_2x_{\tresdos2},
\end{aligned}
\\
\label{eq:endymion-2pts-def-rel-3}
& \zt_{\ell mn}^2, \qquad (\ell mn)\in\{(010), (001), (101),(011)\},
\\
\label{eq:endymion-2pts-def-rel-4}
& x_{\tresdos}\yt - q_{12}^2q_{13}^2\yt x_{\tresdos}+q_{12}q_{13}\zt_{001}\zt_{101},
\\
\label{eq:endymion-2pts-def-rel-last}
& \zt_{110}\yt -\yt\zt_{110}+\zt_{001}\yt.
\end{align}
A PBW-basis is formed by the monomials
\begin{align}
&
\zt_{000}^{m_{000}} \zt_{100}^{m_{100}} \zt_{010}^{m_{010}} \zt_{001}^{m_{001}} \yt^m \zt_{110}^{m_{110}} \zt_{101}^{m_{101}} \zt_{011}^{m_{011}} \zt_{111}^{m_{111}}
x_{\tresdos}^{a_1} x_{3\tresdos}^{a_2} x_{13\tresdos}^{a_3} x_3^{a_4} x_{13}^{a_5} x_1^{a_6},
\nonumber\\
& a_i, m_{100}, m_{010}, m_{001}, m_{101},m_{011} \in \{0,1\}, \qquad m_{000}, m, m_{110}, m_{111} \in\N_0.\label{eq:endymion-2pts-case-a-PBW}
\end{align}
Hence $\GK \toba\big(\enyIIb{\infty}\big(\bqs\big)\big) = 4$.
\end{Theorem}

\begin{proof} We proceed by steps.
\begin{stepi}
Note that $V_1\oplus V_3$ is of Cartan type $A_3$ with parameter $q= -1$. Now the defining relations of $\toba(V_1\oplus V_3)$
are \eqref{eq:endymion-2pts-def-rel-1}, see {\rm \cite{AA-diag-survey}}. Thus these relations hold in $\toba\big(\enyIIb{\infty}\big(\bqs\big)\big)$.
Also the following set is a PBW-basis of $\toba(V_1\oplus V_3)$:
\begin{align}\label{eq:Cartan-A3-PBWbasis}
& x_{\tresdos}^a x_{3\tresdos}^b x_{13\tresdos}^c x_3^d x_{13}^e x_1^f, \qquad a,b,c,d,e,f \in\{0,1\}.
\end{align}
Exchanging $1$ and $\tresdos$ we obtain another presentation and PBW-basis of $\toba(V_1\oplus V_3)$.
We will use both presentations and basis in the sequel.
\end{stepi}

\begin{stepi}
The subspace $\langle x_2, x_3 \rangle$ is a quantum plane, and $\langle x_1,x_{\tresdos},x_2\rangle\simeq \eny_{+}(q_{12})$.
Hence the relations \eqref{eq:endymion-2pts-def-rel-2} hold in $\toba\big(\enyIIb{\infty}\big(\bqs\big)\big)$.
\end{stepi}

\begin{stepi}
$B = \{ \zt_{\ell mn} | 0\le \ell,m,n\leq 1 \}$ is a basis of $\K^1 :=\adc(\toba(V_1\oplus V_3))(V_2)$.
\end{stepi}

\begin{proof}[Proof of Steps 1, 2 and 3]
The following formulas are easy to check:
\begin{align*}
&g_1 \cdot \zt_{\ell mn} = (-1)^{\ell +m}q_{12}q_{13}^{m+n}\zt_{\ell mn},
\\ &g_3 \cdot \zt_{\ell mn} = q_{31}^{\ell +m+2n}q_{32}(-1)^{m+n}\zt_{\ell mn},
\\
&g_2 \cdot \zt_{\ell mn} = \begin{cases}
q_{21}^{\ell+m+2n}q_{23}^{m+n}\zt_{\ell mn}, & \ell mn\ne 110,
\\
q_{21}^2q_{23}(\zt_{110}+\zt_{001}), & \ell mn=110.
\end{cases}
\end{align*}
Next we claim that the following relations hold:
\begin{align}\label{eq:endymion-2pts-adc-on-z}
\begin{aligned}
&(\adc x_1)\zt_{\ell mn} = \delta_{m,1}\delta_{n,0}(-1)^{\ell }\zt_{\ell 01}, \qquad
(\adc x_{\tresdos})\zt_{\ell mn} = \delta_{\ell ,0}\zt_{1mn},
\\
&(\adc x_3)\zt_{\ell mn} = \delta_{\ell ,1}\delta_{m,0}\zt_{01n}.
\end{aligned}
\end{align}
The verification uses \eqref{eq:endymion-2pts-def-rel-1}, \eqref{eq:endymion-2pts-def-rel-2} and the definition \eqref{eq:defzt-abc}.
Summarizing, the adjoint action of~$x_i$, $g_j$ can be read in the following graph:
$$
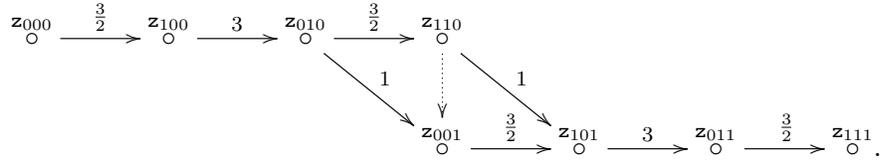
\xymatrix@C=30pt{\overset{\zt_{000}}{\circ} \ar@{->}[r]^{\tresdos} &
\overset{\zt_{100}}{\circ} \ar@{->}[r]^{3} &
\overset{\zt_{010}}{\circ} \ar@{->}[r]^{\tresdos} \ar@{->}[rd]^{1} &
\overset{\zt_{110}}{\circ} \ar@{->}[rd]^{1} \ar@{.>}[d] &
& &
\\
& & &
\overset{\zt_{001}}{\circ} \ar@{->}[r]^{\tresdos} &
\overset{\zt_{101}}{\circ} \ar@{->}[r]^{3} &
\overset{\zt_{011}}{\circ} \ar@{->}[r]^{\tresdos} &
\overset{\zt_{111}}{\circ}. }
$$
\begin{itemize}[leftmargin=*]
\item The elements (one or two) in the $n$-th column have degree $n$.
\item We draw an arrow from $\zt_{\ell mn}$ to $\zt_{pqr}$ labeled with $i\in\big\{1,\tresdos,3\big\}$ if and only if
$(\adc x_i)\zt_{\ell mn}\in\Bbbk^{\times} \zt_{pqr}$. Moreover, this non-zero scalar is $1$ if $i\ne 1$.
\item If there is not an arrow starting in $\zt_{\ell mn}$ with label $i$, then $(\adc x_i)\zt_{\ell mn}=0$.
\item The dotted arrow from $\zt_{110}$ to $\zt_{001}$ means $g_2 \cdot \zt_{110}= q_{21}^2q_{23}(\zt_{110}+\zt_{001})$. Otherwise, the action of $g_i$ on $\zt_{\ell mn}$ is diagonal.
\end{itemize}

By \eqref{eq:endymion-2pts-def-rel-2} and \eqref{eq:Cartan-A3-PBWbasis}, $\K^1$ is spanned by $B$.
Also, $\partial_1(\zt_{\ell mn})=\partial_{\tresdos}(\zt_{\ell mn})=\partial_3(\zt_{\ell mn})=0$ for all $\ell$, $m$, $n$ since $\ker \partial_i$ is a subalgebra of $\toba(V_1\oplus V_3)$ stable by $\adc x_i$. Now we compute $\partial_2(\zt_{100})= -x_1$,
\begin{alignat*}{3}
&\partial_2 (\zt_{010}) = -x_{31}, &&
\partial_2 (\zt_{001}) = -2x_1x_{31}, &&
\partial_2 (\zt_{110}) = -\big(x_{\tresdos31}+x_1x_{31}\big),
\\
&\partial_2 (\zt_{101}) = 2x_1x_{\tresdos31},\qquad &&
\partial_2 (\zt_{011}) = 2x_{31}x_{\tresdos31},\qquad &&
\partial_2 (\zt_{111}) = -2x_1x_{31}x_{\tresdos31}.
\end{alignat*}
Hence $B$ is linearly independent.
\end{proof}

\begin{stepi}\label{lem:endymion-2pts-def-rel-1}
The relations \eqref{eq:endymion-2pts-def-rel-3} and \eqref{eq:endymion-2pts-def-rel-4} hold in $\toba\big(\enyIIb{\infty}\big(\bqs\big)\big)$.
\end{stepi}

\begin{proof}[Proof of Step 4]
As $\partial_i(\zt_{\ell mn})=0$ for all $i\ne 2$, it is enough to check that $\partial_2$ annihilates each one of these relations. Using \eqref{eq:graded-bracket-jacobi} and \eqref{eq:endymion-2pts-adc-on-z},
\begin{align*}
&\partial_2 \big(\zt_{010}^2\big) = -q_{21}q_{23}x_{31}\zt_{010}-\zt_{010}x_{31}
= -q_{21}q_{23}[x_{31}, \zt_{010}]_c
= -q_{21}q_{23}[x_{3}, [x_1,\zt_{010}]_c]_c
\\ &\phantom{\partial_2 (\zt_{010}^2)}
= -q_{21}q_{23}[x_{3}, \zt_{001}]_c =0,
\\
&\partial_2 \big(\zt_{001}^2\big) = -q_{21}^2q_{23}[x_1x_{31}, \zt_{001}]_c
=-q_{21}^2q_{23}x_1[x_{31}, \zt_{001}]_c-q_{21}q_{31}[x_1, \zt_{001}]_cx_{31}=0,
\\
&\partial_2 \big(\zt_{101}^2\big) = 2q_{21}^3q_{23}\big[x_1x_{\tresdos31}, \zt_{101}\big]_c
= 2q_{21}^3q_{23}x_1\big[x_{\tresdos}, [x_{31},\zt_{101}]_c\big]_c
\\&\phantom{\partial_2 (\zt_{101}^2)}
=-4q_{21}^2q_{23}q_{31} x_1\big[x_{\tresdos}, \zt_{011}x_1\big]_c=-4q_{21}^2q_{23}q_{31} x_1\zt_{111}x_1=0,
\\
&\partial_2 \big(\zt_{011}^2\big) = 2q_{21}^3q_{23}^2\big[x_{31}x_{\tresdos31}, \zt_{011}\big]_c
= 2q_{21}^3q_{23}^2x_{31}\big[x_{\tresdos}, [x_{31},\zt_{011}]_c\big]_c=0,
\end{align*}
and \eqref{eq:endymion-2pts-def-rel-3} follows. Next we check that
\begin{align*}
\partial_2 (\yt) = -q_{21}^2q_{23} \big[x_{\tresdos31}+x_1x_{31},\zt_{001}\big]_c-2q_{21}^2q_{23} [x_1x_{31},\zt_{110}]_c+2\zt_{001}x_1x_{31}
 =2\zt_{001}x_1x_{31}.
\end{align*}
Using this equality and \eqref{eq:endymion-2pts-adc-on-z}, we see that \eqref{eq:endymion-2pts-def-rel-4} holds because
\begin{align*}
&\partial_{2} \big(x_{\tresdos}\yt -q_{12}^2q_{13}^2\yt x_{\tresdos}+q_{12}q_{13}\zt_{001}\zt_{101}\big)
\\
&\qquad{}= 2\big[x_{\tresdos}, \zt_{001}x_1x_{31}\big]_c
-2q_{13}q_{21}^2q_{23}x_1x_{31}\zt_{101}+2q_{12}q_{13}\zt_{001} x_1x_{\tresdos31}=0.\tag*{\qed}
\end{align*}
\renewcommand{\qed}{}
\end{proof}

\begin{stepi}\label{lem:endymion-2pts-rels-other-1}
Let $\wtoba$ be an algebra and $x_i\in\wtoba$ such that \eqref{eq:endymion-2pts-def-rel-1}, \eqref{eq:endymion-2pts-def-rel-2}, \eqref{eq:endymion-2pts-def-rel-3}, \eqref{eq:endymion-2pts-def-rel-4} hold. Then the following relations also hold:
\begin{gather}\label{eq:q-commut-zs}
\begin{aligned}
&\zt_{100}\zt_{000}=q_{12}\zt_{000}\zt_{100},&&
\zt_{010}\zt_{100} = -q_{31}q_{32}\zt_{100}\zt_{010},
\\
&\zt_{001}\zt_{010} = q_{13}q_{12}\zt_{010}\zt_{001},&&
\zt_{110}\zt_{010} = q_{13}q_{12}\zt_{010}\zt_{110},
\\
&\zt_{101}\zt_{001} = -q_{13}q_{12}\zt_{001}\zt_{101},&&
\zt_{011}\zt_{101} = q_{31}^3q_{32}\zt_{101}\zt_{011},
\\
&\zt_{111}\zt_{011} = q_{13}^2q_{12}\zt_{011}\zt_{111},&&
\yt\zt_{001} = \zt_{001}\yt,
\\
&\zt_{101}\zt_{110} = -q_{12}q_{13}(\zt_{110}+2\zt_{001})\zt_{110},\quad&&
\zt_{100}^2=0;
\end{aligned}
\\ \label{eq:rels-zs}
\begin{aligned}
& [\zt_{011},\zt_{001}]_c=0,\quad&&
[\zt_{111},\zt_{001}]_c=0,
\\
&[\zt_{111},\zt_{001}]_c=0,&&
[\zt_{011},\zt_{110}]_c =-q_{12}q_{31}q_{32}\zt_{001}\zt_{011}.
\end{aligned}
\end{gather}
In particular, these relations hold in $\toba\big(\enyIIb{\infty}\big(\bqs\big) \big)$.
\end{stepi}

\begin{proof}[Proof of Steps 5]
The relation $\zt_{100}\zt_{000}=q_{12}\zt_{000}\zt_{100}$ is \eqref{eq:endymion-2pts-def-rel-2}, and from this relation we deduce that $\zt_{100}^2=0$. Using \eqref{eq:endymion-2pts-def-rel-4} and \eqref{eq:endymion-2pts-adc-on-z},
\begin{align*}
-q_{12}q_{13}\zt_{001}\zt_{101} &= [x_{\tresdos}, \yt]_c = [x_{\tresdos}, [\zt_{001}, \zt_{110}]_c]_c
= [\zt_{101}, \zt_{110}]_c
\\
&= \zt_{101}\zt_{110}+ q_{12}q_{13}(\zt_{110}+\zt_{001})\zt_{110}.
\end{align*}
All the other relations involve $\zt_{\ell mn}$ and $\zt_{def}$ such that $\zt_{\ell mn}=(\adc x_i)\zt_{def}$ for some $d,e,f\in\{0,1\}$ and $i\in\big\{1,3,\tresdos\big\}$, and also $\zt_{def}^2=0$. If $i=1,\tresdos$, then
\begin{align*}
\zt_{\ell mn}\zt_{def} &= (x_i\zt_{def}-(-1)^{d+e}q_{12}q_{13}^{e+f}\zt_{def}x_i) \zt_{def}
\\
&= -(-1)^{d+e}q_{12}q_{13}^{e+f}\zt_{def}x_i\zt_{def}
= -(-1)^{d+e}q_{12}q_{13}^{e+f}\zt_{def}\zt_{\ell mn},
\end{align*}
If $i=3$, then an analogous proof shows that $\zt_{\ell mn}$ and $\zt_{def}$ $q$-commute. For the last relation, we use the definition of $\yt$ and that $\zt_{001}^2=0$.

By \eqref{eq:q-commut-zs}, elements $\zt_{\ell mn}$ and $\zt_{def}$ joined by an arrow $q$-commute.
The relations \eqref{eq:rels-zs} are $q$-commutations between other $\zt_{\ell mn}$'s.
By the defining relations, \eqref{eq:graded-bracket-jacobi} and \eqref{eq:q-commut-zs} we have
\begin{gather*}
0 = [x_3, [\zt_{101},\zt_{001}]_c]_c = [[x_3, \zt_{101}]_c,\zt_{001}]_c =[\zt_{011},\zt_{001}]_c,
\\
0 = \big[x_{\tresdos}, [\zt_{011},\zt_{101}]_c\big]_c = \big[\big[x_{\tresdos}, \zt_{011}\big]_c,\zt_{101}\big]_c =[\zt_{111},\zt_{001}]_c,
\\
0 = \big[x_{\tresdos}, [\zt_{011},\zt_{001}]_c\big]_c
=[\zt_{111},\zt_{001}]_c,
\\
0 = [x_3, [\zt_{101}, \zt_{110}]_c+q_{12}q_{13}\zt_{001}\zt_{101}]_c
= [\zt_{011}, \zt_{110}]_c+q_{12}q_{31}q_{32}\zt_{001}\zt_{011},
\end{gather*}
and the step follows.
\end{proof}

\begin{stepi}
The relation \eqref{eq:endymion-2pts-def-rel-last} holds in $\toba\big(\enyIIb{\infty}\big(\bqs\big)\big)$.
\end{stepi}

\begin{proof}[Proof of Step 6]
By the formulas for $\partial_2$ and the relations in Steps~\ref{lem:endymion-2pts-def-rel-1} and~\ref{lem:endymion-2pts-rels-other-1},
we have
\begin{gather*}
\partial_2 (\zt_{110}\yt -\yt\zt_{110}) =
2\zt_{110}\zt_{001}x_1x_{31}
-q_{21}^4q_{23}^2\big(x_{\tresdos31}+x_1x_{31}\big)\yt+\yt\big(x_{\tresdos31}+x_1x_{31}\big)
\\
\phantom{\partial_2 (\zt_{110}\yt -\yt\zt_{110})=}
-2q_{21}^2q_{23}\zt_{001}x_1x_{31}(\zt_{110}+\zt_{001})
\\
\phantom{\partial_2 (\zt_{110}\yt -\yt\zt_{110})}
= 2\yt x_1x_{31}
-2q_{21}^2q_{23}\zt_{001}[x_1x_{31},\zt_{110}]_c
-q_{21}^4q_{23}^2\big[x_{\tresdos31}+x_1x_{31},\yt \big]_c
\\
\phantom{\partial_2 (\zt_{110}\yt -\yt\zt_{110})} = 2\yt x_1x_{31},
\\
\partial_2 (\zt_{001}\yt) = \partial_2 (\yt\zt_{001})
=-2\yt x_1x_{31} + 2q_{21}^2q_{23}\zt_{001}x_1x_{31}\zt_{001}=-2\yt x_1x_{31}.
\end{gather*}
Hence \eqref{eq:endymion-2pts-def-rel-last} holds in $\toba\big(\enyIIb{\infty}\big(\bqs\big)\big)$.
\end{proof}

Let $\toba$ be the algebra with the claimed presentation. By the previous steps,
there is a~surjective map $\toba \to \toba\big(\enyIIb{\infty}\big(\bqs\big)\big)$.\
To prove that this is an isomorphism, we order the set of PBW generators $(S,<)$ from \eqref{eq:endymion-2pts-case-a-PBW}
by
\begin{gather*}
\zt_{000}>\zt_{100}>\zt_{010}> \zt_{001} > \yt > \zt_{110} > \zt_{101} > \zt_{011} > \zt_{111}
\\ \phantom{\zt_{000}}
{}> x_{\tresdos} > x_{3\tresdos} > x_{13\tresdos} > x_3 > x_{13} > x_1.
\end{gather*}
Let $Z$ be the subspace spanned by the set of monomials \eqref{eq:endymion-2pts-case-a-PBW}.
We establish new relations using \eqref{eq:graded-bracket-jacobi},
\eqref{eq:endymion-2pts-adc-on-z}, \eqref{eq:q-commut-zs} and \eqref{eq:rels-zs}:
\begin{align}\label{eq:rels-zs-more}
\begin{aligned}
&[ \zt_{010},\zt_{000}]_c = [ \zt_{110},\zt_{100}]_c = [ \yt,\zt_{010}]_c = 0,
\\
&[ \zt_{101}, \yt ]_c = [ \zt_{001},\zt_{100}]_c = [ \zt_{111},\zt_{101}]_c = 0.
\end{aligned}
\end{align}
The relations \eqref{eq:rels-zs-more} together with \eqref{eq:q-commut-zs} and \eqref{eq:rels-zs}
say that for every pair $s<s'\in S$ joined by an arrow or that have only one element in the middle, $ss'$ is a linear combination of monomials in $Z$ which are products of elements $>s$. Recursively we get the same statement for every pair $s<s'\in S$. Hence the monomials
\eqref{eq:endymion-2pts-case-a-PBW} generate $\toba$ and \emph{a fortiori} $\toba\big(\enyIIb{\infty}\big(\bqs\big)\big)$.
Since $V = (V_1 \oplus V_3) \oplus V_2$, the multiplication gives a linear isomorphism $\toba\big(\enyIIb{\infty}\big(\bqs\big)\big) \simeq \K \otimes \toba(V_1\oplus V_3)$. Then the problem reduces to prove that the monomials
\begin{align*}
& \zt_{000}^{m_{000}} \zt_{100}^{m_{100}} \zt_{010}^{m_{010}} \zt_{001}^{m_{001}} \yt^m \zt_{110}^{m_{110}} \zt_{101}^{m_{101}} \zt_{011}^{m_{011}} \zt_{111}^{m_{111}},
\\
& m_{100}, m_{010}, m_{001}, m_{101},m_{011} \in \{0,1\}, \qquad m_{000}, m, m_{110}, m_{111} \in\N_0,
\end{align*}
are linearly independent (so they form a basis of $\K$). Suppose on the contrary that there exists a non-trivial linear combination $\mathtt{S}$ of these elements: we may assume that $\mathtt{S}$ is homogeneous of minimal degree. By \eqref{eq:endymion-2pts-adc-on-z},
\begin{align*}
x_1x_{31}x_{\tresdos31}\zt_{111}=\zt_{111}x_1x_{31}x_{\tresdos31},
\end{align*}
and by direct computations,
\begin{align*}
\partial_1\partial_3\partial_1\partial_{\tresdos}\partial_3\partial_1\big(x_1x_{31}x_{\tresdos31}\big)=
\partial_1\partial_3\partial_1(4x_1x_{31})=8.
\end{align*}
As $\partial_i(\zt_{\ell mn})=0$ if $i\ne 2$ and $\partial_2(\zt_{\ell mn})$ has degree $<7$ if $\ell mn\ne 111$ (so $\partial_1\partial_3\partial_1\partial_{\tresdos}\partial_3\partial_1$ annihilates $\partial_2(\zt_{\ell mn})$), we have that
\begin{align*}
&\partial_1\partial_3 \partial_1\partial_{\tresdos}\partial_3\partial_1\partial_2\big(\zt_{000}^{m_{000}} \zt_{100}^{m_{100}} \zt_{010}^{m_{010}} \zt_{001}^{m_{001}} \yt^m \zt_{110}^{m_{110}} \zt_{101}^{m_{101}} \zt_{011}^{m_{011}} \zt_{111}^{m_{111}}\big)
\\ &\qquad{} = -16m_{111} \zt_{000}^{m_{000}} \zt_{100}^{m_{100}} \zt_{010}^{m_{010}} \zt_{001}^{m_{001}} \yt^m \zt_{110}^{m_{110}} \zt_{101}^{m_{101}} \zt_{011}^{m_{011}} \zt_{111}^{m_{111}-1}.
\end{align*}
Hence all the elements in $\mathtt{S}$ with non-zero coefficient have $m_{111}=0$ by the minimality of the degree. Analogously, $m_{011}=m_{101}=0$ since
\begin{align*}
& \partial_3 \partial_1\partial_{\tresdos}\partial_3\partial_1\partial_2\big(\zt_{000}^{m_{000}} \zt_{100}^{m_{100}} \zt_{010}^{m_{010}} \zt_{001}^{m_{001}} \yt^m \zt_{110}^{m_{110}} \zt_{101}^{m_{101}} \zt_{011}^{m_{011}}\big)
\\
& \qquad {}=16\delta_{m_{011},1} \zt_{000}^{m_{000}} \zt_{100}^{m_{100}} \zt_{010}^{m_{010}} \zt_{001}^{m_{001}} \yt^m \zt_{110}^{m_{110}} \zt_{101}^{m_{101}},
\\
& \partial_1\partial_{\tresdos}\partial_3\partial_1\partial_2\big(\zt_{000}^{m_{000}} \zt_{100}^{m_{100}} \zt_{010}^{m_{010}} \zt_{001}^{m_{001}} \yt^m \zt_{110}^{m_{110}} \zt_{101}^{m_{101}}\big)
\\
& \qquad {}=16\delta_{m_{101},1} \zt_{000}^{m_{000}} \zt_{100}^{m_{100}} \yt^m \zt_{010}^{m_{010}} \zt_{001}^{m_{001}} \zt_{110}^{m_{110}}.
\end{align*}
Next we compute
\begin{align*}
\partial_2 \big(\zt_{110}^2\big) &= -q_{21}^2q_{23} \big[x_{\tresdos31}+x_1x_{31}, \zt_{110}\big]_c
-q_{21}^2q_{23} (x_{\tresdos31}+x_1x_{31})\zt_{001}
\\
& = -q_{21}^2q_{23} \zt_{111}
-q_{21}q_{23}q_{13}^2 \zt_{011}x_1
-q_{21}q_{31} \zt_{101}x_{31}
-\zt_{001}\partial_2(\zt_{110}).
\end{align*}
By induction on $t\in\N$, we obtain that
\begin{gather*}
\partial_2 \big(\zt_{110}^{2t-1}\big) \in \zt_{110}^{2t-2}\partial_2(\zt_{110})+ \sum_{j=0}^2 \K \toba^j(V_1\oplus V_3),
\\[-2mm]
\partial_2 \big(\zt_{110}^{2t}\big) \in -\zt_{001}\zt_{110}^{2t-2}\partial_2(\zt_{110})+ \sum_{j=0}^2 \K \toba^j(V_1\oplus V_3).
\end{gather*}
Using these equalities we obtain the following:
\begin{gather*}
 \partial_{\tresdos}\partial_3\partial_1\partial_2\big(\zt_{000}^{m_{000}} \zt_{100}^{m_{100}} \zt_{010}^{m_{010}} \zt_{001}\yt^m \zt_{110}^{2t-1}\big)=-4 \zt_{000}^{m_{000}} \zt_{100}^{m_{100}} \zt_{010}^{m_{010}}\zt_{001}\yt^m \zt_{110}^{2t-2},
\\
 \partial_{\tresdos}\partial_3\partial_1\partial_2\big(\zt_{000}^{m_{000}} \zt_{100}^{m_{100}} \zt_{010}^{m_{010}} \zt_{001}\yt^m \zt_{110}^{2t}\big)=0,
\\
 \partial_{\tresdos}\partial_3\partial_1\partial_2\big(\zt_{000}^{m_{000}} \zt_{100}^{m_{100}} \zt_{010}^{m_{010}}\yt^m \zt_{110}^{2t-1}\big)=-4 \zt_{000}^{m_{000}} \zt_{100}^{m_{100}} \zt_{010}^{m_{010}}\yt^m \zt_{110}^{2t-2},
\\
 \partial_{\tresdos}\partial_3\partial_1\partial_2\big(\zt_{000}^{m_{000}} \zt_{100}^{m_{100}} \zt_{010}^{m_{010}} \yt^m \zt_{110}^{2t}\big)=4 \zt_{000}^{m_{000}} \zt_{100}^{m_{100}} \zt_{010}^{m_{010}} \zt_{001}\yt^m \zt_{110}^{2t-1},
\\
 \partial_1\partial_3\partial_1\partial_2\big(\zt_{000}^{m_{000}} \zt_{100}^{m_{100}} \zt_{010}^{m_{010}} \zt_{001}\yt^m \zt_{110}^{2t}\big)=-4\zt_{000}^{m_{000}} \zt_{100}^{m_{100}} \zt_{010}^{m_{010}}\yt^m \zt_{110}^{2t},
\\
 \partial_1\partial_3\partial_1\partial_2\big(\zt_{000}^{m_{000}} \zt_{100}^{m_{100}} \zt_{010}^{m_{010}} \zt_{001}\yt^m\big)=-4\zt_{000}^{m_{000}} \zt_{100}^{m_{100}} \zt_{010}^{m_{010}}\yt^m,
\\
 \partial_1\partial_3\partial_1\partial_2\big(\zt_{000}^{m_{000}} \zt_{100}^{m_{100}} \zt_{010}^{m_{010}} \yt^m\big)=-4m\zt_{000}^{m_{000}} \zt_{100}^{m_{100}} \zt_{010}^{m_{010}}\zt_{001}\yt^{m-1},
\\
 \partial_1\partial_3\partial_1\partial_2\big(\zt_{000}^{m_{000}} \zt_{100}^{m_{100}} \zt_{010}^{m_{010}} \zt_{001}^{m_{001}}\big)=-4\delta_{m_{101},1} \zt_{000}^{m_{000}} \zt_{100}^{m_{100}} \zt_{010}^{m_{010}}.
\end{gather*}
Thus we get that all the elements in $\mathtt{S}$ with non-zero coefficient have $m_{110}=m=m_{001}=0$
applying either $\partial_{\tresdos}\partial_3\partial_1\partial_2$ or else $\partial_1\partial_3\partial_1\partial_2$. Next,
\begin{gather*}
 \partial_3\partial_1\partial_2\big(\zt_{000}^{m_{000}} \zt_{100}^{m_{100}} \zt_{010}^{m_{010}}\big)
=-2\delta_{m_{101},1} \zt_{000}^{m_{000}} \zt_{100}^{m_{100}},
\\[1mm]
 \partial_1\partial_2\big(\zt_{000}^{m_{000}} \zt_{100}^{m_{100}}\big)
=-\delta_{m_{101},1} \zt_{000}^{m_{000}},
\end{gather*}
so $\mathtt{S}=a \zt_{000}^n$, $a\in\Bbbk^{\times}$, and we get a contradiction since $\zt_{000}^n\ne0$ for all $n\in\mathbb N_0$. Thus \eqref{eq:endymion-2pts-case-a-PBW} is a~basis of $\toba\big(\enyIIb{\infty}\big(\bqs\big)\big)$, and $\toba=\toba\big(\enyIIb{\infty}\big(\bqs\big)\big)$.

Finally, the ordered monomials \eqref{eq:endymion-2pts-case-a-PBW} define an ascending algebra filtration whose associated graded algebra is a
(truncated) quantum polynomial algebra. Hence
\begin{equation*}
\GK \toba\big(\enyIIb{\infty}\big(\bqs\big)\big) =4.
\tag*{\qed}
\end{equation*}
\renewcommand{\qed}{}
\end{proof}

\subsubsection[Case (II), finite ghost]{Case~\ref{case:II}, finite ghost}\label{subsubsec:caseII-2}
Here
$q_{22} = \widetilde{q}_{23} = 1$, $q_{33} = \widetilde{q}_{13}=-1$.
Let $\bqs = (q_{12}, q_{13}, q_{23}) \in (\Bbbk^{\times})^3$. Define
$\bq$ by \eqref{eq:bq-IIa-b}.

\begin{Proposition} \label{prop:endymion-two-pts-case-b}
Assume that $a\ne 0$. Then $\GK \toba(V)=\infty$.
\end{Proposition}

\begin{proof}[Overview of the proof]
By the spliting technique, see Section~\ref{subsec:splitting-technique}, it suffices to
show that $\GK \K = \infty$, where $\K = \toba\left( \K^1 \right)$ and $\K^1 = \adc(\toba(V_3))(V_1 \oplus V_2)$.
Clearly, $V_3^*$, which is generated by $f_3$ with $f_3(x_3)=1$, belongs to $\yd{\Bbbk\Gamma}$ with structure
\begin{align*}
\delta(f_3)=g_3^{-1}\otimes f_3, \qquad g_i\cdot f_3=q_{i3}^{-1}f_3;
\end{align*}
particularly, $\toba(V_3^*) \simeq \Lambda (V_3^*)$.
Thus we may consider $\toba(V_3^*)\#\Bbbk \Gamma$ and use the braided monoidal isomorphism of~\cite[Remark~12.3.8]{HS-book}
\begin{align*}
(\Omega,\omega)\colon\ \yd{\toba(V_3)\#\Bbbk \Gamma}\to \yd{\toba(V_3^*)\#\Bbbk \Gamma}.
\end{align*}
By~\cite[Corollary~12.3.9]{HS-book}, $(\Omega,\omega)(\K)\simeq \toba(Z)$, where $Z=(\Omega,\omega)\big(\K^1\big)$.
Now we introduce $W=W_1\oplus V_2\oplus V_3^*$, see Step~\ref{stepii:def-W}, and apply the splitting technique again:
let $\Ku = \toba\left(W\right)^{\co \toba(V_3^*)} \simeq \toba \big(\Ku^1\big)$, where
\begin{align*}
\Ku^1=\adc(\toba(V_3^*))(W_1 \oplus V_2).
\end{align*}
We shall derive from \cite[Lemma 5.4.11]{AAH-triang} that $\GK\toba(W)=\infty$, hence $\GK \toba\big(\Ku^1\big)=\infty$ since
$\toba(W) \simeq \toba\big(\Ku^1\big) \# \toba(V_3^*)$ and $\dim \toba(V_3^*)=2$.

Finally, we show in Step~\ref{stepii:isomorphism} that $Z\simeq \Ku^1$.
Since the functor $(\Omega,\omega)$ preserves the algebra structure,
$\GK \K=\GK\Ku=\infty$, so $\GK\toba(V)=\infty$.

\begin{stepii}\label{stepii:base1}
The set $B=\big\{x_1,x_{\tresdos},x_{31},x_{3\tresdos},x_2\big\}$ is a basis of $\K^1$
and the coaction of the elements of~$B$ is $\delta(x_i) = g_{\lfloor i\rfloor}\otimes x_i$,
where $\lfloor i\rfloor$ is the integral part of $i$,
\begin{align*}
&\delta(x_{31}) = 2x_3g_1 \otimes x_1 + g_1g_3 \otimes x_{31},
\\
&\delta\big(x_{3\tresdos}\big) = x_3g_1 \otimes \big(2x_{\tresdos}+ax_1\big) + g_1g_3 \otimes x_{3\tresdos}.
\end{align*}
\end{stepii}

Indeed, $(\adc x_3)x_2=0$ and $x_3^2=0$, so $B$ spans $\K^1$.
The computation of the coaction is direct; it implies in turn that $B$ is linearly independent.

\begin{stepii}
Here is the structure of $Z\in \yd{\toba(V_3^*)\#\Bbbk \Gamma}$.
By definition, $Z=\K^1$ as vector space, and the $\Gamma$-action on $Z$ coincides with the one of $\K^1$.
Next:
\begin{enumerate}[label=$(\roman*)$]
\item The $\toba(V_3^*)$-action on $Z$ is given by:
\begin{align*}
f_3 \cdot x_i =0,\qquad f_3 \cdot x_{31} = 2x_1,\qquad f_3 \cdot x_{3\tresdos} = 2x_{\tresdos}+ax_1.
\end{align*}
\item The coaction $\delta\colon Z\to \toba(V_3^*)\#\Bbbk\Gamma\otimes Z$ is given by:
\begin{alignat*}{2}
&\delta(x_1) = f_3g_1 \otimes x_{31}+g_1\otimes x_1,&&
\delta(x_{\tresdos}) = f_3g_1 \otimes \big(2x_{3\tresdos}-ax_{31}\big)+g_1\otimes x_{\tresdos},
\\
&\delta(x_{3j})=g_1g_3\otimes x_{3j}, \quad j=1,\tfrac{3}{2},\qquad
&&\delta(x_2) = g_2\otimes x_2.
\end{alignat*}
\end{enumerate}
\end{stepii}
This follows from \cite[Theorem 12.3.2 and Remark 12.3.8]{HS-book} by Step~\ref{stepii:base1}.

\begin{stepii} \label{stepii:def-W}
Let $W=W_1\oplus V_2\oplus V_3^*$,
where $W_1\in\yd{\Bbbk\Gamma}$ is homogeneous of degree $g_1g_3$, has a basis $w_1$, $w_{\tresdos}$ and $\Gamma$-action given by
\begin{alignat*}{2}
&g_i\cdot w_1=q_{i1}q_{i3}w_1, &&
g_1 \cdot w_{\tresdos}=-w_{\tresdos}+w_1,
\\
&g_2 \cdot w_{\tresdos}=q_{21}q_{23}\big(w_{\tresdos}+w_1\big), \qquad&&
g_3 \cdot w_{\tresdos}=-q_{31}\big(w_{\tresdos}+aw_1\big).
\end{alignat*}
As $a\ne 0$, $W_1$ is a $-1$-block and $W$ is a sum of a block with two points, where $V_3^*$ has mild interaction and $V_2$ has weak interaction.
By {\rm \cite[Lemma~5.4.11]{AAH-triang}}, $\GK\toba(W)=\infty$.
\end{stepii}

\begin{stepii}\label{stepii:base2} Let $w_{3i}:=(\adc f_3)w_i$, $i=1,\tresdos$.
\begin{enumerate}[label=$(\roman*)$]
\item The set $\underline{B}=\big\{w_1,w_{\tresdos},w_{31},w_{3\tresdos},x_2\big\}$ is a basis of $\Ku^1$.

\item The coaction $\delta\colon\ \Ku^1\to \toba(V_3^*)\#\Bbbk\Gamma\otimes \Ku^1$ is given by:
\begin{alignat*}{2}
&\delta(w_{31}) = 2f_3g_1 \otimes w_{1}+g_1\otimes w_{31},
&&\delta(w_{j})=g_1g_3\otimes x_{j}, \quad j=1,\tfrac{3}{2},
\\
&\delta(w_{3\tresdos}) = f_3g_1 \otimes \big(2w_{\tresdos}+aw_{1}\big)+g_1\otimes w_{3\tresdos}, \qquad
&&\delta(x_2) = g_2\otimes x_2.
\end{alignat*}
\end{enumerate}
\end{stepii}
The proof follows as for $\K^1$. Finally we deduce from Step~\ref{stepii:base2}:

\begin{stepii}\label{stepii:isomorphism}
The linear isomorphism $Z\to\Ku^1$ given by
\begin{align*}
x_{31} \mapsto 2w_1,\qquad
x_{3\tresdos} \mapsto 2w_{\tresdos},\qquad
x_{1} \mapsto w_{31},\qquad
x_{\tresdos} \mapsto w_{3\tresdos}-aw_{31},\qquad
x_{2} \mapsto x_2,
\end{align*}
is $\toba(V_3^*)\#\Bbbk\Gamma$-linear and $\toba(V_3^*)\#\Bbbk\Gamma$-colinear. \hfill \qed
\end{stepii}
\renewcommand{\qed}{}
\end{proof}

\subsection[Case~(III)]{Case~\ref{case:III}}
In this subsection, we assume that $\widetilde{q}_{12}=-1$, $\widetilde{q}_{13}=1$.
Hence $q_{22}=-1$, and $V_1\oplus V_2$ is isomorphic to $\eny_{\star}(q_{12})$.

\begin{Lemma}\label{lemma:caseIII-general}
If $\GK \toba(V)$ is finite, then either of the following holds:{\samepage
\begin{enumerate}[label=$(\Alph*)$]
\item\label{item:caseIII-A} $a = 0$, $q_{33} =\widetilde{q}_{23}= -1,$

\item\label{item:caseIII-B} $a \ne 0$, $q_{33} =\widetilde{q}_{23}= 1,$

\item\label{item:caseIII-C} $a \ne 0$, $q_{33} =\widetilde{q}_{23}= -1$.
\end{enumerate}}
\end{Lemma}

\begin{proof}
Let $\toba=\toba(V)$. Then $0 \subset \langle x_1, x_2, x_3\rangle \subset V$ is a flag in $\yd{\Bbbk G}$, $\gr V$ is a braided vector space of diagonal type, and the corresponding graded Hopf algebra $\gr \toba$ is a pre-Nichols algebra of $\gr V$, see \cite[Lemma 3.4.2]{AAH-triang}.

We assume first that $a=0$, so $\widetilde{q}_{23}\ne 1$.
Let $u$ be the class of $z_1^2$ in $\gr \toba$. Then $u$ is a non-zero primitive element in $\gr \toba$, see the proof of \cite[Proposition 8.1.8]{AAH-triang}. Let $\mathcal{H}=\gr \toba\#\Bbbk\Gamma$: $\mathcal{H}$ is a~pointed Hopf algebra and the diagram of $\mathcal{H}$ is of diagonal type. Let $W$ be the infinitesimal braiding of $\mathcal{H}$. In $\mathcal{H}$, $u$ has degree 4, the $x_i$'s are linearly independent of degree 1 and
\begin{align*}
\Delta(u) =u\otimes 1+g_1^2g_2^2 \otimes u, \qquad \Delta(x_i)=x_i\otimes 1+g_i\otimes x_i,
\end{align*}
so $u$ and the $x_i$'s are linearly independent vectors in $W$.
Computing the actions of $g_1^2g_2^2$ and $g_3$ on $u$ and $x_3$, we see that
$$
\xymatrix@C=35pt{\overset{-1}{\underset{1}{\circ}} \ar@{-}[r]^{-1} & \overset{-1}{\underset{2}{\circ}} \ar@{-}[d]^{-1} \ar @{-}[r]^{\widetilde{q}_{23}} &
\overset{q_{33}}{\underset{3}{\circ}} \ar @{-}[r]^{\widetilde{q}_{23}^2} & \overset{1}{\underset{u}{\circ}}
\\ & \overset{-1}{\underset{\tresdos}{\circ}} & &}
$$
is a subdiagram of the Dynkin diagram of $W$. Thus $\widetilde{q}_{23}=-1$ by \cite[Lemma 2.3.7]{AAH-triang},
and $q_{33}=-1$ by Theorem~\ref{thm:conj-AAH}.

Now we assume $a\ne 0$. Hence $q_{33}= \pm 1$.
Let $z$ be the class of $\big(\adc x_{\tresdos}\big)x_3$ in $\gr \toba$. Then $z$ is a~non-zero primitive element in $\gr \toba$ by \cite[Propositions 8.1.6 and 8.1.7]{AAH-triang}. Let $\mathcal{H}=\gr \toba\#\Bbbk\Gamma$: $\mathcal{H}$~is a pointed Hopf algebra and the diagram of $\mathcal{H}$ is of diagonal type. Let $W$ be the infinitesimal braiding of $\mathcal{H}$. In $\mathcal{H}$, $z$ has degree 2, the $x_i$'s are linearly independent of degree 1 and
\begin{align*}
\Delta(z) =z\otimes 1+g_1g_2 \otimes z, \qquad \Delta(x_i)=x_i\otimes 1+g_i\otimes x_i,
\end{align*}
so $z$ and the $x_i$'s are linearly independent elements in $W$.
Computing the actions of the corresponding group-like elements on $z$ and $x_i$, we see that
$$
\xymatrix@C=40pt{\overset{-1}{\underset{1}{\circ}} \ar@{-}[r]^{-1} & \overset{-1}{\underset{2}{\circ}} \ar@{-}[d]_{-1} \ar @{-}[r]^{\widetilde{q}_{23}}
\ar@{-}[rd]^{-\widetilde{q}_{23}} & \overset{q_{33}}{\underset{3}{\circ}}
\\ & \overset{-1}{\underset{\tresdos}{\circ}} & \overset{-q_{33}}{\underset{z}{\circ}}}
$$
is a subdiagram of the Dynkin diagram of $W$, thus $\widetilde{q}_{23}=q_{33}$ by \cite[Lemma 2.3.7]{AAH-triang}.
\end{proof}

Notice that~\ref{item:caseIII-B} corresponds to Lemma~\ref{lemma:caseII-general} up to exchanging $x_2$ and $x_3$, so this situation was treated previously. Also, \ref{item:caseIII-C} was discarded in Lemma~\ref{lemma:caseII-general}, up to exchanging $x_2$ and $x_3$. Thus we only have to deal with~\ref{item:caseIII-A}.

\begin{Proposition}\label{prop:caseIII-1}
If $a = 0$ and $q_{33} =\widetilde{q}_{23}= -1$, then $\GK \toba(V)=\infty$.
\end{Proposition}

\begin{proof}
We consider $\K^1 = \adc(\toba(V_2))(\langle V_1\oplus V_3 \rangle)$; as $x_2^2=0$, the set
\begin{align*}
\big\{x_1, x_{\tresdos}, x_3, \yt_{1}:=(\adc x_2)x_{1}, \yt_{\tresdos}:=(\adc x_2)x_{\tresdos}, \yt_{3}:=(\adc x_2)x_{3} \big\}
\end{align*}
is a basis of $\K^1$. The coaction for the $\yt_i$'s is given by
\begin{align*}
\begin{split}
&\delta(\yt_1) =2x_2g_1 \otimes x_1 + g_1g_2 \otimes \yt_1,
\\
&\delta(\yt_3) =2x_2g_3 \otimes x_3 + g_2g_3 \otimes \yt_3,
\\
&\delta(\yt_{\tresdos}) =2x_2g_1 \otimes x_{\tresdos} +x_2g_1 \otimes x_1 + g_1g_2 \otimes \yt_{\tresdos}.
\end{split}
\end{align*}
Then the subspace $W$ spanned by the $\yt_i$'s is a braided subspace with braiding
\begin{align*}
\begin{pmatrix}
- \yt_1 \otimes \yt_1& (-\yt_{\tresdos} - \yt_1) \otimes \yt_1& -q_{12}q_{13}q_{23} \yt_2 \otimes \yt_1
\\
- \yt_1 \otimes \yt_{\tresdos} & (-\yt_{\tresdos} - \yt_1) \otimes \yt_{\tresdos}& -q_{12}q_{13}q_{23} \yt_2 \otimes \yt_{\tresdos}
\\
-q_{21}q_{31}q_{32} \yt_1 \otimes \yt_2 & -q_{21}q_{31}q_{32}\big(\yt_{\tresdos} + \yt_1\big) \otimes \yt_2& -\yt_2 \otimes \yt_2
\end{pmatrix}\!.
\end{align*}
Then the braiding corresponds to a sum of a block, in the basis $\bigl\{-\yt_1, \yt_{\tresdos}\bigr\}$, with $\epsilon=-1$, and a~point $\yt_2$ with label $-1$: the ghost is $-1$, so by \cite[Theorem~4.1.1]{AAH-triang}, $\GK \toba(W)=\infty$. Thus $\GK \toba(V)=\infty$.
\end{proof}

\subsection[Case (IV)]{Case~\ref{case:IV}}
In this subsection, we suppose that $\widetilde{q}_{12}=-1$, $\widetilde{q}_{13} \neq 1$.

\begin{Proposition}\label{prop:caseIV}
If $\widetilde{q}_{12}=-1$, $\widetilde{q}_{13} \neq 1$, then $\GK \toba(V)=\infty$.
\end{Proposition}

\begin{proof} Here $0 \subset \langle x_1, x_2, x_3\rangle \subset V$ is a flag of YD modules: $\gr V$ is of diagonal type with diagram
$$
\xymatrix{
\overset{-1}{\underset{1} {\circ}}\ar @{-}[d]_{-1} \ar @{-}[rr]^{\widetilde{q}_{13}} & & \overset{q_{33}} {\underset{3} {\circ}}\ar @{-}[d]^{\widetilde{q}_{13}} \\
\overset{q_{22}}{\underset{2} {\circ}} \ar @{-}[rr]^{-1} && \overset{-1}{\underset{\tresdos} {\circ}\,.} }
$$
There are no cycles of length $4$ in \cite[Table 3]{H-classif-RS}, so $\GK \toba(V) = \infty$ by Theorem~\ref{thm:conj-AAH}.
\end{proof}

\section{Two blocks}\label{sec:2blocks}
In this section, we consider $V\!\in \yd{\Bbbk\Gamma}$ satisfying Hypothesis~\ref{hypo:recargada}
with $\theta = 2$, ${\dim V_1= \dim V_2 = 2}$ so that $V_1$ is pale, and $V_2$ is either a pale block or a block.
Let $g_i\in\Gamma$ such that $V_i \subset V_{g_i}$, $i\in \I_2$.
We fix bases $\big\{x_i, x_{i+\frac{1}{2}}\big\}$ of $V_i$, $i\in \I_2$,
such that there exist $q_{ij} \in \Bbbk^{\times}$, $i,j\in \I_2$, and $a,b\in \Bbbk$ satisfying
\begin{alignat*}{2}
&g_i\cdot x_j = q_{ij}x_j, \quad i,j\in \I_2,\qquad&& g_1\cdot x_{\tresdos} = q_{11}x_{\tresdos},
\\
&g_2\cdot x_{\tresdos} = q_{21}\big(x_{\tresdos}+x_1\big),&&
g_1\cdot x_{\cincodos} = q_{12}\big(x_{\cincodos} + a x_{2}\big),\qquad
g_2\cdot x_{\cincodos} = q_{22}\big(x_{\cincodos} + b x_{2}\big).
\end{alignat*}

Thus the braiding of $V$ is determined by the matrix $\bq = (q_{ij})_{i,j \in \I_{2}}$
and the scalars~$a$,~$b$.
Again we consider some special cases; for $q \in \Bbbk^{\times}$ we set
\begin{alignat}{5}\label{eq:selene-defi-20}
&\sele_{2,0}(q) &&\text{where} &&q_{11} = -1 = q_{22}, && q_{12} = q = q_{21}^{-1}, \quad&& a=1, \ b=0,
\\ \label{eq:selene-defi-1+}
&\sele_{1,+}(q, a) \qquad&&\text{where}\quad &&q_{11} = -1 =- q_{22}, \quad&& q_{12} = q = q_{21}^{-1}, && b=1,
\\ \label{eq:selene-defi-1-}
&\sele_{1,-}(q) &&\text{where} &&q_{11} = -1 = q_{22}, && q_{12} = q = q_{21}^{-1}, && a= -1, \ b=1.
\end{alignat}
The diagrams of $\sele_{2,0}(q)$, $\sele_{1,+}(q, a)$ and $\sele_{1,-}(q)$ are respectively
$$
\xymatrix{
\underset{1}{\begin{tikzpicture}
\draw[dashed] (0,0) rectangle (0.36,0.36);
\end{tikzpicture}} \ar@{.}^{(1,0)}[r]
&\underset{2}{\begin{tikzpicture}
\draw[dashed] (0,0) rectangle (0.36,0.36);
\end{tikzpicture}}\, ,}
\qquad
\xymatrix{
\underset{1}{\begin{tikzpicture}
\draw[dashed] (0,0) rectangle (0.36,0.36);
\end{tikzpicture}} \ar @{.}^{(a,1)}[rr]
& & \underset{2}{\boxplus}\,,}
\qquad
\xymatrix{
\underset{1}{\begin{tikzpicture}
\draw[dashed] (0,0) rectangle (0.36,0.36);
\end{tikzpicture}} \ar @{.}^{(1,1)}[r]
& \underset{2}{\boxminus}\,.}
$$
The dotted line means that
$\widetilde{q}_{12}=1$ and is labeled by the pair $(a,b)$.

Here is the main result of this section.

\begin{Theorem} \label{th:2-blocks}
The algebra $\toba (V)$ has finite $\GK$ if and only if $V$ is isomorphic
either to $\sele_{2,0}(q)$, or to $\sele_{1,+}(q, a)$ with $a \in \bigl\{-1, -\frac12\bigr\}$, or to $\sele_{1,-}(q)$ for some $q\in \Bbbk^{\times}$.
\end{Theorem}

Let us overview the proof.
We show that $\toba(\sele_{2,0}(q))$, $\toba\big(\sele_{1,+}\big(q,-\frac{1}{2}\big)\big)$, $\toba(\sele_{1,+}(q,-1))$ and $\toba(\sele_{1,-}(q))$
have finite $\GK$ in Theorems~\ref{th:selene-2-0}, \ref{th:selene-1+} and~\ref{th:selene-1-}.

Suppose then that $\GK \toba(V) < \infty$.
Since $V_1 \oplus \Bbbk x_2$ is a braided subspace with braiding~\eqref{eq:braiding-paleblock-point} up to reindexing,
by Theorem~\ref{th:paleblock-point-resumen}
we may assume that $q_{11} = -1$ and
that either $\widetilde{q}_{12}=1$ (we say that the interaction is \emph{weak}) and $q_{22}=\pm 1$
or else $\widetilde{q}_{12}=-1$ (the interaction is \emph{mild}) and $q_{22}=-1$,
which is discarded in Proposition~\ref{prop:mild-2blocks}.
So we assume that the interaction is weak.

Now $V_2$ is a pale block if and only if $b=0$. In this case, we may assume that $a =1$ after normalizing $x_2$.
By Theorem~\ref{th:paleblock-point-resumen} applied to the braided subspace $\Bbbk x_1\oplus V_2$, $\GK \toba(V) =\infty$ if $q_{22}=1$.
Hence we assume $q_{22}=-1$. That is, we are left with the braided vector space $\sele_{2,0}(q)$ with $q= q_{12}$.

Next we assume that $V_2$ is a block, that is $b\neq 0$; up to normalization, we may assume that $b= q_{22}$.
As in \cite{AAH-triang}, it is convenient to consider the ghost
\begin{equation}\label{eq:discrete-ghost}
\ghost := \begin{cases} -2a, & q_{22} = 1,
\\
\phantom{-2}a, & q_{22} = -1.
\end{cases}
\end{equation}

The subspace $\big\langle x_1, x_2, x_\cincodos \big\rangle$ is of the form
\emph{one block and one point}. Therefore,
by \cite[Lem\-ma~4.2.3]{AAH-triang},
$\GK \toba(V) =\infty$ if $\ghost \notin \N_0$.
Hence we assume that $\ghost\in \N_0$. Then we discard $\ghost = 0$ in
Proposition~\ref{prop:pale+Jordan-discarding} and $\ghost \neq 1$ in Theorems~\ref{th:selene-1+} and~\ref{th:selene-1-}.

\subsection{Mild interaction} 

We show that this implies infinite $\GK$.

\begin{Proposition}\label{prop:mild-2blocks}
If $\widetilde{q}_{12} = -1$, then $\GK \toba(V)=\infty$.
\end{Proposition}
\begin{proof}
Here, $0 \subset \langle x_1\rangle \subset \big\langle x_1, x_{\tresdos} \big\rangle \subset \big\langle x_1, x_{\tresdos}, x_2 \big\rangle \subset V$ is a flag of Yetter--Drinfeld submodules such that $\gr V$ is of diagonal type; its diagram is
$$
\xymatrix{
\overset{-1}{\underset{1} {\circ}}\ar @{-}[d]_{-1} \ar @{-}[rr]^{-1} & & \overset{-1} {\underset{2} {\circ}}\ar @{-}[d]^{-1} \\
\overset{-1}{\underset{\cincodos} {\circ}} \ar @{-}[rr]^{-1} && \overset{-1}{\underset{\tresdos} {\circ}},}
$$
thus $\gr V$ is of affine Cartan type, so $\GK \toba(V) = \infty$ by \cite{AAH-diag} and
\cite[Lemma~3.4.2\,(c)]{AAH-triang}.
\end{proof}

\subsection{Two pale blocks, weak interaction} Recall the Selene braided vector space $\sele_{2,0}(q)$ defined in \eqref{eq:selene-defi-20}.

\begin{Theorem} \label{th:selene-2-0}
The algebra $\toba(\sele_{2,0}(q))$ is generated by $x_1$, $x_{\tresdos}$, $x_2$, $x_\cincodos$ with defining relations
\begin{alignat}{3}
\label{eq:endymion-III-1}
& x_i^2 = x_{i+\tfrac{1}{2}}^2=0, \qquad&&\hspace{-20mm}
x_ix_{i+\tfrac{1}{2}} =-x_{i+\tfrac{1}{2}} x_i, \qquad &&i\in\I_2;
\\ \label{eq:endymion-III-2}
&x_2 x_1=q_{21}x_1 x_2, \qquad&&\hspace{-20mm}
x_2 x_{\tresdos2} = -q_{21} x_{\tresdos2}x_2, \qquad
&&x_1 x_{\cincodos1} = -q_{12} x_{\cincodos1}x_1,
\\ \label{eq:two-pale-def-rel-1}
&x_{1\cincodos} = x_{\tresdos2},
\\
\label{relations-two-pale-blocks-4}
&x_{\tresdos\cincodos}x_{2} = -q_{12} x_{2}x_{\tresdos\cincodos},
\\ \label{relations-two-pale-blocks-5}
& x_{\tresdos\cincodos}x_{\cincodos} = -q_{12} x_{\cincodos}x_{\tresdos\cincodos} - q_{12}x_{2}x_{\tresdos\cincodos},
\\ \label{relations-two-pale-blocks-6}
& x_{\tresdos\cincodos}x_{\tresdos2} = x_{\tresdos2}x_{\tresdos\cincodos} - x_{\tresdos2}^2.
\end{alignat}

A PBW-basis is formed by the monomials
\begin{align}\label{eq:two-pale-PBW-basis}
x_{\cincodos}^{m_1}x_{\tresdos\cincodos}^{m_2}x_{2}^{m_3}x_{\tresdos2}^{m_4} x_{\tresdos}^{m_5} x_1^{m_6},\qquad
 m_1, m_3, m_5, m_6 \in \{0,1\}, \quad m_2, m_4 \in \N_0.
\end{align}
Hence $\GK \toba(\sele_{2,0}(q)) = 2$.
\end{Theorem}

\begin{proof} We proceed by steps.

\begin{stepiii}
As $V_1 \oplus \Bbbk x_2 \simeq \eny_-(q)$ and $V_2 \oplus \Bbbk x_1 \simeq \eny_-\big(q^{-1}\big)$, the relations \eqref{eq:endymion-III-1}, \eqref{eq:endymion-III-2}
hold in $\toba(V)$.
\end{stepiii}

Next we focus on the Nichols algebra $\K = \toba\left(\adc(\toba(V_1))(V_2)\right)$.

\begin{stepiii}\quad
\begin{enumerate}[label=$(\alph*)$]
\item\label{item:two-pale-basis-K1-relation}
The relation \eqref{eq:two-pale-def-rel-1} holds in $\toba(V)$.
\item\label{item:two-pale-basis-K1}
The set $\big\{x_{2}, x_{\tresdos2}, x_{\cincodos}, x_{\tresdos\cincodos} \big\}$ is a basis of $\K^1$.
\end{enumerate}
\end{stepiii}

\begin{proof}[Proof of Step~2]
For~\ref{item:two-pale-basis-K1-relation}, we compute $\partial_i\big(x_{1\cincodos}\big)= \partial_i\big(x_{\tresdos2}\big)=0$, if $i\in\big\{1,\tresdos\big\}$,
\begin{align*}
\partial_2\big(x_{1\cincodos}\big) = \partial_2\big(x_{\tresdos2}\big)=-x_1, \qquad
 \partial_{\cincodos}\big(x_{1\cincodos}\big) = \partial_{\cincodos}\big(x_{\tresdos2}\big)=0.
\end{align*}

For~\ref{item:two-pale-basis-K1}, we use~\ref{item:two-pale-basis-K1-relation}, \eqref{eq:endymion-III-1} and \eqref{eq:endymion-III-2} to check that
\begin{align*}
x_{12} = x_{1\tresdos2} = x_{1\tresdos\cincodos} = 0, \qquad x_{1\cincodos} = x_{\tresdos2}.
\end{align*}
As $\toba(V_1) = \bigwedge(V_1)$, $\K^1 = \adc(\toba(V_1))(V_2)$ is spanned by
$\big\{x_{2}, x_{\tresdos2}, x_{\cincodos}, x_{\tresdos\cincodos} \big\}$. Also, $\partial_i$ annihilates each element of this set if $i=1,\tresdos$, and
\begin{alignat*}{4}
&\partial_2(x_{2}) = 1,\qquad &&\partial_2\big(x_{\cincodos}\big) = 0, \qquad &&\partial_2\big(x_{\tresdos2}\big) = -x_1, \qquad &&\partial_2\big(x_{\tresdos\cincodos}\big) = -\big(x_\tresdos+x_1\big),
\\
&\partial_\cincodos(x_{2}) = 0, &&\partial_\cincodos\big(x_{\cincodos}\big) = 1, &&\partial_\cincodos\big(x_{\tresdos2}\big) = 0, &&\partial_\cincodos\big(x_{\tresdos\cincodos}\big) = -x_1.
\end{alignat*}
Thus $\big\{x_{2}, x_{\tresdos2}, x_{\cincodos}, x_{\tresdos\cincodos} \big\}$ is linearly independent.
\end{proof}

We shall need the action of $g_2$ in $\K^1$: $g_2\cdot x_{2} = -x_{2}$,
$g_2\cdot x_{\cincodos} = -x_{\cincodos}$,
\begin{align*}
g_2 \cdot x_{\tresdos2} = -q_{21} x_{\tresdos2} , \qquad
g_2 \cdot x_{\tresdos\cincodos} = -q_{21} \big(x_{\tresdos\cincodos}+x_{\tresdos2}\big).
\end{align*}

\begin{stepiii}
Let $\wtoba$ be an algebra with elements $x_1$, $x_{\tresdos}$, $x_2$, $x_\cincodos$ satisfying
\eqref{eq:endymion-III-1}, \eqref{eq:endymion-III-2} and \eqref{eq:two-pale-def-rel-1}.
Then the following relations also hold in $\wtoba$:
\begin{align}\label{eq:two-pale-def-rel-other-1}
x_1x_{\tresdos\cincodos} = -q_{12}\big(x_{\tresdos\cincodos}+x_{\tresdos2}\big)x_1,\qquad
x_{\tresdos}x_{\tresdos2}=-q_{12}x_{\tresdos2}x_{\tresdos}.
\end{align}
In particular, \eqref{eq:two-pale-def-rel-other-1} holds in $\toba(V)$. The verification is straightforward.
\end{stepiii}

\begin{stepiii}
The relations \eqref{relations-two-pale-blocks-4}, \eqref{relations-two-pale-blocks-5} and \eqref{relations-two-pale-blocks-6}
hold in $\toba(V)$.
\end{stepiii}

\begin{proof}[Proof of Step~4]
As $\partial_1$ and $\partial_{\tresdos}$ annihilate $x_{2}$, $x_{\cincodos}$, $x_{\tresdos2}$, $x_{\tresdos\cincodos}$,
it suffices to check that $\partial_2$ and $\partial_\cincodos$ annihilate each of these relations. For \eqref{relations-two-pale-blocks-4} and \eqref{relations-two-pale-blocks-5},
\begin{align*}
&\partial_2 \big( x_{\tresdos\cincodos}x_{2} +q_{12} x_{2}x_{\tresdos\cincodos} \big) =
x_{\tresdos\cincodos}+\big(x_\tresdos+x_1\big)x_{2}
 -q_{12} x_{2}\big(x_\tresdos+x_1\big) -\big(x_{\tresdos\cincodos}+x_{\tresdos2}\big)=0,
\\
&\partial_{\cincodos} \big( x_{\tresdos\cincodos}x_{2} +q_{12} x_{2}x_{\tresdos\cincodos} \big) =
x_1x_{2} -q_{12} x_{2}x_1=0,
\\
&\partial_2 \big( x_{\tresdos\cincodos}x_{\cincodos}+q_{12} x_{\cincodos}x_{\tresdos\cincodos} + q_{12}x_{2}x_{\tresdos\cincodos} \big)
\\ &\qquad
{}=\big(x_\tresdos+x_1\big)x_{\cincodos}-q_{12} x_{\cincodos}(x_\tresdos+x_1)
 - q_{12}x_{2}\big(x_\tresdos+x_1\big)-\big(x_{\tresdos\cincodos}+x_{\tresdos2}\big)
 \\ &\qquad
= x_{\tresdos\cincodos}+x_{1\cincodos}-x_{\tresdos\cincodos}-x_{\tresdos2}=0,
\\
 &\partial_{\cincodos} \big( x_{\tresdos\cincodos}x_{\cincodos}\!+q_{12} x_{\cincodos}x_{\tresdos\cincodos} \!+ q_{12}x_{2}x_{\tresdos\cincodos} \big)=
x_1x_{\cincodos}
\!+x_{\tresdos\cincodos}
-\!\big(x_{\tresdos\cincodos}\!+x_{\tresdos2}\big)\!
-q_{12} x_{\cincodos}x_1 -q_{12}x_{2}x_1=0.
\end{align*}
Using \eqref{eq:endymion-III-1}, \eqref{eq:endymion-III-2}, \eqref{eq:two-pale-def-rel-1} and
\eqref{eq:two-pale-def-rel-other-1}, we see finally that \eqref{relations-two-pale-blocks-6} also holds:
\begin{align*}
&\begin{aligned}
\partial_2 \big( x_{\tresdos\cincodos}x_{\tresdos2}-x_{\tresdos2}x_{\tresdos\cincodos} + x_{\tresdos2}^2 \big) ={}&
-x_{\tresdos\cincodos}x_1+q_{21}\big(x_{\tresdos}+x_1\big)x_{\tresdos2}
+x_{\tresdos2}\big(x_{\tresdos}+x_1\big)
\\
& -q_{21}x_1\big(x_{\tresdos\cincodos}+x_{\tresdos2}\big)
- x_{\tresdos2}x_1 +q_{21}x_1x_{\tresdos2}=0,
\end{aligned}
\\
&\partial_{\cincodos} ( x_{\tresdos\cincodos}x_{\tresdos2}-x_{\tresdos2}x_{\tresdos\cincodos} + x_{\tresdos2}^2 ) = q_{21}x_1x_{\tresdos2}+x_{\tresdos2}x_1=0. \tag*{\qed}
\end{align*}
\renewcommand{\qed}{}
\end{proof}

Let $\toba$ be the algebra with the claimed presentation. By the previous steps,
there is a~surjective map $\toba \to \toba(\sele_{2,0}(q))$.
Now $\toba$ is spanned by the monomials \eqref{eq:two-pale-PBW-basis} because of
the defining relations, \eqref{eq:two-pale-def-rel-other-1} and
\begin{align*}
x_{\tresdos2}x_{\cincodos} &= -q_{12} x_{\cincodos}x_{\tresdos2} - q_{12}x_{2}x_{\tresdos2},
\end{align*}
that follows from \eqref{relations-two-pale-blocks-4} and \eqref{eq:endymion-III-1}.\
To prove that the monomials in \eqref{eq:two-pale-PBW-basis} form a basis of $\toba$ and that
$\toba \simeq \toba(\sele_{2,0}(q))$, it suffices to prove that these monomials are linearly independent in $\toba(\sele_{2,0})$. By direct computations,
\begin{gather*}
\partial_1 \big(x_{\cincodos}^{m_1}x_{\tresdos\cincodos}^{m_2}x_{2}^{m_3}x_{\tresdos2}^{m_4} x_{\tresdos}^{m_5} x_1^{m_6}\big)
= \delta_{m_6,1} x_{\cincodos}^{m_1}x_{\tresdos\cincodos}^{m_2}x_{2}^{m_3}x_{\tresdos2}^{m_4} x_{\tresdos}^{m_5},
\\
\partial_{\tresdos} \big(x_{\cincodos}^{m_1}x_{\tresdos\cincodos}^{m_2}x_{2}^{m_3}x_{\tresdos2}^{m_4} x_{\tresdos}^{m_5} \big)
= \delta_{m_5,1} x_{\cincodos}^{m_1}x_{\tresdos\cincodos}^{m_2}x_{2}^{m_3}x_{\tresdos2}^{m_4},
\\
\partial_1\partial_{\cincodos} \big(x_{\cincodos}^{m_1}x_{\tresdos\cincodos}^{m_2}x_{2}^{m_3}x_{\tresdos2}^{m_4}\big)
\in m_2(-q_{12})^{m_3}x_{\cincodos}^{m_1}x_{\tresdos\cincodos}^{m_2-1}x_{2}^{m_3}x_{\tresdos2}^{m_4}
 + \sum_{k\ge1} \Bbbk x_{\cincodos}^{m_1}x_{\tresdos\cincodos}^{m_2-1-k}x_{2}^{m_3}x_{\tresdos2}^{m_4+k},
\\
\partial_1\partial_{2} \big(x_{\cincodos}^{m_1}x_{2}^{m_3}x_{\tresdos2}^{m_4}\big)
= -m_4x_{\cincodos}^{m_1}x_{2}^{m_3} x_{\tresdos2}^{m_4-1}.
\end{gather*}
The claim is established by a recursive argument as in previous proofs.
\end{proof}

\subsection{A pale block and a block, weak interaction}
In this subsection we assume that $q_{11}=- 1$,
$\widetilde{q}_{12} = 1$, $q_{22}=\pm 1=b$ and $\ghost \in \N_0$.

\subsubsection{The vanishing ghost}
We discard here the possibility $a=0$. We start by a lemma that is also useful later
when dealing with a Jordan or a super Jordan plane, i.e., $q_{22}= 1$ or $-1$.

{\samepage\begin{Lemma}\label{lem:pale+Jordan-basis-K1}
Let $\K = \toba(\K^1)$, where $\K^1 = \adc(\toba(V_1))(V_2)$.
\begin{enumerate}[label=$(\alph*)$]\itemsep=0pt
\item\label{item:pale+Jordan-relation-endy} The relations
\eqref{eq:endymion-1} and \eqref{eq:endymion-1b}
hold in $\toba(V)$.
\item\label{item:pale+Jordan-basis-K1-relation} The following relation holds in $\toba(V)$:
\begin{align}\label{eq:pale+Jordan-def-rel-1}
x_{1\cincodos} &= a x_{\tresdos2}.
\end{align}
\item\label{item:pale+Jordan-basis-K1}
The set $\big\{x_{2}, x_{\tresdos2}, x_{\cincodos}, x_{\tresdos\cincodos}\big\}$ is a basis of $\K^1$.
\end{enumerate}
\end{Lemma}}

\begin{proof} Item \ref{item:pale+Jordan-relation-endy} follows since $V_1 \oplus \Bbbk x_2\simeq \eny_{\pm}$.
For~\ref{item:pale+Jordan-basis-K1-relation}, we compute
\begin{align*}
\partial_i\big(x_{1\cincodos}\big) = \partial_i\big(x_{\tresdos2}\big)=0, \quad
i \in\big\{1,\tfrac{3}{2}, \tfrac{5}{2}\big\},\qquad
\partial_2\big(x_{\tresdos2}\big) =-x_1,\qquad
\partial_2\big(x_{1\cincodos}\big) = -ax_1.
\end{align*}

For~\ref{item:pale+Jordan-basis-K1}, we use~\ref{item:pale+Jordan-basis-K1-relation}, \eqref{eq:endymion-1} and \eqref{eq:endymion-1b} to check that
$x_{12}= 0$, $x_{1\tresdos2} = x_{1\tresdos\cincodos}=0$. As $\toba(V_1) = \bigwedge(V_1)$, $\K^1 = \adc(\toba(V_1))(V_2)$ is spanned by
$\big\{x_{2}, x_{\tresdos2}, x_{\cincodos}, x_{\tresdos\cincodos} \big\}$. Also, $\partial_i$ annihilates each element of this set if $i=1,\tresdos$, and
\begin{alignat*}{4}
&\partial_2(x_{2}) = 1,\qquad && \partial_2\big(x_{\cincodos}\big) = 0, \qquad &&\partial_2\big(x_{\tresdos2}\big) = -x_1, \qquad &&\partial_2\big(x_{\tresdos\cincodos}\big) = -a\big(x_\tresdos+x_1\big),
\\
&\partial_\cincodos(x_{2}) = 0, && \partial_\cincodos\big(x_{\cincodos}\big) = 1, &&\partial_\cincodos\big(x_{\tresdos2}\big) = 0, &&\partial_\cincodos\big(x_{\tresdos\cincodos}\big) = -x_1.
\end{alignat*}
Thus $\big\{x_{2}, x_{\tresdos2}, x_{\cincodos}, x_{\tresdos\cincodos} \big\}$ is linearly independent.
\end{proof}

\begin{Proposition}\label{prop:pale+Jordan-discarding}
If $a=0$, then $\GK \toba(V)=\infty$.
\end{Proposition}

\begin{proof}
The coaction of $\mathcal{K}^1$ satisfies
\begin{align*}
\delta(x_i) = g_2 \otimes x_i,\qquad
\delta(x_{\tresdos i}) = g_1g_2 \otimes x_{\tresdos2} - x_1g_2 \otimes x_i,\qquad
 i\in\big\{2,\tfrac{5}{2}\big\}.
\end{align*}
Set $y_1=x_2$, $y_2=x_{\cincodos}$, $y_3=x_{\tresdos2}$. Then $\{y_1, y_2, y_3\}$ is a braided vector subspace of $\K^1$, and the braiding is given by
\begin{align*}
(c(y_i \otimes y_j))_{i,j\in\I_3}=\begin{pmatrix}
q_{22} y_1 \otimes y_1& (q_{22} y_2 + y_1) \otimes y_1 & q_{21}q_{22} y_3 \otimes y_1
\\
q_{22} y_1 \otimes y_2 & (q_{22} y_2 + y_1) \otimes y_2 & q_{21}q_{22} y_3 \otimes y_2
\\
q_{12}q_{22} y_1 \otimes y_3 & q_{12}q_{22} (y_2 + q_{22}y_1) \otimes y_3 & -q_{22} y_3 \otimes y_3
\end{pmatrix}\!.
\end{align*}
This corresponds to \emph{one block and one point} with negative ghost, so by \cite[Theorem~4.1.1]{AAH-triang},
we have $\GK \K=\infty$.
Thus $\GK \toba(V)=\infty$.
\end{proof}

\subsubsection{A pale block and a Jordan plane}\label{subsubsec:pale-jordan}
Here we assume that $q_{11} = -1$, $q_{22}=1$, $q_{12} = q = q_{21}^{-1}$, $b=1$ and $\ghost =-2a \in \N$, cf. \eqref{eq:discrete-ghost}.
When $\ghost= 1$, respectively $\ghost=2$, $V$
is the braided vector space $\sele_{1,+}\big(q, -\frac12\big)$, respectively $\sele_{1,+}(q,-1)$, see \eqref{eq:selene-defi-1+}.
To state our result we need the elements
\begin{align}\label{eq:pale-Jordan-def-w-t}
\ttt= x_{\tresdos\cincodos}x_{\cincodos}+q_{12}x_{\cincodos}x_{\tresdos\cincodos},\qquad
\wt= x_{\tresdos\cincodos}x_{2}+q_{12}x_{2}x_{\tresdos\cincodos}.
\end{align}

\begin{Theorem} \label{th:selene-1+}
The algebra $\toba(V)$ has finite $\GK $ if and only if $\ghost \le 2$.
\begin{itemize}
\item If $\ghost = 1$, then $\toba\big(\sele_{1,+}\big(q, -\frac12\big)\big)$ is presented by generators
$x_1$, $ x_{\tresdos}$, $ x_2$, $x_\cincodos$ with defining relations
\eqref{eq:endymion-1}, \eqref{eq:endymion-1b}, \eqref{eq:endymion-2}, \eqref{eq:pale+Jordan-def-rel-1} and
\begin{alignat}{2}\label{eq:pale+Jordan-def-rel-6}
&x_{\tresdos\cincodos}x_2 =q_{12}x_2x_{\tresdos\cincodos},
&&x_{\tresdos\cincodos}x_{\cincodos} =q_{12}\big(x_{\cincodos}+\tfrac{1}{2}x_{2}\big)x_{\tresdos\cincodos},
\\
\label{eq:pale+Jordan-def-rel-4}
&x_{\tresdos2}^2 =x_{\tresdos\cincodos}^2=0,
&&x_{\tresdos\cincodos}x_{\tresdos2} =-x_{\tresdos2}x_{\tresdos\cincodos},
\\
\label{eq:pale+Jordan-def-rel-5}
&x_{\cincodos}x_{2} =x_{2}x_{\cincodos}-\tfrac{1}{2}x_2^2.\qquad
\end{alignat}

A PBW-basis is formed by the monomials
\begin{align}\label{eq:selene-1+-PBW-basis}
x_{\cincodos}^{m_1}x_{2}^{m_2}x_{\tresdos\cincodos}^{m_3}x_{\tresdos2}^{m_4} x_{\tresdos}^{m_5} x_1^{m_6} ,\qquad
 m_3, m_4, m_5, m_6 \in \{0,1\}, \quad m_1, m_2 \in \N_0.
\end{align}
Hence $\GK \toba\big(\sele_{1,+}\big(q, -\frac12\big) \big)= 2$.

\item
If $\ghost = 2$, then $\toba(\sele_{1,+}(q, -1))$ is presented by generators
$x_1$, $x_{\tresdos}$, $x_2$, $x_\cincodos$ with defining relations
\eqref{eq:endymion-1}, \eqref{eq:endymion-1b}, \eqref{eq:endymion-2}, \eqref{eq:pale+Jordan-def-rel-1}, \eqref{eq:pale+Jordan-def-rel-5} and
\begin{alignat}{2}
\label{eq:pale+Jordan-def-rel-7}
&x_{\tresdos2}x_{\cincodos} =q_{12}x_{\cincodos}x_{\tresdos2}+\wt,
&&\ttt x_{\cincodos} =q_{12}\big(x_{\cincodos}+x_{2}\big)\ttt,
\\\label{eq:pale+Jordan-def-rel-8}
&x_{\tresdos2} \ttt =-q_{12}(\ttt-\wt)x_{\tresdos2},
&&\wt x_{\cincodos} =q_{12}\big(x_{\cincodos}+x_{2}\big)\wt,
\\ \label{eq:pale+Jordan-def-rel-9}
&x_{\tresdos\cincodos} \ttt =-q_{12}(\ttt-\wt)x_{\tresdos\cincodos}.\qquad
\end{alignat}

A PBW-basis is formed by the monomials
\begin{align}\label{eq:selene-1++-PBW-basis}
\begin{split}
&x_{2}^{m_1}x_{\cincodos}^{m_2} \wt^{m_3}\ttt^{m_4} x_{\tresdos2}^{m_5}\xt^{m_6}x_{\tresdos\cincodos}^{m_7} x_{\tresdos}^{m_8} x_1^{m_9},
\\
&m_3, m_4, m_5, m_8,m_9 \in \{0,1\}, \quad m_1, m_2,m_6,m_7 \in \N_0.
\end{split}
\end{align}
Hence $\GK \toba(\sele_{1,+}(q, -1)) = 4$.
\end{itemize}
\end{Theorem}

\begin{proof} We start by observing that
\begin{align*}
g_1\cdot x_{\tresdos\cincodos} = -q_{12}\big(x_{\tresdos\cincodos}+ax_{\tresdos2}\big),\qquad
g_2\cdot x_{\tresdos\cincodos} = q_{21}\big(x_{\tresdos\cincodos}+(a+1)x_{\tresdos2}\big),
\end{align*}
where we used \eqref{eq:pale+Jordan-def-rel-1}.
From these equalities and Lemma~\ref{lem:pale+Jordan-basis-K1}, we get
\begin{align}\label{eq:pale+Jordan-g=1-subspace}
\begin{aligned}
&c(x_{\tresdos i} \otimes x_{\tresdos2}) = -x_{\tresdos2} \otimes x_{\tresdos i}, \\
&c(x_{\tresdos i} \otimes x_{\tresdos\cincodos}) = (-x_{\tresdos\cincodos}+(\ghost-1)x_{\tresdos2}) \otimes x_{\tresdos i},\qquad i = 2,\tfrac{5}{2}.
\end{aligned}
\end{align}
Then $\langle x_{\tresdos\cincodos},x_{\tresdos2}\rangle$ is a braided vector subspace of $\K^1$.

\begin{stepiv}
Assume that $\ghost \neq 1$. Then the Nichols algebra of $\big\langle x_{\tresdos\cincodos},x_{\tresdos2}\big\rangle$ is
isomorphic to the super Jordan plane. Set $\xt=x_{\tresdos\cincodos}x_{\tresdos2}+x_{\tresdos2}x_{\tresdos\cincodos}$. Then
\begin{align}\label{eq:pale+Jordan-def-rel-3}
x_{\tresdos2}^2 =0, \qquad
x_{\tresdos\cincodos}\xt = \xt x_{\tresdos\cincodos} +(\ghost-1) x_{\tresdos2}\xt,
\end{align}
and $\big\{x_{\tresdos2}^a \xt^b x_{\tresdos\cincodos}^c \mid a\in\{0,1\}, \, b,c \in\N_0\big\}$
is a basis of $\Bbbk \big\langle x_{\tresdos2}, x_{\tresdos\cincodos}\big\rangle$.
\end{stepiv}

\begin{stepiv}\label{lem:pale+Jordan-K1-action-derivations}
We define $\wt_n\in\toba(V)$ recursively by
$\wt_0 := x_2$ and
\begin{align*}
\wt_{n+1} := \big[x_{\tresdos\cincodos},\wt_n\big]_c =x_{\tresdos\cincodos}\wt_n -(g_1g_2\cdot \wt_n) x_{\tresdos\cincodos},\qquad
 n\in\N.
\end{align*}
We also define scalars $\ab_n$, $\bb_n$ by $\ab_n := \prod_{j=1}^n \big((\ghost-1)j - \frac{1}{2}\ghost\big)$, $\bb_0=1$ and
\begin{align*}
\bb_{n} &= \begin{cases}
- \frac{1}{2}\ghost\ab_{k-1} +\bb_{2k-1}, & n=2k,
\\
\bb_{2k} \big(k(\ghost-1) + \frac{1}{2}\ghost -1\big) + \frac{1}{2}\ghost\ab_k, & n=2k+1.
\end{cases}
\end{align*}

Then we have
\begin{align}
\label{eq:pale+Jordan-wn-adj-action}
&[x_{\tresdos2},\wt_n]_c= [x_1,\wt_n]_c =[\xt,\wt_n]_c =0,
\\
\label{eq:pale+Jordan-wn-adj-action-2}
&[x_{\tresdos},\wt_n]_c= \begin{cases}
q_{12}^{2k} \ab_k x_{\tresdos2}\xt^{k}, & n=2k, \\
-q_{12}^{2k+1}\ab_k \xt^{k+1}, & n=2k+1.
\end{cases}
\\
\label{eq:pale+Jordan-action-wn}
&g_1\cdot \wt_n =(-1)^n q_{12}^{n+1} \wt_n, \qquad
g_2\cdot \wt_n = q_{21}^{n} \wt_n.
\\
\label{eq:pale+Jordan-derivations-wn}
&\partial_{2}(\wt_n)=\begin{cases}
\bb_{2k} \xt^{k}, & n=2k, \\
\bb_{2k+1} x_{\tresdos2}\xt^{k}, & n=2k+1;
\end{cases} \qquad \partial_{i}(\wt_n) =0, \quad i=1,\tfrac{3}{2},\tfrac{5}{2}.
\end{align}
\end{stepiv}

\begin{proof}[Proof of Steps 1 and 2]
We proceed recursively on $n\in\N_0$.
When $n=0$ \eqref{eq:pale+Jordan-action-wn} and \eqref{eq:pale+Jordan-derivations-wn} are clear. For \eqref{eq:pale+Jordan-wn-adj-action} and \eqref{eq:pale+Jordan-wn-adj-action-2} we compute
\begin{align*}
\partial_2(\xt) = -(\ghost-1)x_{\tresdos2}x_1, \qquad
\partial_i(\xt) = 0, \quad i\in\big\{1,\tfrac{3}{2},\tfrac{5}{2}\big\}.
\end{align*}
Using \eqref{eq:endymion-1b} and \eqref{eq:endymion-2}, we check that
\begin{align*}
\partial_2 \big(\xt x_2-q_{12}^2 x_2\xt\big) = \xt -(\ghost-1)x_{\tresdos2}x_1 x_2 -\xt +q_{12}^2(\ghost-1) x_2x_{\tresdos2}x_1=0.
\end{align*}
As $\partial_i(\xt) = \partial_i(x_2) = 0$, $i\in\{1,\tresdos,\cincodos\}$, the relation
$[\xt, x_2]_c=0$ holds in $\K$. Now $[x_1,x_2]_c=[x_{\tresdos2},x_2]_c=0$ are \eqref{eq:endymion-1b} and \eqref{eq:endymion-2}, respectively, and \eqref{eq:pale+Jordan-wn-adj-action-2} follows.
Now assume that all equations hold for $n$. By the inductive hypothesis,
\begin{align*}
&g_1\cdot \wt_{n+1} = \bigl[-q_{12}\big(x_{\tresdos\cincodos}+ax_{\tresdos2}\big), (-1)^n q_{12}^{n+1} \wt_n\big]_c =(-1)^{n+1} q_{12}^{n+2} \wt_{n+1},
\\
&g_2\cdot \wt_{n+1} = \big[q_{21}\big(x_{\tresdos\cincodos}+(a+1)x_{\tresdos2}\big), q_{21}^{n} \wt_n\big]_c = q_{21}^{n+1} \wt_{n+1},
\end{align*}
where we used \eqref{eq:graded-bracket-jacobi}, \eqref{eq:pale+Jordan-def-rel-3}. Thus \eqref{eq:pale+Jordan-action-wn} is proved.
Next we establish \eqref{eq:pale+Jordan-wn-adj-action}:
\begin{align*}
&\big[x_{\tresdos2},\wt_{n+1}\big]_c = [\xt,\wt_n]_c +\big(g_1g_2 \cdot x_{\tresdos\cincodos}\big)\big[x_{\tresdos2},\wt_n\big]_c \pm q_{12} \big[x_{\tresdos2},\wt_n\big]_c x_{\tresdos\cincodos}=0,
\\
&\big[x_1,\wt_{n+1}\big]_c = \big[\big[x_1,x_{\tresdos\cincodos}\big]_c,\wt_n\big]_c +(g_1 \cdot x_{\tresdos\cincodos})[x_1,\wt_n]_c \pm q_{12} [x_1,\wt_n]_c x_{\tresdos\cincodos}=0,
\\
&[\xt,\wt_{n+1}]_c = (\ghost-1)\big[x_{\tresdos2}\xt,\wt_n\big]_c +\big(g_1^2g_2^2 \cdot x_{\tresdos\cincodos}\big)[\xt,\wt_n]_c \pm q_{12} [\xt,\wt_n]_c x_{\tresdos\cincodos}=0.
\end{align*}
We go on with \eqref{eq:pale+Jordan-wn-adj-action-2} considering separately the cases $n$ odd or even:
\begin{align*}
&\begin{aligned}
\big[x_{\tresdos}, \wt_{2k}\big]_c &=
\big(g_1 \cdot x_{\tresdos\cincodos}\big)\big[x_{\tresdos},\wt_{2k-1}\big]_c -\big[x_{\tresdos},g_1g_2\cdot \wt_{2k-1}\big]_c x_{\tresdos\cincodos}
\\
&= -q_{12}\big(x_{\tresdos\cincodos}+ax_{\tresdos2}\big) \big[x_{\tresdos},\wt_{2k-1}\big]_c+q_{12} \big[x_{\tresdos},\wt_{2k-1}\big]_c x_{\tresdos\cincodos}
\\
&= -q_{12}\ab_{k-1} \big(\big(x_{\tresdos\cincodos}+ax_{\tresdos2}\big) \xt^{k} - \xt^{k}x_{\tresdos\cincodos}\big)
= -q_{12}\ab_{k-1} \big(k(\ghost-1)+a\big) x_{\tresdos2}\xt^{k},
\end{aligned}
\\
&\begin{aligned}
\big[x_{\tresdos}, \wt_{2k+1}\big]_c &= -q_{12}\big(x_{\tresdos\cincodos}+ax_{\tresdos2}\big) \big[x_{\tresdos},\wt_{2k}\big]_c -
q_{12} \big[x_{\tresdos},\wt_{2k}\big]_c x_{\tresdos\cincodos}
\\
&=-q_{12} \ab_k \big(\xt^{k+1} - x_{\tresdos2}\big(x_{\tresdos\cincodos}\xt^{k}
- \xt^{k} x_{\tresdos\cincodos}\big)\big) = -q_{12} \ab_k \xt^{k+1}.
\end{aligned}
\end{align*}
Now we deal with \eqref{eq:pale+Jordan-derivations-wn}. By formula~\eqref{eq:pale+Jordan-action-wn}, $\wt_{n+1}=x_{\tresdos\cincodos}\wt_n-(-1)^nq_{12} \wt_n x_{\tresdos\cincodos}$. Let $i=1,\tresdos$. As~${\partial_i\big(x_{\tresdos\cincodos}\big)=0}$, we have that $\partial_i(\wt_{n+1})=0$. Now,
\begin{align*}
\partial_{\cincodos} (\wt_{n+1}) = -q_{21}^n x_1 \wt_n+(-1)^nq_{12} \wt_n x_1
= -q_{21}^n [x_1, \wt_n]_c=0.
\end{align*}
For the last skew-derivation we consider the cases $n=2k-1$, $n=2k$:
\begin{gather*}
\partial_{2} (\wt_{2k})
= -aq_{21}^{2k-1}\big[x_{\tresdos},\wt_{2k-1}\big]_c -aq_{21}^{2k-1}[x_{1},\wt_{2k-1}]_c
+x_{\tresdos\cincodos}\partial_2(\wt_{2k-1})
\\ \hphantom{\partial_{2} (\wt_{2k})= }
{}+ \partial_2(\wt_{2k-1}) \big(x_{\tresdos\cincodos}+(a+1)x_{\tresdos2}\big)
\\ \hphantom{\partial_{2} (\wt_{2k})}
{}= a\ab_{k-1} \xt^{k} +\bb_{2k-1} \big(x_{\tresdos\cincodos}x_{\tresdos2}\xt^{k-1}+x_{\tresdos2}\xt^{k-1} \big(x_{\tresdos\cincodos}+(a+1)x_{\tresdos2}\big)\big)
\\\hphantom{\partial_{2} (\wt_{2k})}
{}= (a\ab_{k-1} +\bb_{2k-1}) \xt^{k},
\\
\partial_{2} (\wt_{2k+1}) =
-aq_{21}^{2k}\big[x_{\tresdos},\wt_{2k}\big]_c
-aq_{21}^{2k}[x_1,\wt_{2k}]_c
+x_{\tresdos\cincodos}\partial_2(\wt_{2k})
-\partial_2(\wt_{2k}) \big(x_{\tresdos\cincodos}+(a+1)x_{\tresdos2}\big)
\\ \hphantom{\partial_{2} (\wt_{2k+1})}
{} = -a\ab_k x_{\tresdos2}\xt^{k}
+\bb_{2k} \big(x_{\tresdos\cincodos}\xt^{k}-\xt^{k} x_{\tresdos\cincodos}\big)
-(a+1)\bb_{2k} \xt^{k}x_{\tresdos2}
\\ \hphantom{\partial_{2} (\wt_{2k+1})}
{} = \big(\bb_{2k} (k(\ghost-1)-a-1) -a\ab_k\big) x_{\tresdos2}\xt^{k}.
\tag*{\qed}
\end{gather*}
\renewcommand{\qed}{}
\end{proof}

\begin{stepiv}\label{prop:pale+Jordan}
If $\ghost\in\N_{>2}$, then $\GK \toba(V)=\infty$.
\end{stepiv}

\begin{proof}[Proof of Step 3]
We claim that
$\wt_n \ne 0$, $\bb_n \ne 0$, for all $n\in\N_0$.
Indeed, $\tfrac{2\ghost}{2\ghost-1}\notin\Z$, so $\ab_k\ne 0$ for all $k\in\N$.
By \eqref{eq:pale+Jordan-wn-adj-action-2}, $[x_{\tresdos},\wt_n]_c\ne 0$, so $\wt_n \ne 0$ for all $n\in\N_0$. Hence
$0\ne \partial_2(\wt_n)= -\bb_n x_{\tresdos2}^{n+1}$, so $\bb_{n}\ne 0$ for all $n\in\N_0$.

By \cite[Lemma~2.3.4]{AAH-triang}, to prove the step it is enough to show that the set
\begin{align}\label{eq:pale+Jordan-prod-wn}
 \wt_{2n_1}\wt_{2n_2}\cdots \wt_{2n_k},\qquad
 k\in\mathbb N_0, \quad n_1< \cdots< n_k \in\N,
\end{align}
is linearly independent.
Otherwise pick a non-trivial linear combination $\mathtt S$ of elements in \eqref{eq:pale+Jordan-prod-wn}
homogeneous of minimal degree $N$. By Step~\ref{lem:pale+Jordan-K1-action-derivations}, we have
\begin{align*}
&(\partial_1\partial_2)^{2n_k}\partial_2 (\wt_{2n_1}\wt_{2n_2}\cdots \wt_{2n_k})
\\
&\qquad{}= \sum_{i=1}^k \bb_{2n_i}q_{21}^{2n_{i+1}+\cdots +2n_k} (\partial_1\partial_2)^{2n_k} (\wt_{2n_1}\cdots \wt_{2n_{i-1}} \xt^{n_i}\wt_{2n_{i+1}}\cdots \wt_{2n_k})
\\
&\qquad{} = n_k! \bb_{2n_k} \wt_{2n_1}\wt_{2n_2}\cdots \wt_{2n_{k-1}},
\\
&(\partial_1\partial_2)^{2m}\partial_2 (\wt_{2n_1}\wt_{2n_2}\cdots \wt_{2n_k}) =0, \qquad
\text{if}\quad m>n_k.
\end{align*}
Let $M$ be maximal between the $n_k$'s such that $\wt_{2n_1}\wt_{2n_2}\cdots \wt_{2n_k}$ has
coefficient $\neq 0$ in $\mathtt S$.
Then $0=(\partial_1\partial_2)^{2M}\partial_2(\mathtt S)$ is a non-trivial linear combination of degree $N-4M-1$,
a contradiction. Thus \eqref{eq:pale+Jordan-prod-wn} is linearly independent.
\end{proof}

\begin{stepiv}
Assume that $\ghost = 1$. Then \eqref{eq:pale+Jordan-def-rel-6}, \eqref{eq:pale+Jordan-def-rel-4} and \eqref{eq:pale+Jordan-def-rel-5}
hold in $\toba(V)$.
\end{stepiv}

\begin{proof}[Proof of Step 4]
By \eqref{eq:pale+Jordan-g=1-subspace}, the braiding of $Z = \big\langle x_{\tresdos\cincodos},x_{\tresdos2}\big\rangle$ is minus the flip, hence
$\toba(Z) \simeq \Lambda(Z)$ hence \eqref{eq:pale+Jordan-def-rel-4} holds.
Now $\langle x_{2}, x_{\cincodos}\rangle \simeq $ the Jordan plane, so \eqref{eq:pale+Jordan-def-rel-5} holds.
To check \eqref{eq:pale+Jordan-def-rel-6}, we use \eqref{eq:endymion-1}, \eqref{eq:endymion-1b}, \eqref{eq:endymion-2}, \eqref{eq:pale+Jordan-def-rel-1}, \eqref{eq:pale+Jordan-def-rel-4}:
\begin{align*}
&\partial_{\cincodos} \big(x_{\tresdos\cincodos}x_{\cincodos}-q_{12}\big(x_{\cincodos}+\tfrac{1}{2}x_{2}\big)x_{\tresdos\cincodos}\big) \\
&\qquad{}=x_{\tresdos\cincodos}-x_1\big(x_{\cincodos}+x_2\big)
-\big(x_{\tresdos\cincodos}+\tfrac{1}{2}x_{\tresdos2}\big)
+q_{12}\big(x_{\cincodos}+\tfrac{1}{2}x_{2}\big)x_1
= -x_{1\cincodos}-\tfrac{1}{2}x_{\tresdos2}=0,
\\
&\partial_{2} \big(x_{\tresdos\cincodos}x_{\cincodos}-q_{12}\big(x_{\cincodos}+\tfrac{1}{2}x_{2}\big)x_{\tresdos\cincodos}\big)
\\
&\qquad{}=\tfrac{1}{2}\big(\big(x_{\tresdos}+x_1\big)\big(x_{\cincodos}+x_2\big)
-x_{\tresdos\cincodos}-\tfrac{1}{2}x_{\tresdos2}
-q_{12}\big(x_{\cincodos}+\tfrac{1}{2}x_{2}\big)\big(x_{\tresdos}+x_1\big)\big)
\\
&\qquad{}
=\tfrac{1}{2}\big(x_{\tresdos\cincodos}+x_{1\cincodos}+x_{\tresdos2}
-x_{\tresdos\cincodos}-\tfrac{1}{2}x_{\tresdos2} \big)=0,
\\
&\partial_{\cincodos} \big(x_{\tresdos\cincodos}x_2-q_{12}x_2x_{\tresdos\cincodos}\big) =
-x_1x_2+q_{12}x_2x_1=0,
\\
& \partial_{2} \big(x_{\tresdos\cincodos}x_2-q_{12}x_2x_{\tresdos\cincodos}\big) =
x_{\tresdos\cincodos}
+\tfrac{1}{2}\big(x_{\tresdos}+x_1\big) x_2
-\big(x_{\tresdos\cincodos}+\tfrac{1}{2}x_{\tresdos2}\big)
-\tfrac{1}{2}q_{12}x_2(x_{\tresdos}+x_1)=0.
\end{align*}
As $\partial_1$, $\partial_{\tresdos}$ annihilate all the terms in \eqref{eq:pale+Jordan-def-rel-6}, both relations hold in $\toba(V)$.
\end{proof}

\begin{stepiv}
End of the case $\ghost = 1$.
\end{stepiv}

\begin{proof}[Proof of Step 5]
If $\toba$ is the algebra with the claimed presentation, then
there is a surjective map $\toba \twoheadrightarrow \toba\big(\sele_{1,+}\big(q, -\frac12\big)\big)$.
Now the following relations hold in $\toba$:
\begin{align*}
&x_1x_{\tresdos2} = -q_{12}x_{\tresdos2}x_1, \qquad x_{\tresdos}x_{\tresdos2} = -q_{12}x_{\tresdos2}x_{\tresdos},
\qquad x_{\tresdos2}x_{\cincodos} = q_{12}\big(x_{\cincodos}+\tfrac{1}{2}x_{2}\big)x_{\tresdos2},
\\
&x_1x_{\tresdos\cincodos} = -q_{12}\big(x_{\tresdos\cincodos}-\tfrac{1}{2}x_{\tresdos2}\big)x_1, \qquad
x_{\tresdos}x_{\tresdos\cincodos} = -q_{12}\big(x_{\tresdos\cincodos}-\tfrac{1}{2}x_{\tresdos2}\big)x_{\tresdos}.
\end{align*}
Hence $\toba$ is spanned by the monomials in \eqref{eq:selene-1+-PBW-basis}.
It only remains to prove that they are linearly independent in $\toba(\sele_{1,+}(q, -\frac12))$. By direct computations,
\begin{align*}
\begin{split}
&\partial_1 \big(x_{\cincodos}^{m_1}x_{2}^{m_2}x_{\tresdos\cincodos}^{m_3}x_{\tresdos2}^{m_4} x_{\tresdos}^{m_5} x_1^{m_6} \big)
 =\delta_{m_6,1} x_{\cincodos}^{m_1}x_{2}^{m_2}x_{\tresdos\cincodos}^{m_3}x_{\tresdos2}^{m_4} x_{\tresdos}^{m_5},
\\
&\partial_{\tresdos} \big(x_{\cincodos}^{m_1}x_{2}^{m_2}x_{\tresdos\cincodos}^{m_3}x_{\tresdos2}^{m_4} x_{\tresdos}^{m_5} \big)
 =\delta_{m_5,1} x_{\cincodos}^{m_1}x_{2}^{m_2}x_{\tresdos\cincodos}^{m_3}x_{\tresdos2}^{m_4},
\\
&\partial_1\partial_{\cincodos} \big(x_{\cincodos}^{m_1}x_{2}^{m_2}x_{\tresdos\cincodos}^{m_3}x_{\tresdos2}^{m_4}\big)
 = -(-1)^{m_4}\delta_{m_3,1} x_{\cincodos}^{m_1}x_{2}^{m_2}x_{\tresdos2}^{m_4},
\\
&\partial_1\partial_{2} \big(x_{\cincodos}^{m_1}x_{2}^{m_2}x_{\tresdos2}^{m_4}\big) = -\delta_{m_4,1} x_{\cincodos}^{m_1}x_{2}^{m_2}.
\end{split}
\end{align*}
Thus the case $\ghost=1$ follows using again a recursive argument.
\end{proof}

\begin{stepiv}
Assume that $\ghost =2$. Then \eqref{eq:pale+Jordan-def-rel-7}, \eqref{eq:pale+Jordan-def-rel-8} and
\eqref{eq:pale+Jordan-def-rel-9} hold in $\toba(V)$.
\end{stepiv}

\begin{proof}[Proof of Step 6]
We check these relations using derivations.
First we check that
\begin{align*}
\partial_{\cincodos}(\ttt)=x_{\tresdos2}, \qquad
\partial_{2}(\ttt)=x_{\tresdos\cincodos},\qquad
\partial_{\cincodos}(\wt)=0, \qquad
\partial_{2}(\wt)=x_{\tresdos2}.
\end{align*}
Using these computations, \eqref{eq:endymion-1}, \eqref{eq:endymion-1b}, \eqref{eq:endymion-2b}, \eqref{eq:pale+Jordan-def-rel-1} and \eqref{eq:pale+Jordan-def-rel-5}, we have
\begin{gather*}
\partial_{\cincodos} \big(x_{\tresdos2}x_{\cincodos}-q_{12}x_{\cincodos}x_{\tresdos2}-\wt\big) =x_{\tresdos2}-q_{12}q_{21}x_{\tresdos2}=0,
\\
\partial_{2} \big(x_{\tresdos2}x_{\cincodos}-q_{12}x_{\cincodos}x_{\tresdos2}\big) = -x_1\big(x_{\cincodos}+x_2\big)+q_{12}x_{\cincodos}x_1=-x_{1\cincodos}=\partial_2(\wt),
\\
\partial_{\cincodos} \big(\ttt x_{\cincodos}-q_{12}(x_{\cincodos}+x_{2})\ttt\big)
=\ttt+x_{\tresdos2} (x_{\cincodos}+x_2)-q_{12}(x_{\cincodos}+x_{2})x_{\tresdos2}-q_{12}q_{21}(\ttt+\wt)
\\ \hphantom{\partial_{\cincodos} \big(\ttt x_{\cincodos}-q_{12}(x_{\cincodos}+x_{2})\ttt\big)}
{} = \ttt+\wt- (\ttt+\wt)=0,
\\
\partial_{2} \big(\ttt x_{\cincodos}-q_{12}\big(x_{\cincodos}+x_{2}\big)\ttt\big) = x_{\tresdos\cincodos} \big(x_{\cincodos}+x_2\big)-q_{12}\big(x_{\cincodos}+x_{2}\big)x_{\tresdos\cincodos}
-(\ttt+\wt)=0,
\\
\partial_{\cincodos} \big(x_{\tresdos2} \ttt+q_{12}(\ttt-\wt)x_{\tresdos2}\big)= x_{\tresdos2}^2+q_{12}q_{21}x_{\tresdos2}^2=0,
\\
\partial_{2} \big(x_{\tresdos2} \ttt+q_{12}(\ttt-\wt)x_{\tresdos2}\big)
= -q_{21}x_1(\ttt-\wt)+x_{\tresdos2}x_{\tresdos\cincodos} -q_{12}(\ttt-\wt)x_1+\big(x_{\tresdos\cincodos}-x_{\tresdos2}\big)x_{\tresdos2}
\\ \hphantom{\partial_{2} \big(x_{\tresdos2} \ttt+q_{12}(\ttt-\wt)x_{\tresdos2}\big)}
{} = -q_{21}[x_1,\ttt]_c+\xt=0,
\\
\partial_{\cincodos} \big(\wt x_{\cincodos} -q_{12}\big(x_{\cincodos}+x_{2}\big)\wt\big)= \wt-q_{12}q_{21}\wt=0,
\\
\partial_{2} \big(\wt x_{\cincodos} -q_{12}\big(x_{\cincodos}+x_{2}\big)\wt\big) = x_{\tresdos2}(x_{\cincodos}+x_2)-\wt-q_{12}\big(x_{\cincodos}+x_2\big)x_{\tresdos2}=0,
\\
\partial_{\cincodos} \big(x_{\tresdos\cincodos} \ttt +q_{12}(\ttt-\wt)x_{\tresdos\cincodos}\big)=
-q_{21}x_1 (\ttt-\wt) + x_{\tresdos\cincodos}x_{\tresdos2}+x_{\tresdos2}\big(x_{\tresdos\cincodos}+x_{\tresdos2}\big)-q_{12}(\ttt-\wt)x_1
\\ \hphantom{\partial_{\cincodos} \big(x_{\tresdos\cincodos} \ttt +q_{12}(\ttt-\wt)x_{\tresdos\cincodos}\big)}
{} = -q_{21}[x_1,\ttt]_c+\xt = 0,
\\
\partial_{2} \big(x_{\tresdos\cincodos} \ttt +q_{12}(\ttt-\wt)x_{\tresdos\cincodos}\big) =
q_{21}(x_{\tresdos}+x_1) (\ttt-\wt) +\big(x_{\tresdos\cincodos}-x_{\tresdos2}\big)\big(x_{\tresdos\cincodos}+x_{\tresdos2}\big)+ x_{\tresdos\cincodos}^2
\\ \hphantom{\partial_{2} \big(x_{\tresdos\cincodos} \ttt +q_{12}(\ttt-\wt)x_{\tresdos\cincodos}\big) =}
 +q_{12}(\ttt-\wt)\big(x_{\tresdos}+x_1\big)
= \xt-2x_{\tresdos\cincodos}^2+ 2x_{\tresdos\cincodos}^2-\xt=0.
\end{gather*}
As $\partial_1$, $\partial_{\tresdos}$ annihilate all the terms in these relations, they hold in $\toba(V)$.
\end{proof}

\begin{stepiv}
End of the case $\ghost = 2$.
\end{stepiv}

\begin{proof}[Proof of Step 7]
If $\toba$ is the algebra with the claimed presentation, then
there is a surjective map $\toba \twoheadrightarrow \toba(\sele_{1,+}(q, -1))$.
Now the following relations hold in $\toba$:
\begin{alignat*}{2}
&x_1\ttt =-q_{12}^2(\ttt-\wt)x_1 +q_{12}\xt,\qquad&&
x_{\tresdos}\ttt =-q_{12}^2(\ttt-\wt)x_1 +q_{12}\xt-2x_{\tresdos\cincodos}^2,
\\
&x_\tresdos\wt =-q_{12}^2\wt x_{\tresdos}+q_{12}\xt,&&
\wt^2 = \ttt^2 =0,
\end{alignat*}
and $[\xt,\yt]_c=0$ for other PBW generators $\xt$, $\yt$;
thus $\toba$ is spanned by the monomials \eqref{eq:selene-1--PBW-basis}.
We prove linear independence in $\toba(\sele_{1,+}(q, -1))$:
\begin{gather*}
\partial_1\big( x_{2}^{m_1}x_{\cincodos}^{m_2} \wt^{m_3}\ttt^{m_4} x_{\tresdos2}^{m_5}\xt^{m_6}x_{\tresdos\cincodos}^{m_7} x_{\tresdos}^{m_8} x_1^{m_9} \big)
=\delta_{m_9,1}x_{2}^{m_1}x_{\cincodos}^{m_2} \wt^{m_3}\ttt^{m_4} x_{\tresdos2}^{m_5}\xt^{m_6}x_{\tresdos\cincodos}^{m_7} x_{\tresdos}^{m_8},
\\
\partial_{\tresdos} \big( x_{2}^{m_1}x_{\cincodos}^{m_2} \wt^{m_3}\ttt^{m_4} x_{\tresdos2}^{m_5}\xt^{m_6}x_{\tresdos\cincodos}^{m_7} x_{\tresdos}^{m_8} \big)
= \delta_{m_8,1} x_{2}^{m_1}x_{\cincodos}^{m_2} \wt^{m_3}\ttt^{m_4} x_{\tresdos2}^{m_5}\xt^{m_6}x_{\tresdos\cincodos}^{m_7},
\\
\partial_1\partial_{\cincodos} \big(x_{2}^{m_1}x_{\cincodos}^{m_2} \wt^{m_3}\ttt^{m_4} x_{\tresdos2}^{m_5}\xt^{m_6}x_{\tresdos\cincodos}^{m_7} \big)
=-m_7 x_{2}^{m_1}x_{\cincodos}^{m_2} \wt^{m_3}\ttt^{m_4} x_{\tresdos2}^{m_5}\xt^{m_6}x_{\tresdos\cincodos}^{m_7-1}
\\
\partial_1\partial_{2}\partial_1\partial_{2} \big(x_{2}^{m_1}x_{\cincodos}^{m_2} \wt^{m_3}\ttt^{m_4} x_{\tresdos2}^{m_5}\xt^{m_6} \big)
= -2m_6 x_{2}^{m_1}x_{\cincodos}^{m_2} \wt^{m_3}\ttt^{m_4} x_{\tresdos2}^{m_5}\xt^{m_6-1},
\\
\partial_1\partial_{2} \big(x_{2}^{m_1}x_{\cincodos}^{m_2} \wt^{m_3}\ttt^{m_4} x_{\tresdos2}^{m_5} \big)
= -\delta_{m_5,1} x_{2}^{m_1}x_{\cincodos}^{m_2} \wt^{m_3}\ttt^{m_4},
\\
\partial_1\partial_{2}^2 \big(x_{2}^{m_1}x_{\cincodos}^{m_2} \wt^{m_3}\ttt^{m_4}\big)
= -\delta_{m_4,1} x_{2}^{m_1}x_{\cincodos}^{m_2} \wt^{m_3},
\\
\partial_1\partial_{2}\partial_{\cincodos} \big(x_{2}^{m_1}x_{\cincodos}^{m_2} \wt^{m_3}\big)
= -\delta_{m_3,1} x_{2}^{m_1}x_{\cincodos}^{m_2}.
\end{gather*}
Thus the claim follows by a recursive argument as in previous cases.
\end{proof}
\renewcommand{\qed}{}
\end{proof}

\subsubsection{A pale block and a super Jordan plane}
As in Section~\ref{subsubsec:pale-jordan}, we assume that $q_{11} = -1$, $q_{12} = q = q_{21}^{-1}$, $b = q_{22}$ and $\ghost =a \in \N$, cf.\ \eqref{eq:discrete-ghost}. But now $q_{22}= -1$ so that $\toba(V_2)$ is a super Jordan plane.
When $\ghost= 1$, $V$ is the braided vector space $\sele_{1,-}(q)$, see \eqref{eq:selene-defi-1-}. To state our result we need the same elements
$\ttt$ and $\wt$ as in \eqref{eq:pale-Jordan-def-w-t}.

\begin{Theorem} \label{th:selene-1-}
The algebra $\toba(V)$ has finite $\GK $ if and only if $\ghost = 1$.
If $\ghost = 1$, then $\toba(\sele_{1,-}(q))$ is presented by generators
$x_1$, $x_{\tresdos}$, $x_2$, $x_\cincodos$ with defining relations
\eqref{eq:endymion-1}, \eqref{eq:endymion-1b}, \eqref{eq:endymion-2b}, \eqref{eq:pale+Jordan-def-rel-1} and
\begin{alignat}{2}
\label{eq:pale+superJordan-def-rel-6}
&x_{\tresdos2}x_{\cincodos} =-q_{12}x_{\cincodos}x_{\tresdos2}+\wt,&&
\ttt x_{\cincodos} =q_{12}\big(x_{\cincodos}-x_{2}\big)\ttt,
\\\label{eq:pale+superJordan-def-rel-4}
&x_{\tresdos}\ttt =-q_{12}^2(\ttt+2\wt)x_{\tresdos}-q_{12}x_{\tresdos2}x_{\tresdos\cincodos},\qquad&&
\wt x_{\cincodos} =q_{12}\big(x_{\cincodos}-x_{2}\big)\wt,
\\ \label{eq:pale+superJordan-def-rel-5}
&x_{\tresdos\cincodos}x_{\tresdos2} = x_{\tresdos2}x_{\tresdos\cincodos} -\tfrac{1}{2}x_{\tresdos2}^2,&&
 x_{\cincodos}x_{\cincodos2} =x_{\cincodos2}x_{\cincodos}+x_2x_{\cincodos2}.
\end{alignat}

A PBW-basis is formed by the monomials
\begin{align}\label{eq:selene-1--PBW-basis}
\begin{split}
&x_{2}^{m_1}x_{\cincodos2}^{m_2}x_{\cincodos}^{m_3}\wt^{m_4}\ttt^{m_5} x_{\tresdos2}^{m_6}x_{\tresdos\cincodos}^{m_7} x_{\tresdos}^{m_8} x_1^{m_9},
\\
&m_1, m_4, m_5, m_8,m_9 \in \{0,1\}, \qquad m_2, m_3,m_6,m_7 \in \N_0.
\end{split}
\end{align}
Hence $\GK \toba(\sele_{1,-}(q)) = 4$.
\end{Theorem}

\begin{proof} We use \eqref{eq:pale+Jordan-def-rel-1} to check that
\begin{align*}
g_1\cdot x_{\tresdos\cincodos}= -q_{12}\big(x_{\tresdos\cincodos}+ax_{\tresdos2}\big), \qquad
g_2\cdot x_{\tresdos\cincodos}= -q_{21}\big(x_{\tresdos\cincodos}+(a-1)x_{\tresdos2}\big).
\end{align*}
From these equalities and Lemma~\ref{lem:pale+Jordan-basis-K1}, we get
\begin{align*}
c\big(x_{\tresdos i} \otimes x_{\tresdos2}\big) = x_{\tresdos2} \otimes x_{\tresdos i}, \qquad
c\big(x_{\tresdos i} \otimes x_{\tresdos\cincodos}\big) = \big(x_{\tresdos\cincodos}+(2a-1)x_{\tresdos2}\big) \otimes x_{\tresdos i}, \qquad
i = 2,\tfrac{5}{2}.
\end{align*}
Then $\langle x_{\tresdos\cincodos},x_{\tresdos2}\rangle$ is a braided vector subspace of $\K^1$.

\begin{stepv}
The Nichols algebra of $\langle x_{\tresdos\cincodos},x_{\tresdos2}\rangle$ is
isomorphic to the Jordan plane. Then the set $\big\{x_{\tresdos2}^a x_{\tresdos\cincodos}^b \mid a,b \in\N_0\big\}$
is a basis of the subalgebra $\Bbbk \big\langle x_{\tresdos2},x_{\tresdos\cincodos}\big\rangle$ and
\begin{align}\label{eq:pale+superJordan-def-rel-1}
x_{\tresdos\cincodos}x_{\tresdos2}^n = x_{\tresdos2}^nx_{\tresdos\cincodos} -\tfrac{(2a-1)n}{2}x_{\tresdos2}^{n+1} \qquad \text{for all} \quad n\in\N.
\end{align}
\end{stepv}

We define $\wt_n\in\toba(V)$ recursively by
$\wt_0 := x_2$ and
\begin{align*}
\wt_{n+1} := \big[x_{\tresdos\cincodos},\wt_n\big]_c =x_{\tresdos\cincodos}\wt_n -(g_1g_2\cdot \wt_n) x_{\tresdos\cincodos},\qquad
 n\in\N.
\end{align*}
Thus $\wt = \wt_1$, see above. We also define scalars $\ab_n$, $\bb_n$ by
\begin{gather*}
\ab_n := -\frac{1}{2^{n+1} \ghost}\prod_{k=0}^n \big((2\ghost -1)k-2 \ghost \big),
\\
\bb_{n+1} :=(-1)^{n} \ghost \ab_n -\bb_n \bigg(\frac{(2\ghost-1)n}{2}+(\ghost - 1)\bigg).
\end{gather*}

\begin{stepv}\label{lem:pale+superJordan-K1-action-derivations}
We have
\begin{alignat}{2}\label{eq:pale+superJordan-wn-adj-action}
&[x_{\tresdos2},\wt_n]_c = [x_1,\wt_n]_c =0,\qquad&&
[x_{\tresdos},\wt_n]_c = q_{12}^n\ab_n x_{\tresdos2}^{n+1}.
\\
\label{eq:pale+superJordan-action-wn}
&g_1\cdot \wt_n =(-1)^n q_{12}^{n+1} \wt_n,&&
g_2\cdot \wt_n =(-1)^{n+1} q_{21}^{n} \wt_n.
\\
\label{eq:pale+superJordan-derivations-wn}
&\partial_{2}(\wt_n) =\bb_n x_{\tresdos2}^{n}, &&\partial_{i}(\wt_n) =0, \qquad i=1,\tfrac32,\tfrac52, \quad n\in\N_0.
\end{alignat}
\end{stepv}

\begin{proof}[Proof of Steps 1 and 2]
We proceed recursively on $n\in\N_0$.
For $n=0$, the first two equalities of \eqref{eq:pale+superJordan-wn-adj-action} follow since $x_2^2=0=(\adc x_1)x_2$, while the last one, \eqref{eq:pale+superJordan-action-wn} and \eqref{eq:pale+superJordan-derivations-wn} are straightforward.
Assume that \eqref{eq:pale+superJordan-wn-adj-action}, \eqref{eq:pale+superJordan-action-wn} and \eqref{eq:pale+superJordan-derivations-wn} hold for $n$. Then
\begin{gather*}
g_1 \cdot x_{n+1} =
\bigl[ -q_{12}\big(x_{\tresdos\cincodos}+ax_{\tresdos2}\big), (-1)^nq_{12}^{n+1} \wt_n \bigr]_c
= (-1)^{n+1}q_{12}^{n+2} \wt_{n+1},
\\
g_2 \cdot x_{n+1} =
(-1)^{n+2}q_{21}^{n+1} \big[ x_{\tresdos\cincodos}+(a-1)x_{\tresdos2}, \wt_n \big]_c
= (-1)^{n+2}q_{21}^{n+1} \wt_{n+1},
\\
\big[x_{\tresdos2},\wt_{n+1}\big]_c = \big[ \big[x_{\tresdos2}, x_{\tresdos\cincodos}\big]_c, \wt_n \big]_c
+\big(g_1g_2 \cdot x_{\tresdos\cincodos}\big) \big[x_{\tresdos2}, \wt_n \big]_c
+q_{12}\big[x_{\tresdos2}, \wt_n \big]_c x_{\tresdos\cincodos}
\\ \hphantom{[x_{\tresdos2},\wt_{n+1}]_c }
 {} = -\tfrac{2a-1}{2} \big[ x_{\tresdos2}^2, \wt_n \big]_c=0,
\\
[x_1,\wt_{n+1}]_c = \big[ \big[ x_1, x_{\tresdos\cincodos}\big]_c, \wt_n \big]_c
+\big(g_1 \cdot x_{\tresdos\cincodos}\big) [x_1, \wt_n ]_c
+q_{12}[x_1, \wt_n ]_c x_{\tresdos\cincodos} =0,
\\
\big[x_{\tresdos},\wt_{n+1}\big]_c = \big[ \big[ x_{\tresdos}, x_{\tresdos\cincodos}\big]_c, \wt_n \big]_c
+\big(g_1 \cdot x_{\tresdos\cincodos}\big) \big[x_{\tresdos}, \wt_n \big]_c
+q_{12}\big[x_{\tresdos}, \wt_n \big]_c x_{\tresdos\cincodos}
\\ \hphantom{[x_{\tresdos},\wt_{n+1}]_c}
{} =
-q_{12}^{n+1}\ab_n \big(x_{\tresdos\cincodos}+ax_{\tresdos2}\big)x_{\tresdos2}^{n+1}
+q_{12}^{n+1}\ab_n x_{\tresdos2}^{n+1}x_{\tresdos\cincodos}
\\ \hphantom{[x_{\tresdos},\wt_{n+1}]_c}
{} = q_{12}^{n+1}\ab_n \tfrac{(2a-1)(n+1)-2a}{2}x_{\tresdos2}^{n+2} = \ab_{n+1}q_{12}^{n+1} x_{\tresdos2}^{n+2},
\end{gather*}
by \eqref{eq:graded-bracket-jacobi}, Lemma~\ref{lem:pale+Jordan-basis-K1} and the inductive hypothesis.
We conclude that
\begin{align*}
&\wt_{n+1}=x_{\tresdos\cincodos}\wt_n+q_{12}\wt_nx_{\tresdos\cincodos},\qquad \text{so}\quad
\partial_1 (\wt_{n+1})=\partial_{\tresdos}(\wt_{n+1})=0,
\\
&\partial_{\cincodos} (\wt_{n+1})
= (-1)^{n} q_{21}^{n}x_1 \wt_n -q_{12}\wt_n x_1
= (-1)^{n} q_{21}^{n} [x_1, \wt_n]_c=0.
\end{align*}
Finally, we compute the remaining skew-derivation:
\begin{align*}
\partial_{2} (\wt_{n+1})
&= (-1)^{n} q_{21}^{n}a(x_{\tresdos}+x_1) \wt_n+\bb_n x_{\tresdos\cincodos}x_{\tresdos2}^{n}
-\bb_n x_{\tresdos2}^{n} \big(x_{\tresdos\cincodos}+(a-1)x_{\tresdos2}\big)
\\ &\hphantom{\partial_{2} (\wt_{n+1})}
-aq_{12}\wt_n (x_{\tresdos}+x_1)
\\
&= (-1)^{n} q_{21}^{n}a \big(\big[x_{\tresdos},\wt_n\big]_c + [x_{1},\wt_n]_c\big)
-\bb_n \bigg(\frac{(2a-1)n}{2}-(a-1)\bigg)x_{\tresdos2}^{n+1}
\\
&= \bigg((-1)^{n} a \ab_n -\bb_n \bigg(\frac{(2a-1)n}{2}+(a-1)\bigg)\bigg) x_{\tresdos2}^{n+1}
= \bb_{n+1} x_{\tresdos2}^{n+1}. \tag*{\qed}
\end{align*}
\renewcommand{\qed}{}
\end{proof}

\begin{stepv}
If $\ghost\in\N_{\ge 2}$, then $\GK \toba(V)=\infty$.
\end{stepv}

\begin{proof}[Proof of Step 3]
First we claim that $\wt_n \ne 0$, $\bb_n \ne 0$, for all $n\in\N_0$.

If $\ghost=a\ge 2$, then $\tfrac{2\ghost}{2\ghost-1}=1+\tfrac{1}{2\ghost-1}\notin\Z$, so $\ab_n\ne 0$ for $n\in\N$.
By \eqref{eq:pale+superJordan-wn-adj-action}, we have $[x_{\tresdos},\wt_n]_c\ne 0$, so $\wt_n \ne 0$.
By \eqref{eq:pale+superJordan-derivations-wn}, $0\ne \partial_2(\wt_n)$, so $\bb_{n}\ne 0$.

By \cite[Lemma~2.3.4]{AAH-triang}, to prove the step it is enough to show that the set
\begin{gather}\label{eq:pale+superJordan-prod-wn}
 \wt_{n_1}\wt_{n_2}\cdots \wt_{n_k}, \qquad
 k\in\mathbb N_0, \quad n_1< \cdots< n_k \in\N
\end{gather}
is linearly independent.
Otherwise pick a non-trivial linear combination $\mathtt S$ of elements in \eqref{eq:pale+superJordan-prod-wn},
homogeneous of minimal degree $N$. By Step~\ref{lem:pale+superJordan-K1-action-derivations}, we have
\begin{gather*}
(\partial_1\partial_2)^{n_k}\partial_2 (\wt_{n_1}\wt_{n_2}\cdots \wt_{n_k})
\\ \qquad
{}= \sum_{i=1}^k \bb_{n_i}(-1)^{n_{i+1}+\cdots+n_k+k-i} q_{21}^{n_{i+1}+\cdots+n_k}
(\partial_1\partial_2)^{n_k} \big(\wt_{n_1}\cdots \wt_{n_{i-1}} x_{\tresdos2}^{n_i+1}\wt_{n_{i+1}}\cdots \wt_{n_k}\big)
\\ \qquad
{} = (-1)^{n_k}n_k! \bb_{n_k} \wt_{n_1}\wt_{n_2}\cdots \wt_{n_{k-1}},
\\[1mm]
(\partial_1\partial_2)^{m}\partial_2 (\wt_{n_1}\wt_{n_2}\cdots \wt_{n_k}) =0, \qquad \text{if}\quad m>n_k.
\end{gather*}
Let $M$ be maximal between the $n_k$'s such that $\wt_{n_1}\wt_{n_2}\cdots \wt_{n_k}$ has
coefficient $\neq 0$ in $\mathtt S$.
Then $0=(\partial_1\partial_2)^{M}\partial_2(\mathtt S)$ is a non-trivial linear combination of degree $N-2M-1$,
a contradiction. Thus \eqref{eq:pale+superJordan-prod-wn} is linearly independent.
\end{proof}

\begin{stepv}
Assume that $\ghost = 1$. Then \eqref{eq:pale+superJordan-def-rel-6}, \eqref{eq:pale+superJordan-def-rel-4} and
\eqref{eq:pale+superJordan-def-rel-5} hold in $\toba(V)$.
\end{stepv}

The first relation in \eqref{eq:pale+superJordan-def-rel-5} is \eqref{eq:pale+superJordan-def-rel-1} for $n=a=1$ while the second holds since
$\big\langle x_{2}, x_{\cincodos}\big\rangle \simeq $ the Jordan super plane.
Next we check \eqref{eq:pale+superJordan-def-rel-6} and \eqref{eq:pale+superJordan-def-rel-4}.
First we use \eqref{eq:pale+Jordan-def-rel-1} and that $\wt$ is $\wt_1$ in Step~\ref{lem:pale+superJordan-K1-action-derivations} to get
\begin{gather*}
\partial_{\cincodos}(\ttt)=x_{\tresdos2},\qquad
\partial_{2}(\ttt)=x_{\tresdos\cincodos},\qquad
\partial_{\cincodos}(\wt)=0,\qquad
\partial_{2}(\wt)=x_{\tresdos2}.
\end{gather*}
Using these computations, \eqref{eq:endymion-1}, \eqref{eq:endymion-1b}, \eqref{eq:endymion-2b}, \eqref{eq:pale+Jordan-def-rel-1} and \eqref{eq:pale+superJordan-def-rel-5}, we have
\begin{gather*}
\partial_{\cincodos} \big(x_{\tresdos2}x_{\cincodos} +q_{12}x_{\cincodos}x_{\tresdos2}\big) = 0 = \partial_{\cincodos}(\wt),
\\
\partial_{2} \big(x_{\tresdos2}x_{\cincodos} +q_{12}x_{\cincodos}x_{\tresdos2}\big) = -x_{\tresdos}\bigl(-x_{\cincodos}+x_2\bigr) -q_{12}x_{\cincodos}x_{\tresdos}=x_{\tresdos2}=\partial_{2}(\wt);
\\
\partial_{\cincodos} \big(\ttt x_{\cincodos} -q_{12}\big(x_{\cincodos}-x_{2}\big)\ttt\big) =
\partial_{2} \big(\ttt x_{\cincodos} -q_{12}\big(x_{\cincodos}-x_{2}\big)\ttt\big) = 0;
\\
\partial_{\cincodos} \big(\wt x_{\cincodos}-q_{12}\big(x_{\cincodos}-x_{2}\big)\wt\big) =
\partial_{2} \big(\wt x_{\cincodos}-q_{12}\big(x_{\cincodos}-x_{2}\big)\wt\big) =0;
\\
\partial_{\cincodos} \big(x_{\tresdos}\ttt +q_{12}^2(\ttt+2\wt)x_{\tresdos}\big) = x_{\tresdos}x_{\tresdos2} +q_{12}x_{\tresdos2}\big(x_{\tresdos}+x_1\big)
= q_{12}x_{\tresdos2}x_1,
\\
\partial_{\cincodos} \big(x_{\tresdos2}x_{\tresdos\cincodos}\big) = -x_{\tresdos2}x_1,
\\
\partial_{2} \big(x_{\tresdos}\ttt +q_{12}^2(\ttt+2\wt)x_{\tresdos}\big) = x_{\tresdos}x_{\tresdos\cincodos} +q_{12}\big(x_{\tresdos\cincodos}+2x_{\tresdos2}\big)(x_{\tresdos}+x_1)
\\ \hphantom{\partial_{2} \big(x_{\tresdos}\ttt +q_{12}^2(\ttt+2\wt)x_{\tresdos}\big)}
{}= q_{12}x_{\tresdos2}x_{\tresdos}+q_{12}x_{\tresdos\cincodos}x_1+2q_{12}x_{\tresdos2}x_1,
\\
\partial_{2} \big(x_{\tresdos2}x_{\tresdos\cincodos}\big) = q_{21}x_1x_{\tresdos\cincodos}-x_{\tresdos2}\big(x_{\tresdos}+x_1\big)
=-x_{\tresdos\cincodos}x_1 -x_{\tresdos2}x_{\tresdos}-2x_{\tresdos2}x_1.
\end{gather*}
As $\partial_1$, $\partial_{\tresdos}$ annihilate all the terms in \eqref{eq:pale+superJordan-def-rel-6} and \eqref{eq:pale+superJordan-def-rel-4}, they hold in $\toba(V)$.

Let $\toba$ be the algebra with the claimed presentation. Then
there is a surjective map $\toba \twoheadrightarrow \toba(\sele_{1,-}(q))$.
Also the following relations hold in $\toba$:{\samepage
\begin{alignat*}{2}
&x_1\ttt =-q_{12}^2(\ttt+2\wt)x_1 -\tfrac{1}{2}x_{\tresdos2}^2,
&&x_{\tresdos}x_{\cincodos2} = q_{12}^2x_{\cincodos2}x_{\tresdos}+2\wt,
\\
&x_\tresdos\wt =-q_{12}^2\wt x_{\tresdos} -\tfrac{1}{2}x_{\tresdos2}^2,
&&x_{\tresdos2}x_{\cincodos2} = q_{12}^2x_{\cincodos2}x_{\tresdos2}+2q_{12}x_2\wt,
\\
&\ttt x_{2} = q_{12}x_2\ttt+3q_{12}x_2\wt+x_{\cincodos2}x_{\tresdos2},\qquad
&&\wt^2 = \ttt^2 =0,
\end{alignat*}}
and $[\xt,\yt]_c=0$ for other pairs of PBW generators $\xt$, $\yt$.

Hence $\toba$ is spanned by the monomials in \eqref{eq:selene-1--PBW-basis}.
It only remains to prove that they are linearly independent in $\toba(\sele_{1,-}(q))$. By direct computations,
\begin{gather*}
\partial_1\partial_{\tresdos} \big(x_{2}^{m_1}x_{\cincodos2}^{m_2}x_{\cincodos}^{m_3}\wt^{m_4}\ttt^{m_5} x_{\tresdos2}^{m_6}x_{\tresdos\cincodos}^{m_7} x_{\tresdos}^{m_8} x_1^{m_9} \big)
=\delta_{m_9,1}\delta_{m_8,1} x_{2}^{m_1}x_{\cincodos2}^{m_2}x_{\cincodos}^{m_3}\wt^{m_4}\ttt^{m_5} x_{\tresdos2}^{m_6}x_{\tresdos\cincodos}^{m_7},
\\
\partial_1\partial_{\cincodos} \big(x_{2}^{m_1}x_{\cincodos2}^{m_2}x_{\cincodos}^{m_3}\wt^{m_4}\ttt^{m_5} x_{\tresdos2}^{m_6}x_{\tresdos\cincodos}^{m_7} \big)
=-m_7 x_{2}^{m_1}x_{\cincodos2}^{m_2}x_{\cincodos}^{m_3}\wt^{m_4}\ttt^{m_5} x_{\tresdos2}^{m_6}x_{\tresdos\cincodos}^{m_7-1},
\\
\partial_1\partial_{2} \big(x_{2}^{m_1}x_{\cincodos2}^{m_2}x_{\cincodos}^{m_3}\wt^{m_4}\ttt^{m_5} x_{\tresdos2}^{m_6} \big)
= -m_6 x_{2}^{m_1}x_{\cincodos2}^{m_2}x_{\cincodos}^{m_3}\wt^{m_4}\ttt^{m_5} x_{\tresdos2}^{m_6-1},
\\
\partial_1\partial_{2}\partial_{\cincodos} \big(x_{2}^{m_1}x_{\cincodos2}^{m_2}x_{\cincodos}^{m_3}\wt^{m_4}\ttt^{m_5}\big)
= -\delta_{m_5,1} x_{2}^{m_1}x_{\cincodos2}^{m_2}x_{\cincodos}^{m_3}\wt^{m_4},
\\
\partial_1\partial_{2}\partial_{\cincodos} \big(x_{2}^{m_1}x_{\cincodos2}^{m_2}x_{\cincodos}^{m_3}\wt^{m_4}\big)
= -\delta_{m_4,1} x_{2}^{m_1}x_{\cincodos2}^{m_2}x_{\cincodos}^{m_3}.
\end{gather*}
Thus the proof follows using a recursive argument as in previous cases.
\end{proof}

\subsection*{Acknowledgements}

This material is based upon work supported by the National Science Foundation under Grant No.~DMS-1440140 while N.A.\
was in residence at the Mathematical Sciences Research Institute in Berkeley, California, in the Spring 2020 semester.
The work of N.A.\ and I.A.\ was partially supported by CONICET (PIP 11220200102916CO),
FONCyT-ANPCyT (PICT-2019-03660) and Secyt (UNC).
The work of M.M.\ was carried out while he was in residency at the I.H.E.S (2018), University Paris-Est Marne-la-Vall\'ee (2018/2019) and University of Lille (2019/2020).
We thank the referees for a careful reading of this article and many useful remarks.

\pdfbookmark[1]{References}{ref}
\LastPageEnding

\end{document}